\documentclass{article}
\usepackage{graphicx} %Required for inserting images
\usepackage{fancyhdr}
\usepackage{extramarks}
\usepackage{amsmath}
\usepackage{amsthm}
\usepackage{amssymb}
\usepackage{amsfonts}
\usepackage{tikz}
\usepackage[plain]{algorithm}
\usepackage{algpseudocode}
\usepackage{relsize}
\usepackage[shortlabels]{enumitem}
\usepackage{url}
\usepackage{hyperref}
\usepackage{appendix}
\usepackage{listings}
\usepackage{float}

\usepackage{supertabular}
\usepackage{colortbl}
\usepackage{caption}

\usepackage{tikzscale}
\usetikzlibrary{positioning}

\usepackage{amsrefs}

\usepackage{multicol}
\usepackage{rotating}
\usepackage{pgfplotstable}
\usepackage{pgfplots}

\DeclareMathOperator{\tb}{tb}
\DeclareMathOperator{\rot}{rot}

\newcommand{\R}{\mathbb{R}}

\newcommand{\Z}{\mathbb{Z}}

\newtheorem{prop}{Proposition}
\newtheorem{cor}{Corollary}
\newtheorem{theorem}{Theorem}
\newtheorem{conjecture}{Conjecture}
\newtheorem{lemma}{Lemma}
\newtheorem{obs}{Observation}
\newtheorem{question}{Question}

\usepackage{graphicx}
\graphicspath{{Paper/}{Images/}}
\title{Bounds on Mosaic Number of Legendrian Knots}
\author{Margaret Kipe, Samantha Pezzimenti, Leif Schaumann,\\ Luc Ta, Wing Hong Tony Wong}
\begin{document}

\maketitle

\begin{abstract}
    Mosaic tiles were first introduced by Lomonaco and Kauffman in 2008 to describe quantum knots, and have since been studied for their own right. Using a modified set of tiles, front projections of Legendrian knots can be built from mosaics as well. In this work, we compute lower bounds on the mosaic number of Legendrian knots in terms of their classical invariants. We also provide a class of examples that imply sharpness of these bounds in certain cases. An additional construction of Legendrian unknots provides an upper bound on the mosaic number of Legendrian unknots. We also adapt a result of Oh, Hong, Lee, and Lee to give an algorithm to compute the number of Legendrian link mosaics of any given size. Finally, we use a computer search to provide an updated census of known mosaic numbers for Legendrian knots, including all Legendrian knots whose mosaic number is 6 or less.
\end{abstract}

\section{Introduction}

\subsection{Legendrian Knots} \label{sec:basics}
Legendrian knots are an important object of study in the field of contact topology. While a (smooth) \textbf{knot} is any embedding of a circle in $\mathbb{R}^3$, a Legendrian knot has an extra geometric condition imposed by a contact structure. 

In this paper, we will consider knots in the \textbf{standard contact structure on $\mathbb{R}^3$}, pictured in Figure \ref{fig:ContactStructure}. It is the plane field spanned by the vectors $\partial_y$ and $\partial_x+y\partial_z$ at each point $(x,y,z)$ in $\mathbb{R}^3$, arising from the kernel of a differential one-form.
Notice that the planes are horizontal along the $x$-axis and twist in the $y$-direction. They approach verticality, but never reach it. The intersection of these planes with the plane $y=y_0$ has a negative slope if $y_0>0$ and has a positive slope if $y_0<0$. The planes translate in the $x$ and $z$ direction, so all change occurs in the $y$-direction. One important feature of the standard contact structure is that no surface can be everywhere tangent to these planes. However, curves satisfying this property do exist and are called \textit{Legendrian}. 
In particular, a \textbf{Legendrian knot (or link)} is a knot (or link) whose tangent space lies within the planes of the contact structure. For an intuitive introduction to Legendrian knots, we refer the reader to \cite{sabloff}. For a more formal treatment, we refer the reader to \cite{etnyre}.

\begin{figure}[h]
    \centering
    \includegraphics[scale=0.3]{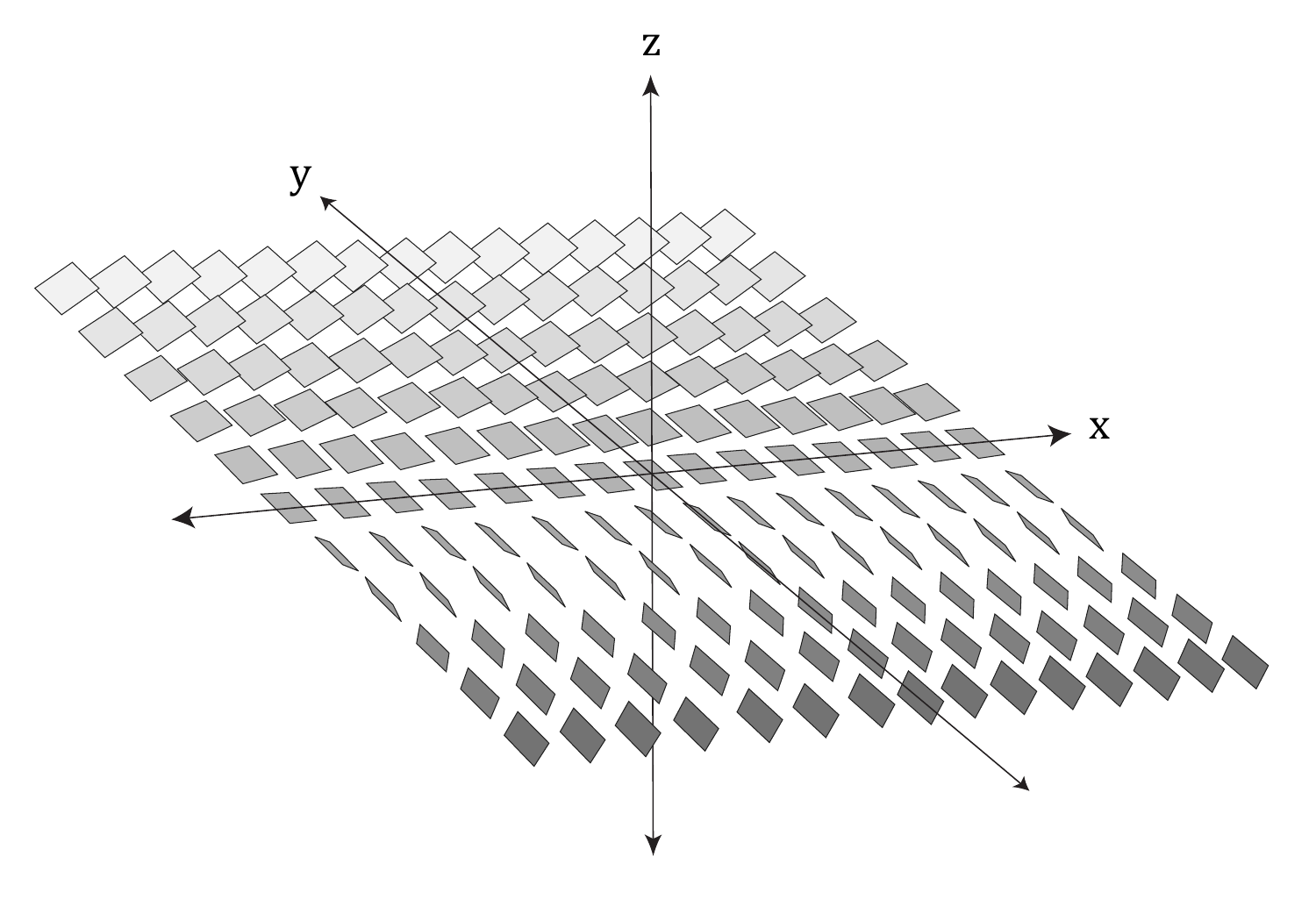}
    \caption{The standard contact structure on $\mathbb{R}^3$}
    \label{fig:ContactStructure}
\end{figure}

Legendrian knots are commonly displayed in either their $xy$ (Lagrangian) or $xz$ (front) projections. In this paper, we will use the latter. 
Front projections of Legendrian knots have two key features that distinguish them from smooth knots: (1) since tangent lines can never be vertical, they have cusps in place of vertical tangencies; and (2) the direction of twisting ensures the strand with the more negative slope is the overstrand. Figure \ref{fig:SmoothandLEg} shows a smooth positive trefoil and a Legendrian positive trefoil in its front projection. 
In this paper we will commonly refer to a  ``front projection of a Legendrian knot" as just a ``Legendrian knot."

\begin{figure}
    \centering
    \includegraphics[scale=0.3]{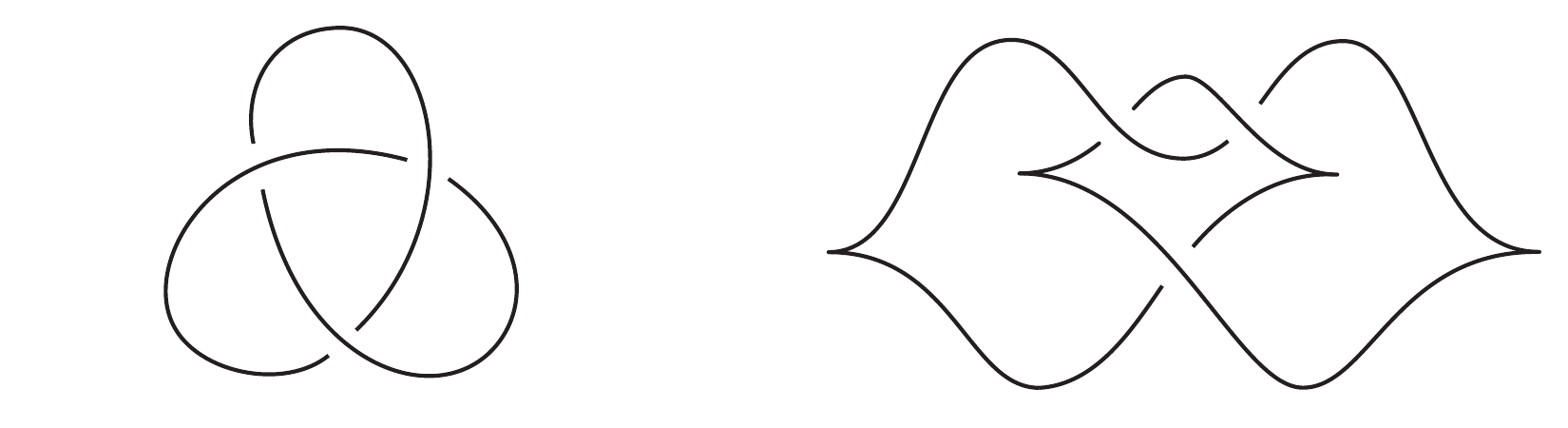}
    \caption{Smooth positive trefoil (left) and Legendrian positive trefoil (right).}
    \label{fig:SmoothandLEg}
\end{figure}

For a given smooth knot type, there are infinitely many Legendrian knots of that type. One way to distinguish these are through their \textit{classical invariants}---namely, the Thurston-Bennequin ($\tb$) and rotation  ($\rot$) numbers. These invariants can be computed from an oriented front projection diagram. The \textbf{Thurston-Bennequin number} of a Legendrian knot $\Lambda$ is computed by
$$\tb(\Lambda) = P - N - \frac{1}{2} C,$$
where $P$ is the number of positive crossings, $N$ is the number of negative crossings, and $C$ is the number of cusps in the diagram.
(Note that the quantity $P-N$ is often called the \textbf{writhe} $w(\Lambda)$ of $\Lambda$, which is not an invariant.)
The \textbf{rotation number} of $\Lambda$ is 
$$\rot(\Lambda)=\frac{1}{2}(D-U),$$
where $D$ is the number of downward pointing cusps and $U$ is the number of upward pointing cusps.
For certain knot classes, including unknots \cite{eliashberg}, torus knots \cite{torus}, and certain twist knots \cite{twist}, Legendrian knots are entirely classified by these classical invariants. It is common to organize this information into a ``mountain range" like the ones in Figure \ref{fig:mtn}. A mountain range is associated to a given smooth knot. Each dot represents a Legendrian knot with the specified classical invariants.

\begin{figure}[h]
    \centering
    \includegraphics[scale=0.5]{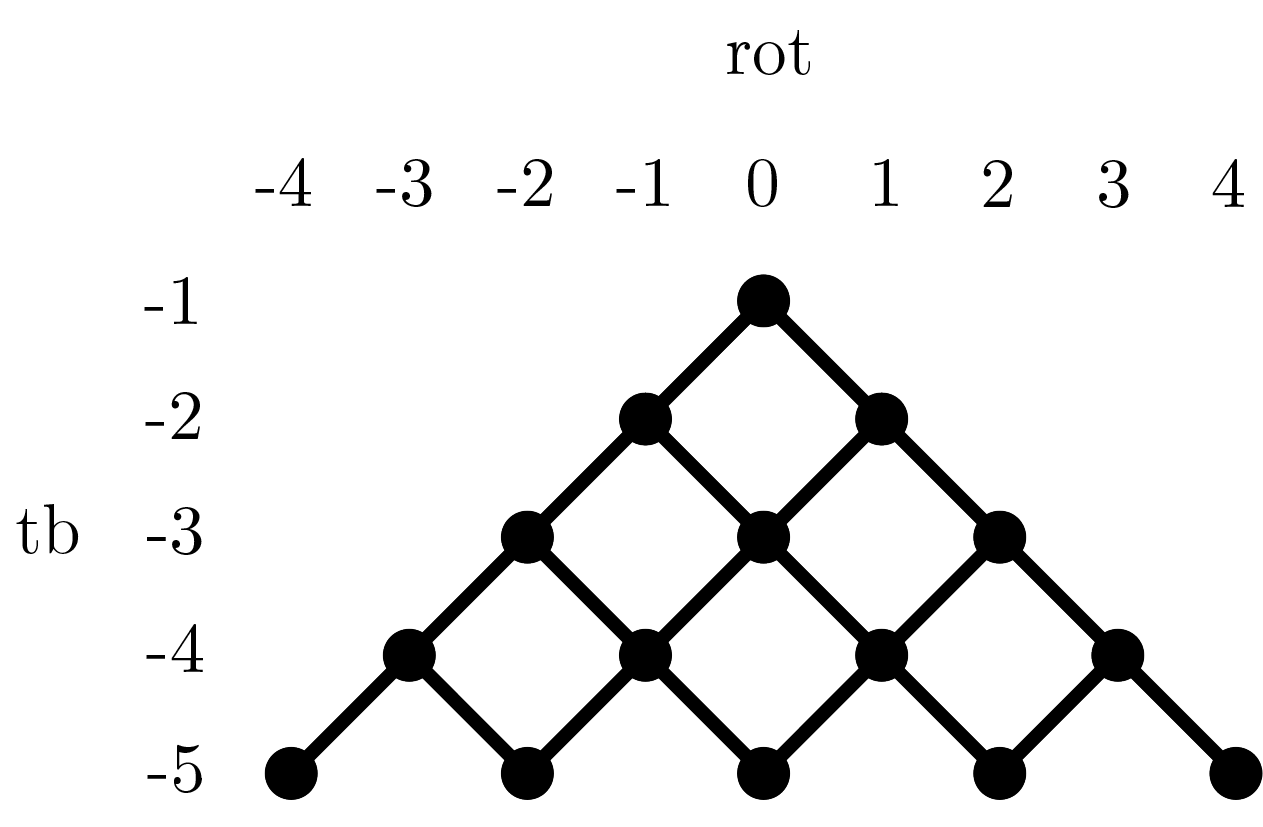} \includegraphics[scale=0.5]{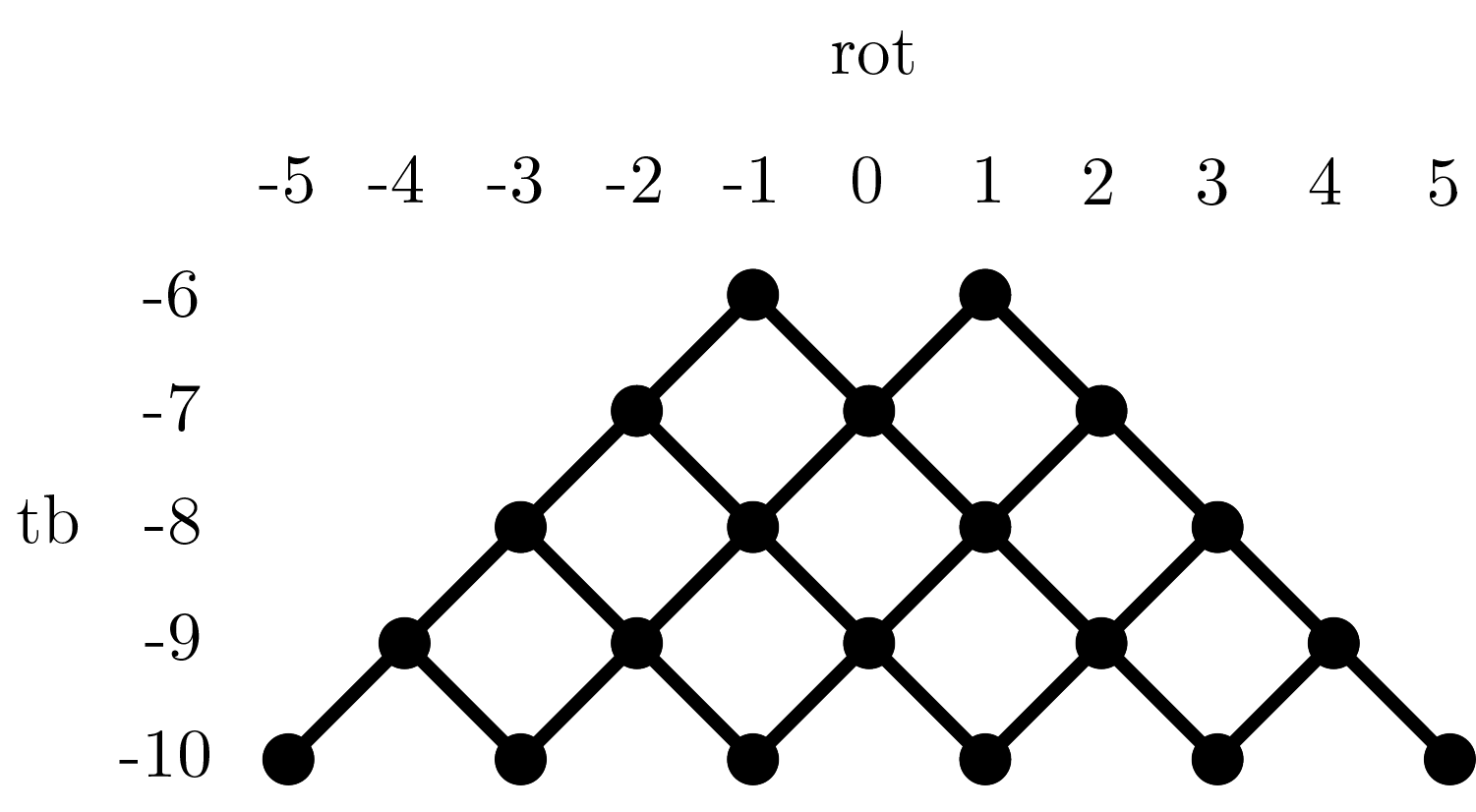}
    \caption{Mountain ranges of the Legendrian unknot (left) and Legendrian negative trefoil (right). 
    }
    \label{fig:mtn}
\end{figure}

Each Legendrian knot represented in a mountain range can be obtained from an element higher up the mountain through a process called \textit{stabilization}, pictured in Figure \ref{fig:Stabilization}. From the front diagram, this can be viewed as adding a pair of cusps. A \textbf{positive} (respectively, \textbf{negative}) \textbf{stabilization} replaces an oriented strand with a pair of downward-oriented (respectively, upward-oriented) cusps.  A Legendrian knot which can be obtained from another via stabilization is called \textit{stabilized}. Note that a positive stabilization increases the rotation number by 1, while a negative stabilization decreases the rotation number by 1. Either stabilization decreases the Thurston-Bennequin invariant by 1. Also note that inserting a ``twist" on an upward (respectively, downward) oriented cusp is equivalent to a positive (respectively, negative) stabilization.

\begin{figure}[h]
    \centering
    \includegraphics[scale=0.5]{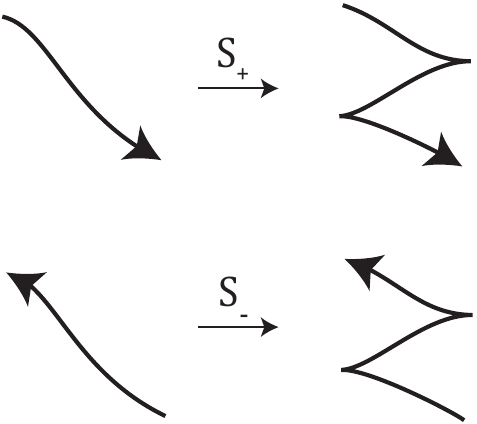}
    \caption{Positive and negative stabilization moves.}
    \label{fig:Stabilization}
\end{figure}

\subsection{Knot Mosaics}
In 2008, Lomonaco and Kauffman \cite{lomonaco} introduced a method of building knots using the mosaic tiles in Figure \ref{fig:OGtiles} as a means of describing quantum knots. An \textbf{$n$-mosaic} is an $n\times n$ array of tiles $T_0,\ldots,T_{10}$, or equivalently, an $n\times n$ matrix whose entries are integers $0,\ldots, 10$. An $n$-mosaic is \textbf{suitably connected} if the connection points of each tile coincide with connection points of contiguous tiles. An $n$-mosaic depicts a link if and only if it is suitably connected.

\begin{figure}[h]
    \centering
    \includegraphics[scale=0.2]{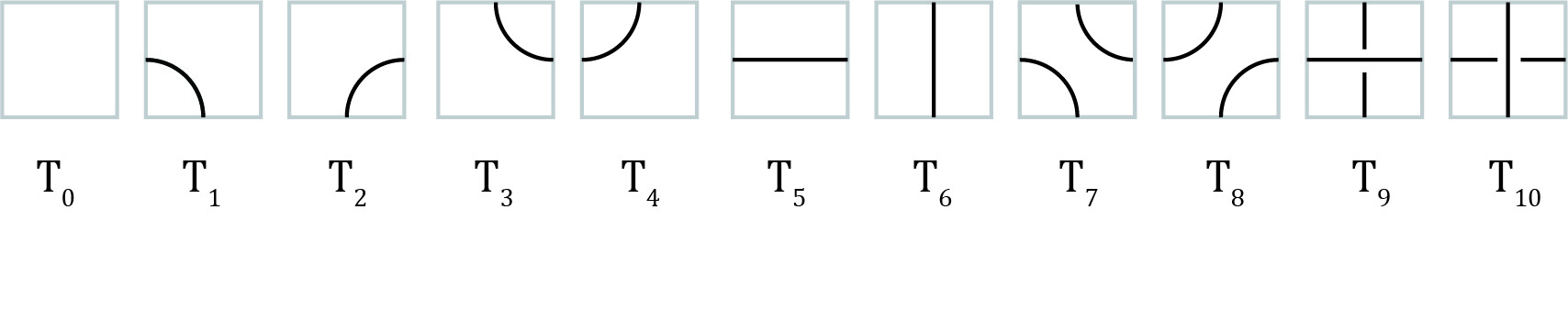}
    \vspace{-30pt}
    \caption{Original mosaic tiles to construct smooth knot mosaics.}
    \label{fig:OGtiles}
\end{figure}

Figure \ref{fig:SmoothMosaics} shows a trefoil knot on a $4$- and $5$-mosaic. Although both  are a valid mosaic representations of the trefoil, the $4$-mosaic (right) appears to utilize the grid space more efficiently. To that end, the \textbf{mosaic number of a knot $K$} is the smallest integer $n$ such that $K$ can be represented on an $n$-mosaic. Indeed, it can be shown that the mosaic number of the trefoil knot is $4$. The \textbf{inner board} of an $n$-mosaic is the $(n-2)$-submosaic consisting of all tiles except the boundary tiles. Since crossing tiles must be placed on the inner board (otherwise the boundary would contain a ``loose end"), a $3$-mosaic of a knot can have at most one crossing.

\begin{figure}[h]
    \centering
    \includegraphics[scale=0.6]{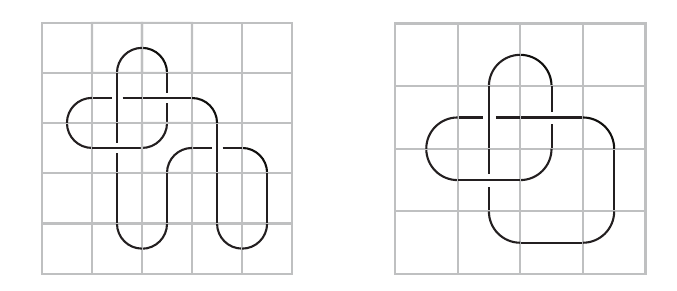}
    \caption{A trefoil knot on a $5$-mosaic (left) and a $4$-mosaic (right).}
    \label{fig:SmoothMosaics}
\end{figure}

\subsection{Legendrian Knot Mosaics}\label{sec:initial_questions}

The remainder of this paper will focus on the intersection of knot mosaics and Legendrian knots. 
In 2022, Pezzimenti and Pandey \cite{pezzimenti} introduced the modified set of mosaic tiles in Figure \ref{fig:Legtiles} to construct front projections of Legendrian knots. Notice that the tile $T_9$ is eliminated since the strand with the more negative slope must be the overstrand. A \textbf{Legendrian $n$-mosaic} is an array of the tiles in Figure \ref{fig:Legtiles} arranged in an $n \times n$ grid that has been rotated 
$45^{\circ}$ counterclockwise. We refer to the $(i,j)$th position of the Legendrian mosaic as the $(i,j)$th position of the original grid.
%corresponding matrix rotated $45^{\circ}$ clockwise.

\begin{figure}[h]
    \centering
    \includegraphics[scale=0.7]{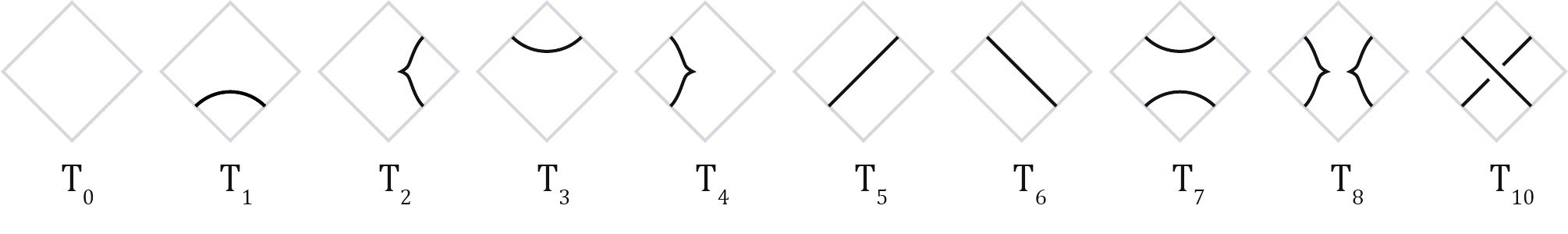}
    \caption{Legendrian mosaic tiles.}
    \label{fig:Legtiles}
\end{figure}

Figure \ref{fig:LegTrefMosaic} displays a Legendrian positive trefoil on a Legendrian $5$-mosaic. 
Although the mosaic number $m(3_1)$ of the smooth trefoil is $4$, it can be shown that if $\Lambda$ represents a Legendrian positive trefoil with maximum Thurston-Bennequin invariant, its mosaic number $m(\Lambda)$  is indeed $5$. Although the inner board contains 4 tiles, the restriction on crossing tiles means that there cannot exist a set of three alternating crossings. 

Given a smooth knot type $K$, we also define the \textbf{Legendrian mosaic number} $m_L(K)$ of $K$ to be the minimum mosaic number taken over all Legendrian representatives of $K$. For example, taking $K$ to be the trefoil $3_1$, there exist positive and negative Legendrian trefoils that can be realized on Legendrian 5-mosaics, but no Legendrian trefoil can be realized on a Legendrian 4-mosaic. Thus, $m_L(3_1)=5$. Note that the Legendrian mosaic number $m_L(K)$ is an invariant of \emph{smooth} knots. On the other hand, we will  consider the mosaic number $m(\Lambda)$  an invariant of \emph{Legendrian} knots.

\begin{figure}[h]
    \centering
    \includegraphics[scale=0.4]{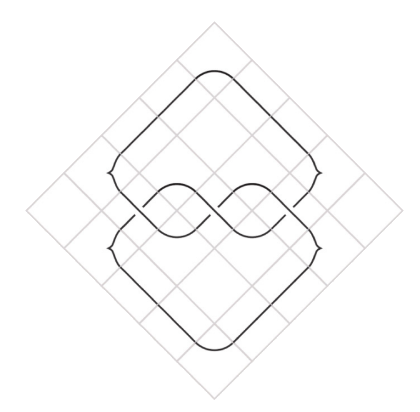}
    \caption{Legendrian positive trefoil on a $5$-mosaic.}
    \label{fig:LegTrefMosaic}
\end{figure}

In \cite{pezzimenti}, Pezzimenti and Pandey pose some open questions. One such question seeks bounds on the mosaic number of Legendrian knots via the classical invariants. In Section \ref{bounds}, we provide some bounds using both combinatorial and linear algebraic approaches. We also provide an infinite family of examples that prove sharpness of one of these bounds.

In Section \ref{unknot-section}, we provide an algorithm to construct mosaics of Legendrian unknots, which gives an upper bound on their mosaic number. This construction builds on the sequence in \cite{pezzimenti} of an infinite family of unknots that realize their mosaic numbers in non-reduced projections (i.e., projections with more than the minimum number of crossings). 

In Section \ref{sec:counting}, we adapt the proof of Theorem 1 in \cite{mosaic-count} to give an algorithm to compute the number of $m\times n$ Legendrian link mosaics for all $m,n\in\Z^+$. In Appendix \ref{app:counting}, we give an implementation of this algorithm in Mathematica and our calculations for all $1\leq n\leq m\leq 11$.

In Section \ref{algos}, we describe a computer algorithm to produce all suitably connected Legendrian mosaics, and use it to compute the mosaic number of all Legendrian knots with mosaic number up to $6$ via an exhaustive search. In Appendix \ref{census-list}, we provide an updated census reflecting these results, which now extends to knots with up to $8$ crossings. (The census in \cite{pezzimenti} previously included Legendrian unknots, and trefoils, up to size $5$.) We highlight several instances of stabilized knots having a smaller mosaic number than their pre-stabilized counterparts, including a family of smooth knots whose Legendrian mosaic number is not realized by any Legendrian representatives with maximal Thurston-Bennequin number. This answers two more open questions from \cite{pezzimenti}.

Finally, in Section \ref{questions}, we suggest some further areas to explore related to our findings.

\section{Lower Bounds from Classical Invariants}
\label{bounds}

\subsection{Oriented Legendrian Mosaic Tiles}
\label{bounds-prelims}
In Appendix B of \cite{lomonaco}, Lomonaco and Kauffman consider all 29 possible combinations of orientations for strands on the 11 classical mosaic tiles in Figure \ref{fig:OGtiles}. In this section, we similarly consider all 25 combinations of orientations for strands on the ten Legendrian mosaic tiles in Figure \ref{fig:Legtiles}. These \textbf{oriented Legendrian mosaic tiles}, hereafter referred to as ``oriented tiles," are tabulated in Figure \ref{fig:oriented_tiles}. In Sections \ref{combo_bounds} and \ref{linear}, we will use these tiles to provide two lower bounds on the mosaic number of a Legendrian knot in terms of its classical invariants, applying both combinatorial and linear algebraic arguments.  

\begin{figure}[h]
    \centering
    \includegraphics[width=1\linewidth]{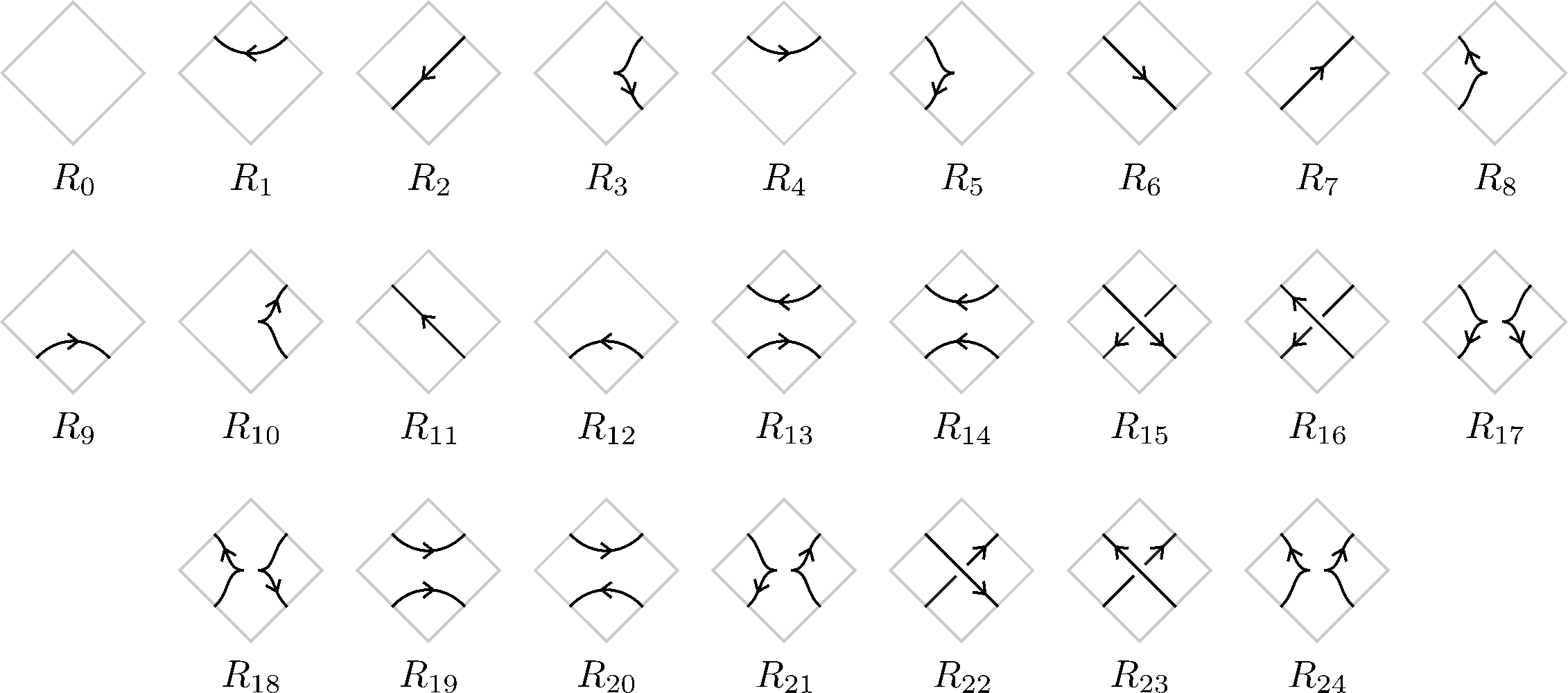}
    \caption{The 25 distinct oriented Legendrian knot mosaic tiles. (Note that the labeling is arbitrary).}
    \label{fig:oriented_tiles}
\end{figure}

We will begin by defining some useful functions based on these oriented tiles.
Given an oriented Legendrian knot mosaic $M$, define $|M|_{R_i}$ as the number of times that the oriented tile $R_i$ appears in $M$. Similarly, define $|M|_{T_i}$ as the number of times that any oriented (or unoriented) version of $T_i$ appears in $M$.

For any oriented tile $R_i$, define $h(R_i)\in\{-2,-1,0,1,2\}$ as the net horizontal movement of strands within $R_i$, 
measured in units of length $\sqrt{2}/2$ times the side length of a tile. Similarly, define $v(R_i)\in\{-2,-1,0,1,2\}$ as the net vertical movement of strands within $R_i$. For example, Figure \ref{fig:hv_examples} shows the values of $h$ and $v$ for tiles $R_3$, $R_6$, and $R_{22}$.
These functions are useful in finding restrictions of the possible valid combinations of tiles that can appear in a Legendrian knot mosaic, as shown in Proposition \ref{prop:net_movement} below.

\begin{figure}[h]
    \centering
    \includegraphics[width=0.6\linewidth]{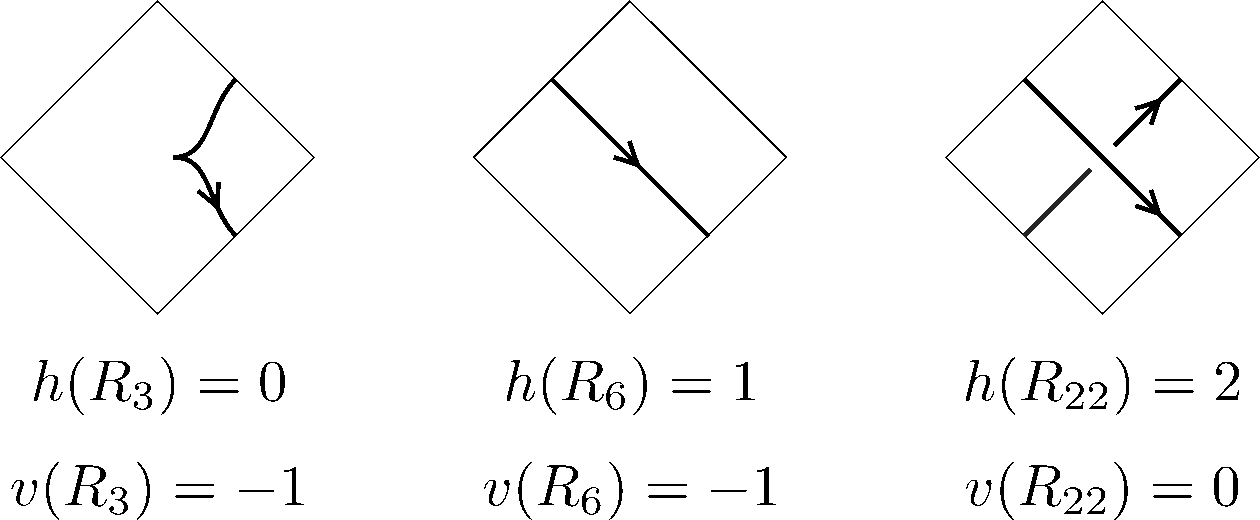}
    \caption{The values of the functions $h$ and $v$ for tiles $R_3$, $R_6$, and $R_{22}$.}
    \label{fig:hv_examples}
\end{figure}

\begin{prop}\label{prop:net_movement}
    Let $M$ be an oriented knot mosaic representing a Legendrian knot $\Lambda$. Then
    $$\sum_{i=0}^{24}|M|_{R_i}h(R_i)=\sum_{i=0}^{24}|M|_{R_i}v(R_i)=0.$$
\end{prop}

\begin{proof}
When tracing the path of a knot on a mosaic, one passes through all of the strands exactly once, ending where the tracing began. This means that the net horizontal and vertical movement of all the tiles in a knot mosaic must sum to 0.
\end{proof}

We will use the following consequence of this result to prove Theorem \ref{tony_bound} in Section \ref{combo_bounds}.

\begin{lemma} \label{lemma:vertical}
    Let $M$ be an oriented knot mosaic representing a Legendrian knot $\Lambda$. Let $U$ and $D$ denote the numbers of upward- and downward-oriented cusps in $M$, respectively, and let $N$ denote the number of negative crossings. If $U\geq D$, then
    \[
    2|\rot(\Lambda)| \leq 2N + |M|_{T_5} + |M|_{T_6}.
    \]
\end{lemma}

\begin{proof}
    First, note that the quantity $D-U=2\rot(\Lambda)$ captures the signed vertical movement through cusp tiles while traversing the oriented knot mosaic. (Also note that $U$ decreases the value of $2\rot(\Lambda)$, while $U$ increases the value of $v$, defined above.)
    Now, consider the following two sets of integers, which are disjoint:
    \[
    S:=\{3,5,8,10,17,18,21,24\},\;S':=\{2,6,7,11,15,23\}.
    \]
    By Figure \ref{fig:oriented_tiles}, an oriented tile $R_i$ contains a cusp if and only if $i\in S$. Therefore, our characterization of the quantity $2\rot(\Lambda)$ may be reformulated as the following equation:
    \[
        2\rot(\Lambda) = -\sum_{i\in S} |M|_{R_i}v(R_i).
    \]
    It follows from Proposition \ref{prop:net_movement} that
    \begin{equation} \label{cusp-movement}
        2\rot(\Lambda) = \sum_{0\leq i \leq 24, i\notin S} |M|_{R_i}v(R_i).
    \end{equation}
    
    It is also evident from Figure \ref{fig:oriented_tiles} that for a given $R_i$ with $i\notin S$ (i.e., $R_i$ containing no cusps), then $i\in\{2,6,15\}$ if and only if $v(R_i)<0$ (i.e., every strand in $R_i$ moves downward and $i\neq 0$). Since $\{2,6,15\}\subset S'$, it follows that
    \begin{equation} \label{no-cusps}
        \sum_{0\leq i\leq 24, i\notin S} |M|_{R_i}v(R_i) \geq \sum_{i\in S'} |M|_{R_i}v(R_i).
    \end{equation}
    Also, note that $i\in S'$ if and only if $R_i$ either contains a negative crossing (in which case $v(R_i)=\pm 2$) or is an oriented $T_5$ or $T_6$ tile (in which case $v(R_i)=\pm 1$). Therefore,
    \begin{equation} \label{eq:unoriented}
        -\sum_{i\in S'} |M|_{R_i}v(R_i) \leq 2N + |M|_{T_5} + |M|_{T_6}.
    \end{equation}

    Now, our assumption that $U\geq D$ implies that $2\rot(\Lambda)\leq 0$, or equivalently, $2|\rot(\Lambda)|=-2\rot(\Lambda)$.
    Combining this equality with (\ref{cusp-movement}) yields
    \begin{align*}
        2|\rot(\Lambda)| &= -\sum_{0\leq i\leq 24, i\notin S} |M|_{R_i}v(R_i) &\\
        & \leq -\sum_{i\in S'} |M|_{R_i}v(R_i) & \text{by }(\ref{no-cusps})\\
        & \leq 2N + |M|_{T_5} + |M|_{T_6}, & \text{by }(\ref{eq:unoriented})
    \end{align*}
    as desired.
\end{proof}

\subsection{A Combinatorial Approach}
\label{combo_bounds}
In this section, we provide two lower bounds for the mosaic number of a Legendrian knot in terms of its classical invariants via combinatorial arguments.
The first lower bound, stated in Theorem \ref{tony_bound} below, is most effective when the value of the quantity 
$$k :=|\rot(\Lambda)|+\tb(\Lambda)$$
is sufficiently high.
The proof relies on manipulating inequalities involving $k$, the number of positive and negative crossings, and the number of upward- and downward-oriented cusps of an oriented Legendrian knot mosaic.

\begin{theorem}\label{tony_bound}
    If $\Lambda$ is a Legendrian knot and $4|\rot(\Lambda)|+\operatorname{tb}(\Lambda)\geq 0$, then \[m(\Lambda) \geq \left\lceil\sqrt{4|\rot(\Lambda)|+\operatorname{tb}(\Lambda)}\right\rceil.\]
\end{theorem}

\begin{proof}
    Suppose a Legendrian knot $\Lambda$ can be represented by an  $n$-mosaic $M$. 
    We will show that $n\geq \sqrt{3|\rot(\Lambda)|+k}$. 
    For the particular front projection of $\Lambda$ represented in $M$, let $U$ and $D$ denote the numbers of upward- and downward-oriented cusps, respectively, and let $P$ and $N$ denote the numbers of positive and negative crossings, respectively. 

    Since $\rot(\Lambda)=\frac{1}{2}(D-U)$, orientation only affects the sign of $\rot(\Lambda)$. Without loss of generality, choose an orientation such that $U\geq D$. Since $D\geq 0$, this inequality implies $U\geq |D-U| = 2|\rot(\Lambda)|$. Then, by definition,
    \begin{align*}
        \operatorname{tb}(\Lambda) &= P-N -\frac{1}{2}(D+U) \\
        & = P-N - \frac{1}{2}(D-U)-U \\
        & = P-N-\rot(\Lambda)-U\\
        & \leq P-N + |\rot(\Lambda)| -2|\rot(\Lambda)|\\
        &= P-N -|\rot(\Lambda)|.
    \end{align*}
    %Since $r\leq 0$, this means $-U\leq 2\rot(\Lambda)$. 
    Substituting $\operatorname{tb}(\Lambda)=k-|\rot(\Lambda)|$ yields $k\leq P-N$, or,
    \begin{equation} \label{p-bound}
        P \geq k+N.
    \end{equation}

    Now, let $c:= \sum_{i\in\{2,4,8\}} |M|_{T_i}$ and observe that $c$ is the number of tiles in $M$ that contain cusps.
    Our earlier remark that $U\geq 2|\rot(\Lambda)|$ implies
    \[
    c \geq \frac{1}{2}U \geq |\rot(\Lambda)|.
    \]
    We can also consider the number of tiles in $M$ that do not contain cusps:
    \begin{align*}
        n^2 - c &\geq P+N + |M|_{T_5} + |M|_{T_6}\\
        & \geq k+2N + |M|_{T_5} + |M|_{T_6} & \text{by (\ref{p-bound})}\\
        &\geq k+2|\rot(\Lambda)|. & \text{by Lemma \ref{lemma:vertical}}
    \end{align*}
    Hence,
    \[
    n^2 \geq |\rot(\Lambda)| + k+2|\rot(\Lambda)|.
    \]
    Solving for $n$ completes the proof.
\end{proof}

Since the bound in Theorem \ref{tony_bound} increases as $k=|\rot(\Lambda)|+\tb(\Lambda)$ increases, it is most effective for Legendrian knots on the ``outermost boundary" of the mountain range of its smooth knot type. For example, consider the mountain range of Legendrian unknots (Figure \ref{fig:mtn}, left). For every Legendrian unknot $\Lambda$ on the outermost boundary of the mountain range, we have $k=-1$. Thus, Theorem \ref{tony_bound} tells us, for example, that any such unknot with $|\rot(\Lambda)|\geq 4$, $6$, or $9$ satisfies
\[m(\Lambda) \geq \left\lceil \sqrt{3|\rot(\Lambda)|+k} \right\rceil = \left\lceil \sqrt{3|\rot(\Lambda)|-1} \right\rceil = 4, 5, \text{or } 6,\] 
respectively. Our computational results in Figure \ref{fig:new-unknots} suggest that these bounds are not sharp. 
Nevertheless, the linear algebraic argument used to prove Theorem \ref{thm:lin-alg} in Section \ref{linear} replicates the bound in Theorem \ref{tony_bound}. This suggests that this bound cannot be significantly improved without considering the geometric properties of Legendrian knot mosaics.

For all Legendrian unknots not on the outermost boundary of the mountain range, we have $k<-1$, making the bound in Theorem \ref{tony_bound} less useful.
This issue motivates the bound in Theorem \ref{sqrtbound} below. This bound is usually more powerful than the one in Theorem \ref{tony_bound}, and it is based solely on $\tb(\Lambda)$.
The proof takes an extremal approach by considering every possible way the tiles in the boundary and the inner board can contribute to the Thurston-Bennequin invariant of a Legendrian knot. We also employ the observation in Propositions 3.5 and 3.6 of \cite{pezzimenti} that a mosaic diagram of a Legendrian link containing $b$ tiles in its outer boundary can have at most $b/2$ cusps in its outer boundary.

\begin{theorem}\label{sqrtbound}
    If $\Lambda$ is a Legendrian knot with negative Thurston-Bennequin number, then
    \[
    m(\Lambda)\geq \left\lceil \sqrt{-\operatorname{tb}(\Lambda)-\frac{3}{4}} + \frac{3}{2} \right\rceil.
    \]
\end{theorem}

\begin{proof}
    Consider a Legendrian $n$-mosaic of $\Lambda$.
    Since none of the $4n-4$ boundary tiles can contain crossings, the only way boundary tiles can contribute to $\tb(\Lambda)$ is by having cusps. Since these tiles contain at most $(4n-4)/2 = 2n-2$ cusps, they collectively contribute $-(n-1)$ or more to $\tb(\Lambda)$.
    
    Now, consider the $(n-2)^2$ tiles in the inner board.
    Each may contain either a crossing (contributing $-1$ or $1$ to $\tb(\Lambda)$), a cusp (contributing $-1/2$ to $\tb(\Lambda)$), or neither (contributing $0$ to $\tb(\Lambda)$).
    In other words, each tile in the inner board contributes $-1$ or more to $\tb(\Lambda)$, so they collectively contribute $-(n-2)^2$ or more to $\tb(\Lambda)$.
    Altogether,
    \[
    \operatorname{tb}(\Lambda) \geq -(n-2)^2 -(n-1),
    \]
    and thus,
    \[
    -\operatorname{tb}(\Lambda) \leq n^2-3n+3 =\left(n-\frac{3}{2}\right)^2 + \frac{3}{4}.
    \]
   Solving for $n$ achieves the desired inequality.
\end{proof}

\begin{figure}[H]
    \centering
    \includegraphics[scale=0.1]{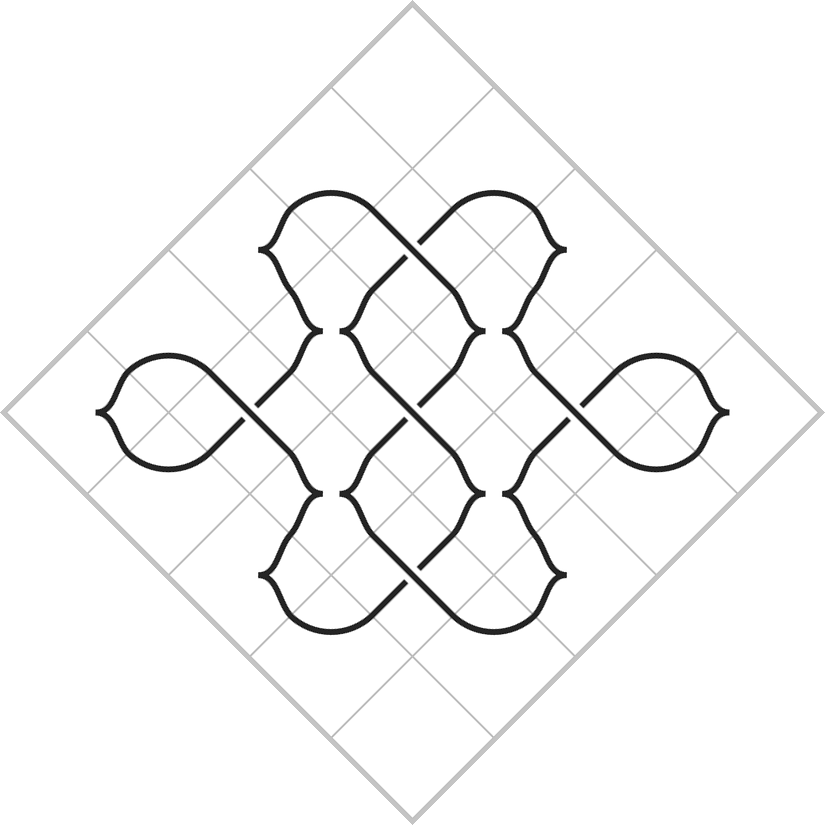}
    \includegraphics[scale=0.1]{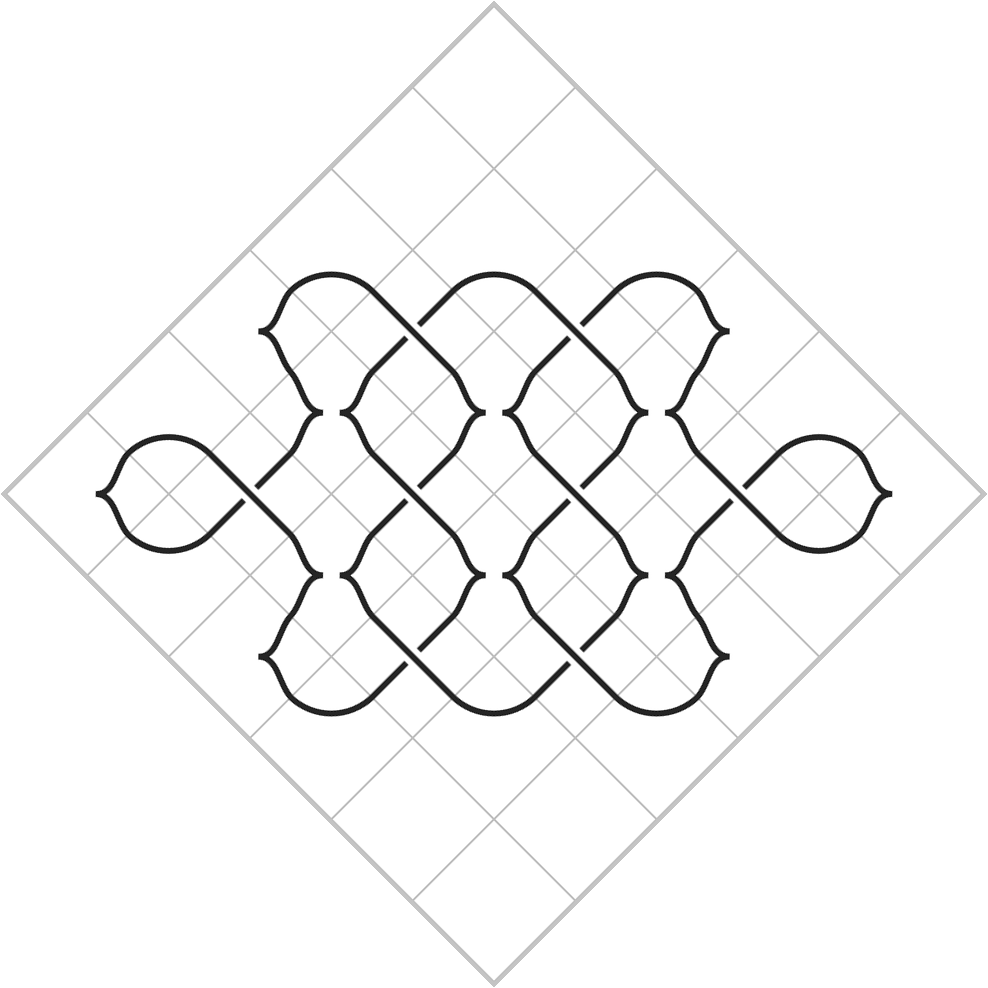}
    \includegraphics[scale=0.1]{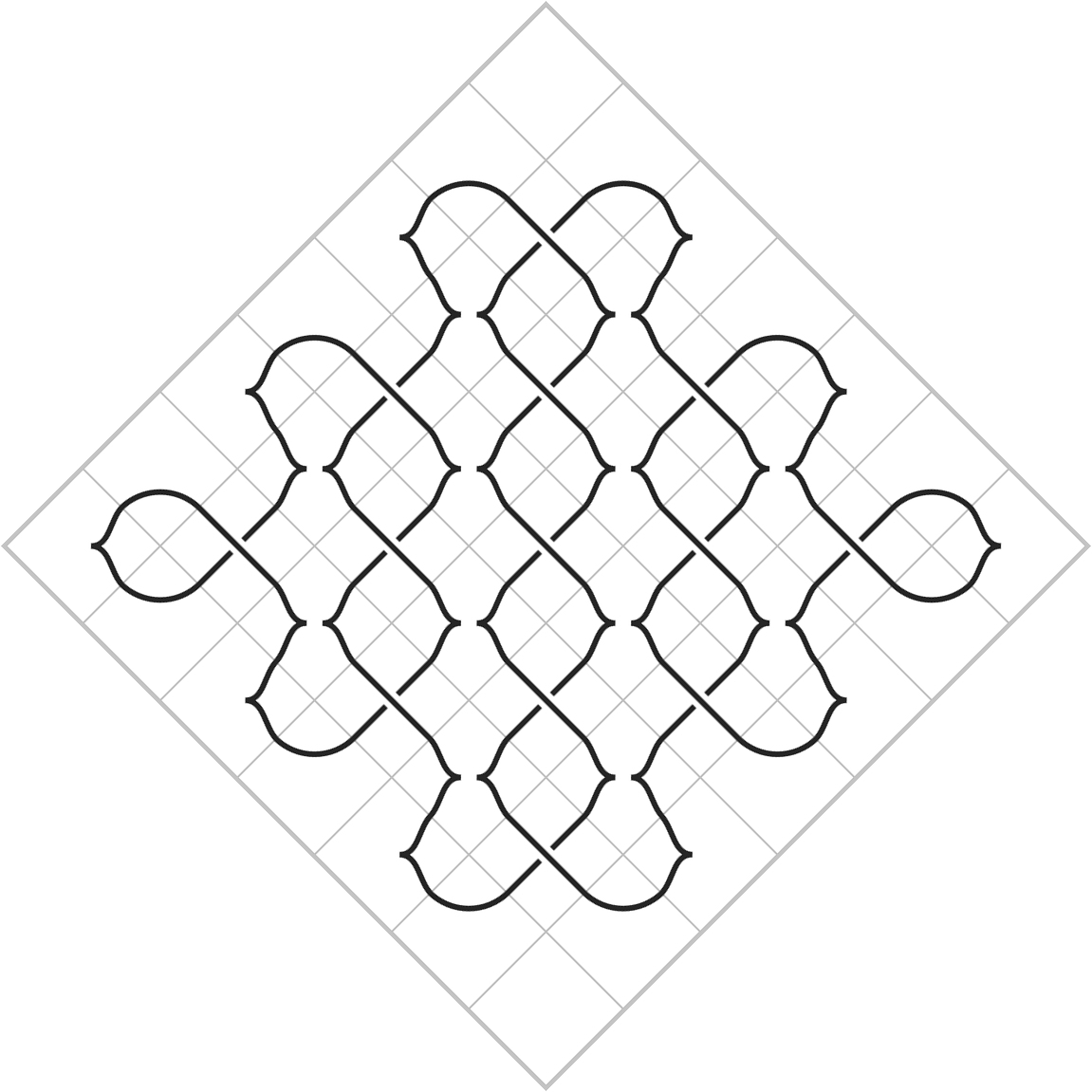}
    \includegraphics[scale=0.1]{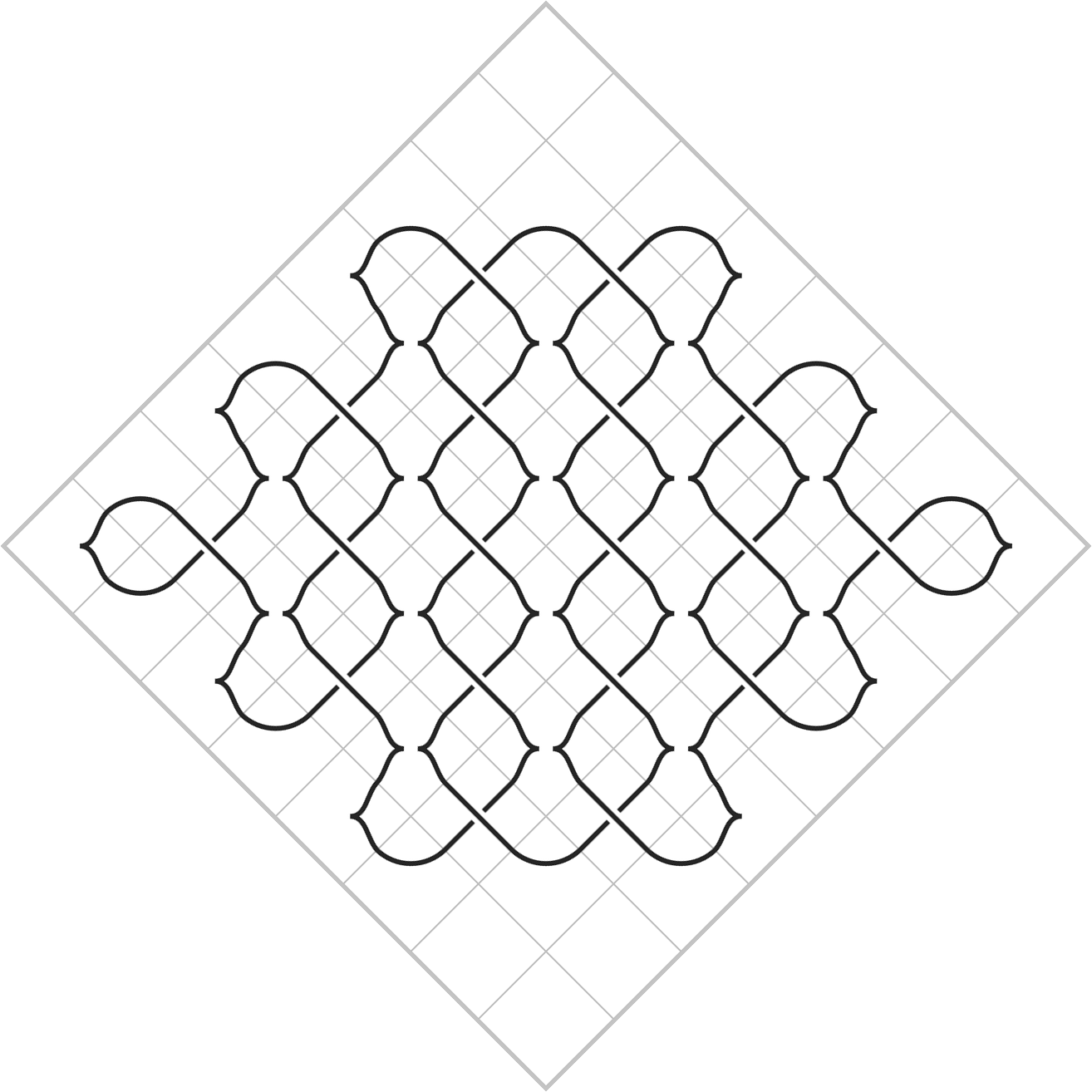}
    \includegraphics[scale=0.1]{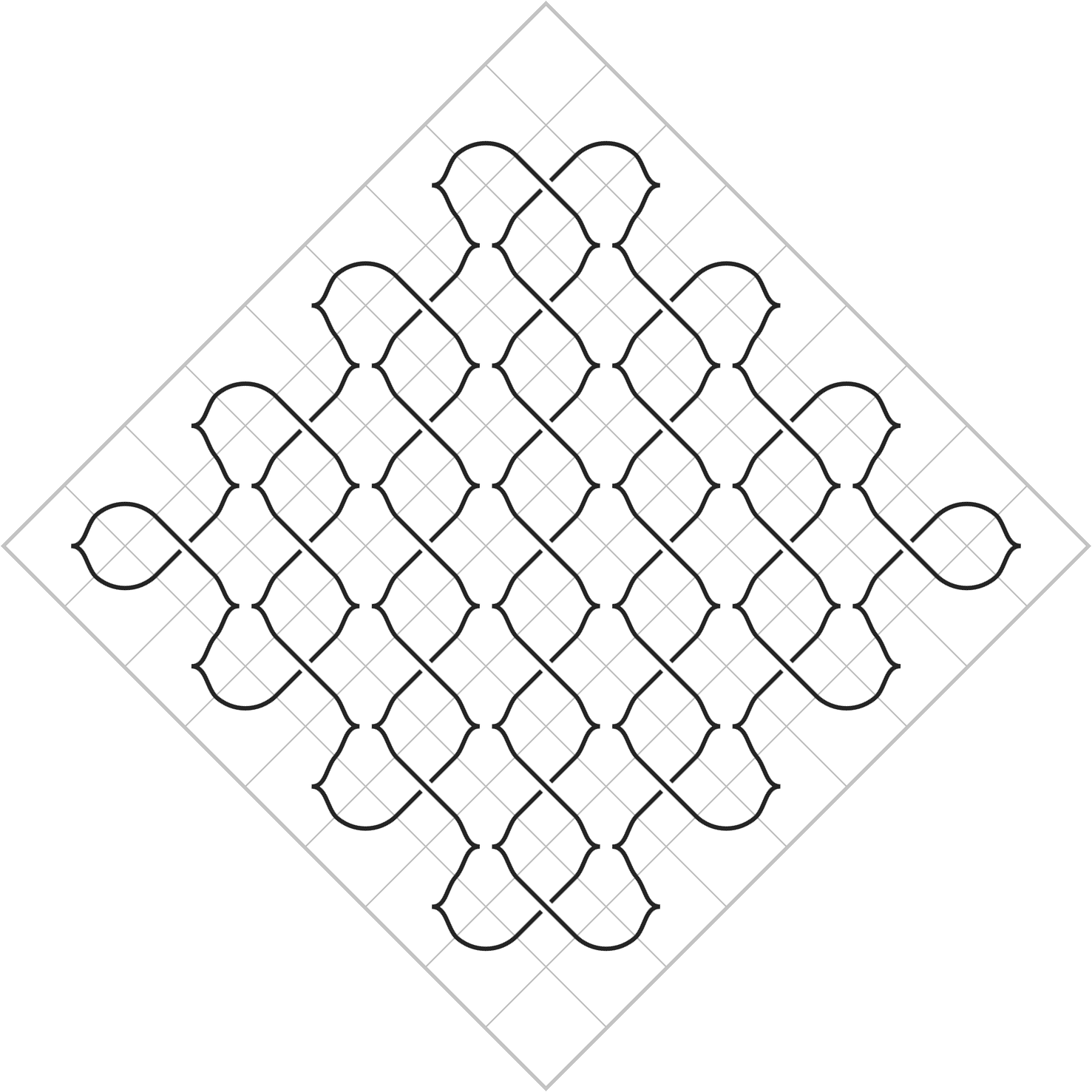}
    \caption{Crab buckets $\beta_5,\beta_6,\dots,\beta_{9}$.}
    \label{fig:crabs}
\end{figure}

The bound in Theorem \ref{sqrtbound} is sharp in infinitely many cases; in fact, it is attained by Legendrian representatives of infinitely many distinct smooth knot types. To show this, we construct an infinite sequence of Legendrian knots $\beta_n$, for $n\geq5$ that each attain this bound. We call this sequence the sequence of \emph{crab buckets} and refer to each $\beta_n$ as the $n$th \emph{crab bucket}. (See Section \ref{questions} for an explanation of the inspiration for this name.)
Figure \ref{fig:crabs} depicts the first five crab buckets $\beta_5,\dots,\beta_9$. For all $n\geq 5$, the $n$th crab bucket $\beta_n$ can be constructed on a Legendrian $n$-mosaic as follows:
\begin{enumerate}
    \item Starting with the tile in position $(2,2)$, add as many nonadjacent $T_{10}$ tiles into the inner board as possible.
    \item If $n$ is even, then add $T_1$ to the tile in position $(2,n-1)$ and $T_3$ to the tile in position $(n-1,2)$.
    \item Fill all remaining empty tiles in the inner board with $T_8$.
    \item Add $T_1$, $T_2$, $T_3$, and $T_4$ tiles to the outer boundary of the mosaic to form a knot. (Note that there is only one way to do this due to the restricted movement of the strands connecting to the $T_{10}$ tile in position $(2,2)$, cf.\ the \emph{twofold rule} stated in Section 2 of \cite{mosaic-count}.)
\end{enumerate}
The steps of this construction in the odd and even cases are illustrated in Figures \ref{fig:crabstruction7} and \ref{fig:crabstruction8}, respectively.

Recall that the \textit{Legendrian connected sum} operation, denoted by $\#$, is a way of gluing Legendrian knots that preserves their orientations and tangency to the standard contact structure, creating a composite Legendrian knot.
(For a formal treatment of the Legendrian connected sum operation, we refer the reader to \cite{connect-sum}.)
Topologically, each crab bucket is the Legendrian connected sum of Legendrian $(2,q)$-torus knots, as specified in Observation \ref{cb-obs} below.

\begin{figure}
    \centering
    \includegraphics[width=0.9\linewidth]{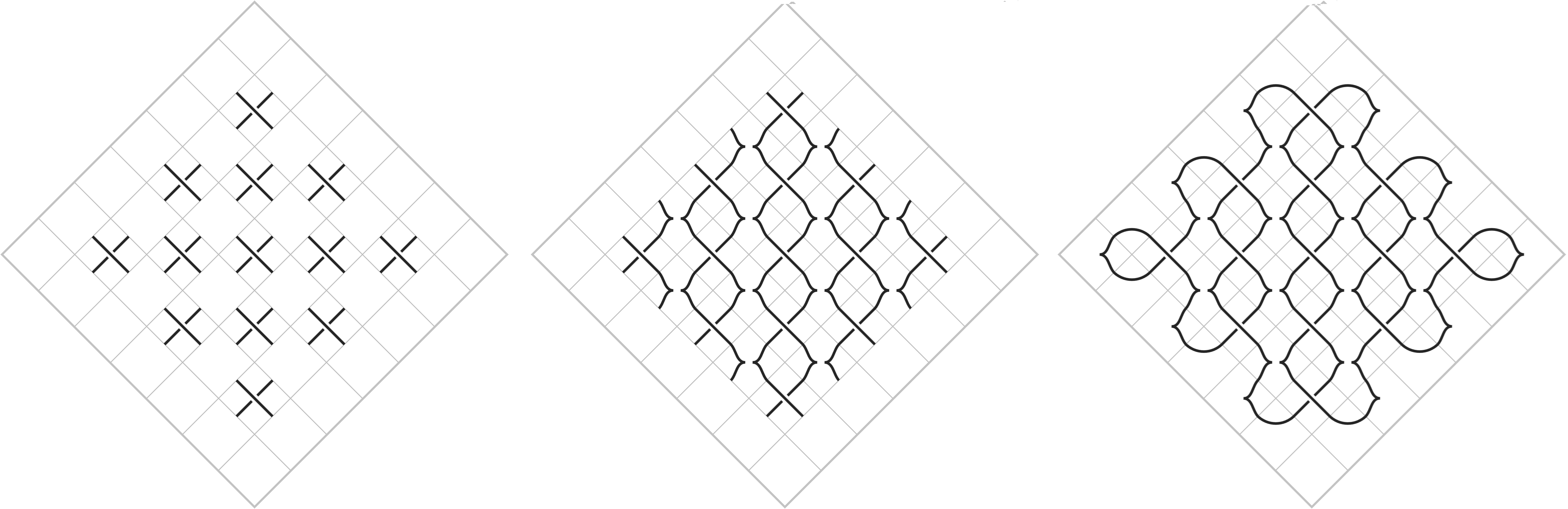}
    \caption{The construction of $\beta_7$.}
    \label{fig:crabstruction7}
\end{figure}

\begin{figure}
    \centering
    \includegraphics[width=0.9\linewidth]{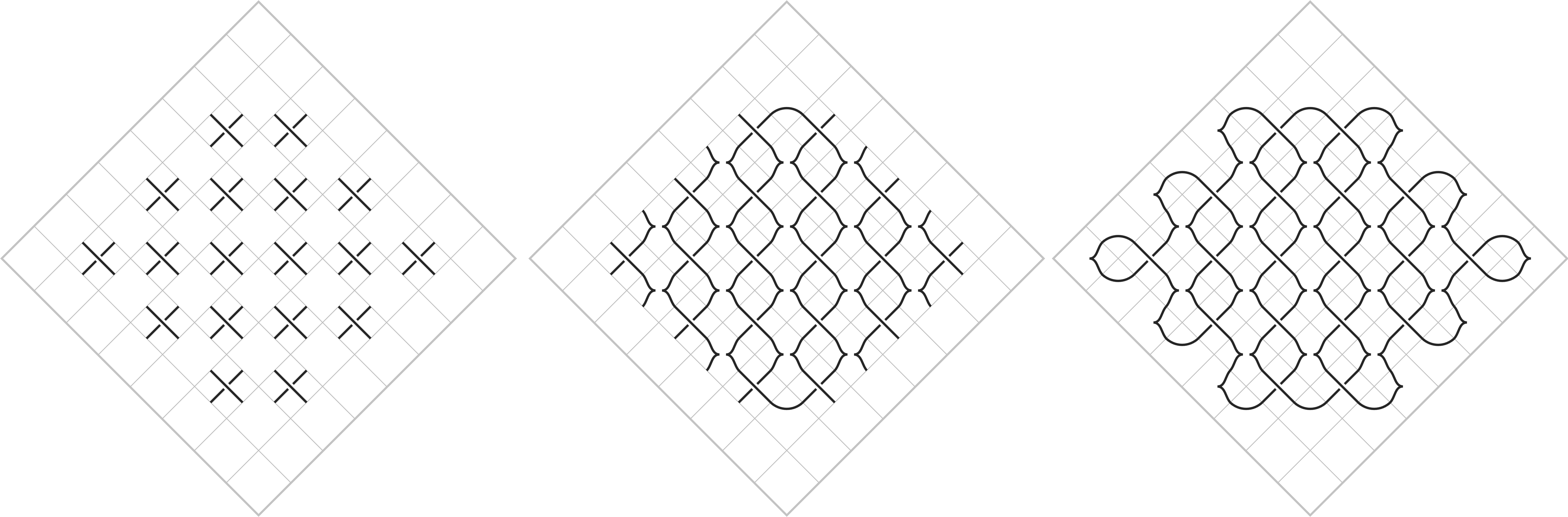}
    \caption{The construction of $\beta_8$.}
    \label{fig:crabstruction8}
\end{figure}

\begin{obs} \label{cb-obs}
    The crab bucket $\beta_5$ is the Legendrian $(2,3)$-torus knot (i.e., the Legendrian negative trefoil). 
        For all even $n\geq 6$, $\beta_n$ is the Legendrian connected sum of Legendrian torus knots
    \[
    (2,3) \# (2,5) \# \dots \# (2,n-3) \# (2, n-3) \# (2,n-5) \# \dots \# (2,3).
    \]
    For all odd $n\geq 7$, the crab bucket $\beta_{n}$ is the Legendrian connected sum of Legendrian torus knots
    \[
    (2,3) \# (2,5) \# \dots \# (2,n-2) \# (2,n-4) \# \dots \# (2,3).
    \]
\end{obs}
In particular, no two crab buckets share the same smooth knot type.

\begin{theorem} \label{crab-mosaic-number}
    Every member of the crab bucket family $\beta_{5},\beta_{6},\dots$ attains the bound in Theorem \ref{sqrtbound}. In particular,  $m(\beta_n)=n$ for all $n\geq 5$. 
\end{theorem}

\begin{proof}
    By construction, $n\geq m(\beta_n)$ for all $n\geq 5$.
    With this in mind, we will begin by considering all odd $n\geq 5$.
    By construction, every tile on the inner board of the crab bucket mosaic projection of $\beta_n$ contains either a negative crossing or two cusps.
    Therefore, these $(n-2)^2$ inner tiles each contribute $-1$ to $\operatorname{tb}(\beta_n)$. Meanwhile, the boundary tiles collectively contain $(n-1)+(n-3)=2n-4$ cusps and no crossings, so
    \[
    \operatorname{tb}(\beta_n) = -(n-2)^2 -\frac{1}{2}(2n-4) = -(n-1)(n-2).
    \]
    In particular, $\tb(\beta_n)<0$, so Theorem \ref{sqrtbound} implies that if we define
    \[f(x):=\sqrt{(x-1)(x-2)-\frac{3}{4}} + \frac{3}{2},\]
    then for all odd $n\geq 5$, we have
    \[
    m(\beta_n)\geq \lceil f(n)\rceil.
    \]
    However, $0.15>x-f(x)> 0$ for all $x\geq 5$. (Indeed, some calculus shows that $x-f(x)$ is monotone decreasing on $(2.5,\infty)$ and converges to 0 as $x$ tends toward infinity.) Hence, $\lceil f(n)\rceil = n$. Since $n \geq m(\beta_n)$,
    this proves the claim for all odd $n\geq 5$.
    %, giving us $m(\beta_n)=\lceil f(n)\rceil=n$ as desired.

    Now, suppose $n\geq 6$ is even.
    The only tiles in the inner board not containing crossings or cusps are the $T_1$ tile in position $(2, n-1)$ and the $T_3$ tile in position $(n-1,2)$, each contributing 0 to $\tb(\beta_n)$.
    By construction, each of the other tiles in the inner board contains either a negative crossing or two cusps.
    Therefore, these $(n-2)^2-2$ tiles each contribute $-1$ to $\operatorname{tb}(\beta_n)$. Meanwhile, the boundary tiles collectively contain $(n-2)+(n-4)=2n-6$ cusps and no crossings, so
    \[
    \operatorname{tb}(\beta_n) = -\left[(n-2)^2 -2\right] -\frac{1}{2}(2n-6) = -n^2+3n+1.
    \]
    In particular, $\tb(\beta_n)<0$, so Theorem \ref{sqrtbound} implies that if we define
    \[g(x):=\sqrt{x^2-3x-1-\frac{3}{4}} + \frac{3}{2},\]
    then for all even $n\geq 6$, we have
    \[
    m(\beta_n)\geq \lceil g(n)\rceil.
    \]
    However, $0.5>x-g(x)> 0$ for all $x\geq 6$. (Once again, some calculus shows that $x-g(x)$ is monotone decreasing on $(3.5,\infty)$ and converges to 0 as $x$ tends toward infinity.) Hence,
    $\lceil g(n)\rceil = n $. Since $n\geq m(\beta_n)$, this proves the claim for all even $n\geq 6$.
    %, giving us $m(\beta_n)=\lceil g(n)\rceil=n$ and completing the proof.
\end{proof}

\subsection{A Linear Algebraic Approach} \label{linear}

 We now provide an alternative method for deriving bounds similar to those obtained in the previous section.
 Although the bounds are similar (in fact, slightly weaker), the approach is more robust. %and indicates that improving these bounds would require consideration of additional geometric information.
 %%%I reworded to spin in a more positive light. Not sold to this wording -S
 %This approach is more robust, but it results in slightly weaker bounds.

% First, rather than considering the regular 10 unoriented tiles, we can distinguish each tile by all possible combinations of orientation for strands on that tile, yielding the 25 distinct oriented tiles shown in in Figure \ref{fig:oriented_tiles}. Given an oriented Legendrian knot mosaic $M$, define $|M|_R$ as the number of times that the tile $R$ appears in $M$. 

When working with the oriented tiles in Figure \ref{fig:oriented_tiles}, it is possible to track the effect each individual tile has on the classical invariants of the Legendrian knot it is a part of. For any oriented tile $R_i$, define $\tb^*(R_i) := P-N-\frac{1}{2}C$, where $P$ and $N$ are the number of positive and negative crossings appearing in $R_i$, respectively, and $C$ is the number of cusps appearing in $R_i$. In a similar fashion, define $\rot^*(R_i) := \frac{1}{2}(D-U)$, where $D$ and $U$ are the number of downward- and upward-oriented cusps, respectively, appearing in $R_i$. These definitions naturally lead to the following proposition. 

% \begin{figure}
%     \centering
%     \includegraphics[width=1\linewidth]{oriented_tiles.png}
%     \caption{The 25 distinct oriented Legendrian knot mosaic tiles (the labeling is arbitrary).}
%     \label{fig:oriented_tiles}
% \end{figure}

\begin{prop}\label{prop:classical_sum}
    Let $M$ be an oriented knot mosaic representing a Legendrian knot $\Lambda$. Then
    \begin{itemize}
        \item[i.] $\tb(\Lambda)=\sum_{i=0}^{24}|M|_{R_i}\tb^*(R_i)$, and
        \item[ii.] $\rot(\Lambda)=\sum_{i=0}^{24}|M|_{R_i}\rot^*(R_i)$.
    \end{itemize}
\end{prop}

%Also, for any tile $R$, we define $h(R)$ as the net horizontal movement of strands within $R$, measured in units of length $\frac{\sqrt{2}}{2}$ times the side length of a tile. Similarly, define $v(R)$ as the net vertical movement of strands within $R$. For example, Fig \ref{fig:hv_examples} shows the values of these functions for tiles $R_3$, $R_6$, and $R_{22}$. 
% These functions are useful to us because we can use them to restrict the possible valid combinations of tiles that can appear in a Legendrian knot mosaic, as shown by Proposition \ref{prop:net_movement}.

% % \begin{figure}
% %     \centering
% %     \includegraphics[width=0.6\linewidth]{hv_examples.png}
% %     \caption{The values of the functions $h$ and $v$ for tiles $R_3$, $R_6$, and $R_{22}$.}
% %     \label{fig:hv_examples}
% % \end{figure}

% \begin{prop}\label{prop:net_movement}
%     Let $M$ be an oriented knot mosaic representing a Legendrian knot $\Lambda$. Then
%     $$\sum_{i=0}^{24}|M|_{R_i}h(R_i)=\sum_{i=0}^{24}|M|_{R_i}v(R_i)=0$$
% \end{prop}

% \begin{proof}
% When tracing the path of a knot on a mosaic, one passes through all of the strands exactly once, ending where the tracing began. This means that the net horizontal and vertical movement of all the tiles in a knot mosaic must sum to 0.
% \end{proof}

For any $0\leq i\leq24$, define the vector

$$\textbf{p}_i:=\begin{pmatrix}\tb^*(R_i)\\\rot^*(R_i)\\h(R_i)\\v(R_i)\\1\end{pmatrix},$$
which records relevant properties of the tile $R_i$. The 1 in the final coordinate will help to keep track of how many tiles are contained in the mosaic. Define the $5\times25$ matrix

$$P:=\begin{bmatrix}\textbf{p}_0&\textbf{p}_1&\dots&\textbf{p}_{24}\end{bmatrix}.$$

When viewed as a linear transformation from $\R^{25}$ to $\R^{5}$, $P$ can be thought of as a matrix that transforms the amounts of each tile type present into a vector containing important properties describing the resulting mosaic. This idea is captured by Proposition \ref{prop:P_relavance}.

\begin{prop}\label{prop:P_relavance}
    Let $M$ be an $n\times n$ oriented knot mosaic representing a Legendrian knot $\Lambda$. Let
    $$\textbf{c}:=\begin{pmatrix}|M|_{R_0}\\|M|_{R_1}\\\vdots\\|M|_{R_{24}}\end{pmatrix}.$$
    Then the product
    $$P\textbf{c}=\begin{pmatrix}\tb(\Lambda)\\\rot(\Lambda)\\0\\0\\n^2\end{pmatrix}.$$
\end{prop}

\begin{proof}
    Using the definitions of $P$ and $\textbf{c}$, carrying out the matrix multiplication, and then applying Propositions \ref{prop:classical_sum} and \ref{prop:net_movement} yields
    $$P\textbf{c}=\begin{pmatrix}\sum_{i=0}^{24}|M|_{R_i}\tb^*(R_i)\\\sum_{i=0}^{24}|M|_{R_i}\rot^*(R_i)\\\sum_{i=0}^{24}|M|_{R_i}h(R_i)\\\sum_{i=0}^{24}|M|_{R_i}v(R_i)\\\sum_{i=0}^{24}|M|_{R_i}\end{pmatrix}=\begin{pmatrix}\tb(\Lambda)\\\rot(\Lambda)\\0\\0\\n^2\end{pmatrix}.$$

    The last coordinate sums the amount of each tile type, resulting in the total number of tiles in the mosaic, $n^2$.
\end{proof}

We have now established the foundation needed to prove Theorem \ref{thm:lin-alg} below. Note that the first bound is precisely the bound in Theorem \ref{tony_bound}, and the second bound is always $1$ or $2$ lower than the bound in Theorem \ref{sqrtbound}. While the statement of Theorem \ref{thm:lin-alg} does not provide any sharper restrictions on mosaic number, this alternative approach suggests that the bounds in Section \ref{combo_bounds} cannot be significantly improved without considering the geometric properties of Legendrian knot mosaics.

\begin{theorem} \label{thm:lin-alg}
    Let $M$ be an $n\times n$ oriented Legendrian knot mosaic representing a Legendrian knot $\Lambda$.
    \begin{itemize}
        \item[i.] If $4|\rot(\Lambda)|+\tb(\Lambda)\geq0$, then $m(\Lambda)\geq\left\lceil\sqrt{4|\rot(\Lambda)|+\tb(\Lambda)}\right\rceil$.
        \item[ii.] If $\tb(\Lambda)\leq0$, then $m(\Lambda)\geq\left\lceil\sqrt{-\tb(\Lambda)}\right\rceil$.
    \end{itemize}
\end{theorem}

%are in some sense the ``best possible" bounds (\textcolor{red}{SP: I don't love this wording of ``best possible"... can we come up with another way to say this?}) \textcolor{blue}{[TW: cannot be significantly improved without considering the geometric side of the mosaics]} on the mosaic number of a Legendrian knot in terms of its classical invariants. % let me know if there's a better way to word this -Luc

\begin{proof}
    % As in Proposition \ref{prop:P_relavance}, define
    % $$c=\begin{pmatrix}|M|_{R_0}\\|M|_{R_1}\\\vdots\\|M|_{R_{24}}\end{pmatrix}.$$
    Consider the vector $\textbf{c}$ defined in Proposition  \ref{prop:P_relavance} above.
    Since each $|M|_{R_i}$ is nonnegative, we have
    $$\textbf{c}\in\R^{25}_{\geq0}=\{(x_1,x_2,\dots,x_{25})\in\R^{25}:x_i\geq0\text{ for all }1\leq i\leq 25\}.$$
    Thus, $P\textbf{c}\in P(\R^{25}_{\geq0})$, where $P(\R^{25}_{\geq0})$ is the image of $\R^{25}_{\geq0}$ under the linear transformation $P$.
    Define 
    $$V:=\{(x_1,x_2,x_3,x_4,x_5)\in\R^5:x_3=x_4=0\}.$$ 
    By Proposition \ref{prop:net_movement}, any $\textbf{c}\in\R^{25}_{\geq0}$ which represents a valid Legendrian knot mosaic will have $P\textbf{c}\in V$. So we must have $P\textbf{c}\in P(\R^{25}_{\geq0})\cap V$. To calculate $P(\R^{25}_{\geq0})$, we employ the right inverse of $P$, which we will denote as $Q$. (Note that  $Q$ must exist as it can be confirmed  computationally that $P$ has full rank.) Define
    $$S=\{\textbf{y}\in\R^5:\text{There exists } \textbf{v}\in\ker(P)\text{ such that }Q\textbf{y}+\textbf{v}\in\R^{25}_{\geq0}\}.$$

    We claim that $P(\R^{25}_{\geq0})=S$. If $\textbf{y}\in S$, then for some $\textbf{v}\in\ker(P)$,
    $$\textbf{y}=PQ\textbf{y}+\textbf{0}=PQ\textbf{y}+P\textbf{v}=P(Q\textbf{y}+\textbf{v})\in P(\R^{25}_{\geq0}).$$
    So $S\subseteq P(\R^{25}_{\geq0})$. Conversely,  if $\textbf{y}\in P(\R^{25}_{\geq0})$, there  exists $\textbf{x}\in\R^{25}_{\geq0}$ such that $P\textbf{x}=\textbf{y}$. Let $\textbf{v}=\textbf{x}-Q\textbf{y}$. Then,
    $$P\textbf{v}=P(\textbf{x}-Q\textbf{y})=P\textbf{x}-PQ\textbf{y}=\textbf{y}-\textbf{y}=\textbf{0}.$$
    So, $\textbf{v}\in\ker(P)$.
    We also have
    $$Q\textbf{y}+\textbf{v}=Q\textbf{y}+\textbf{x}-Q\textbf{y}=\textbf{x}\in\R^{25}_{\geq0}.$$
    So, by definition, $\textbf{y}\in S$. Thus, $P(\R^{25}_{\geq0})=S$.

    We can now take advantage of this fact to calculate $P(\R^{25}_{\geq0})$ using the definition of $S$. Using a computer algebra system such as Mathematica to carry out quantifier elimination and remove all dependencies on $\textbf{v}$ from the system of inequalities, results in a set of inequalities which constrain $P(\R^{25}_{\geq0})\cap V$. The details of this calculation can be found in Appendix \ref{app:lin-alg-details}. The calculation tells us that if $\textbf{c}\in\R^{25}_{\geq0}$, both of the following inequalities must be satisfied.
    \begin{align*}
    n^2&\geq4|\rot(\Lambda)|+\tb(\Lambda);\\
        n^2&\geq-\tb(\Lambda).
    \end{align*}
  Restricting the domains to satisfy $4|\rot(\Lambda)|+\tb(\Lambda)\geq0$ and $-\tb(\Lambda)\geq0$, respectively, yields  
    \begin{align*}
        n&\geq\sqrt{4|\rot(\Lambda)|+\tb(\Lambda)};\\
        n&\geq\sqrt{-\tb(\Lambda)},
    \end{align*}
    as desired.
 
\end{proof}

\section{Upper Bounds for Legendrian Unknots} \label{unknot-section}

\subsection{Barn Tiles and Soil Setups}

In \cite{pezzimenti}, Pezzimenti and Pandey constructed an infinite family of Legendrian unknots 
called \emph{the Kraken sequence}
that realize their mosaic numbers only in non-reduced projections (i.e., projections having more than the minimum number of crossings). 
In this section, we define modified configurations for Legendrian $n$-mosaics. We also define operations on these configurations using components from the Kraken sequence. We will use these configurations to give a method of constructing mosaics of Legendrian unknots. From this construction, we obtain an upper bound on the mosaic number of any Legendrian unknot.

    \begin{figure}
        \centering
        \includegraphics[width=0.95\linewidth]{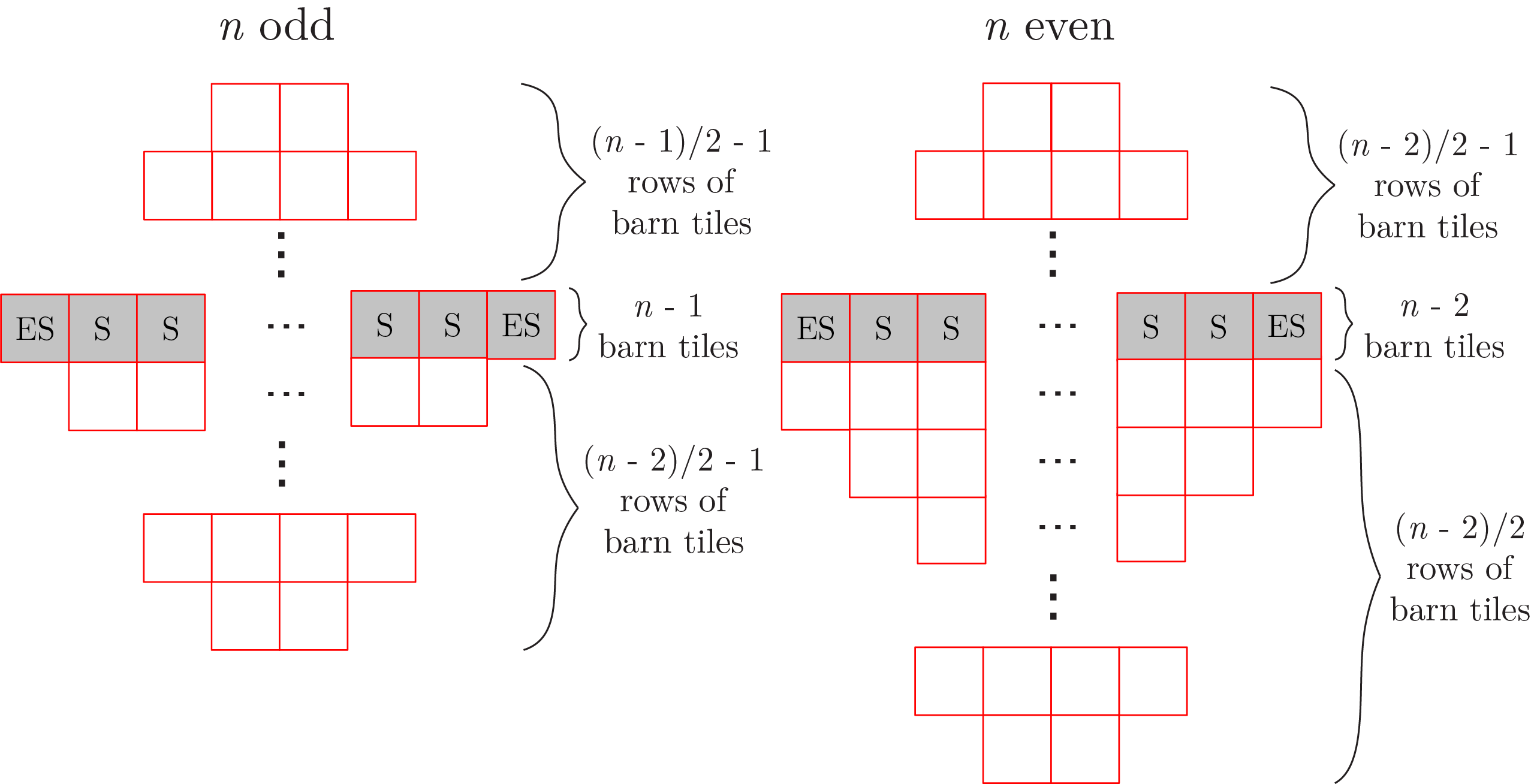}
        \caption{The default barn configuration of a Legendrian $n$-mosaic.}
        \label{fig:soil-setup-proof}
    \end{figure}

We begin the construction by defining a specified mosaic configuration.
Given a Legendrian $n$-mosaic, a \textbf{barn configuration} is a grid formed by connecting the corners of adjacent tiles. Each resulting square is called a \textbf{barn tile}, named so because the ``X"-shape in each barn tile resembles the cross-brace design on sliding barn doors.
Explicit examples of barn tiles are shown in Figures \ref{fig:default-soil} and \ref{fig:barn-effects}.

While there are several valid barn configurations of a particular $n$-mosaic, we are interested in the one such that the tile in position $(1,n)$ is not contained in any barn tiles. We call this specific barn configuration the \textbf{default barn configuration}.
In the default barn configuration, there are $n-2$ rows of barn tiles. In particular, if $n$ is odd, there are rows of barn tiles of lengths $2, 4,\dots, n-1,n-3,\dots, 2$, going from top to bottom. If $n$ is even, there are rows of barn tiles of lengths $2,4,\dots, {n-2},{n-2},{n-4},{n-6},\dots,2$, going from top to bottom. Figures \ref{fig:soil-setup-proof} and \ref{fig:default-soil} depict the geometry of the default barn configuration for both odd and even $n$.

\begin{figure}
    \centering
    \includegraphics[width=0.7\linewidth]{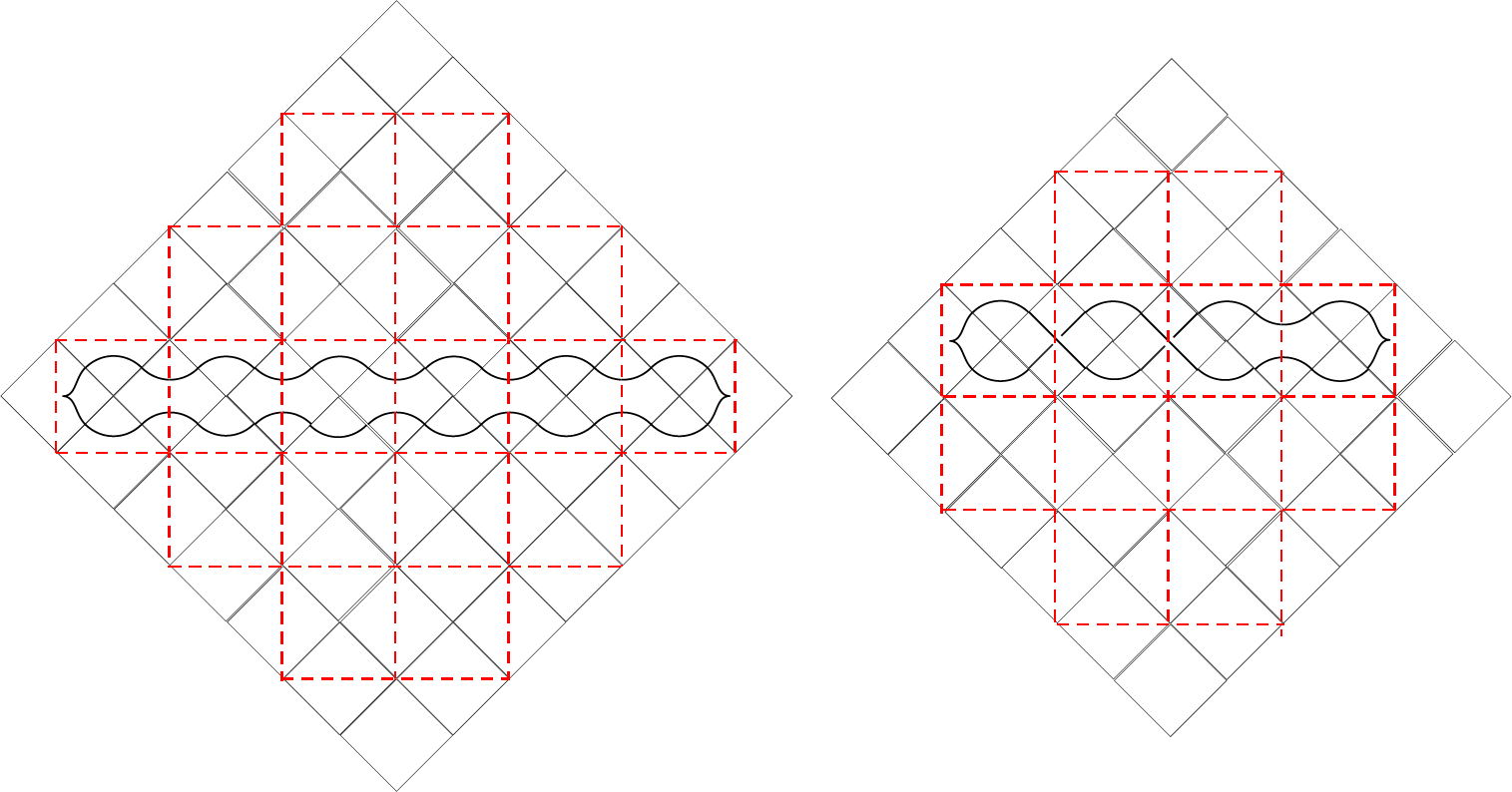}
    \caption{A Legendrian 7-mosaic and a Legendrian 6-mosaic, each having a soil setup. Note that each mosaic depicts an unknot.}
    \label{fig:default-soil}
\end{figure}

We  begin to fill the default barn configuration with tiles as follows. Let $i:=0$ if $n$ is odd and $i:=1$ if $n$ is even. Let the mosaic tiles in positions ${(1,1+i)}$ and $(n-i,n)$ be $T_2$ and $T_4$, respectively. Also, let the tiles in positions $(1,2+i),(2,3+i),\dots,(n-1-i,n)$ be $T_1$. Likewise, let the tiles in positions ${(2,1+i)},{(3,2+i)},\dots,{(n-i,n-1)}$ be $T_3$. 
Now, consider the $n-2-i$ mosaic tiles in positions ${(2,2+i)},{(3,3+i)},\dots,$ ${(n-1-i,n-1)}$, and let each of them be either $T_7$ or $T_{10}$. We call the 
$n-1-i$ barn tiles intersecting these mosaic tiles \textbf{soil tiles}. In particular, we call the leftmost and rightmost soil tiles \textbf{edge-soil tiles}. Note that the soil tiles are precisely the barn tiles in the $[(n-1-i)/2]$th row of barn tiles.

If we insert tiles in a Legendrian $n$-mosaic having the default barn configuration exactly as described above, then we say the resulting mosaic has a \textbf{soil setup}. Figure \ref{fig:default-soil} gives examples of soil setups for $n=7$ and $n=6$. Note that any barn tile in a soil setup is empty if and only if it is not a soil tile.
By construction, we have the following:
\begin{obs} \label{obs:default-soil}
    If $M$ is a Legendrian $n$-mosaic having a soil setup, then $M$ depicts a Legendrian unknot $\Lambda_U$ with $\tb(\Lambda_U)=-|M|_{T_{10}}-1$. Moreover, $\rot(\Lambda_U)=0$ if $|M|_{T_{10}}$ is even, and $\rot(\Lambda_U)=\pm 1$ if $|M|_{T_{10}}$ is odd.
\end{obs}

The last part of this observation follows from the fact that the two cusps in the edge-soil tiles of $M$ share the same orientation if and only if $|M|_{T_{10}}$ is odd.

\begin{figure}[h]
    \centering
    \includegraphics[width=0.5\linewidth]{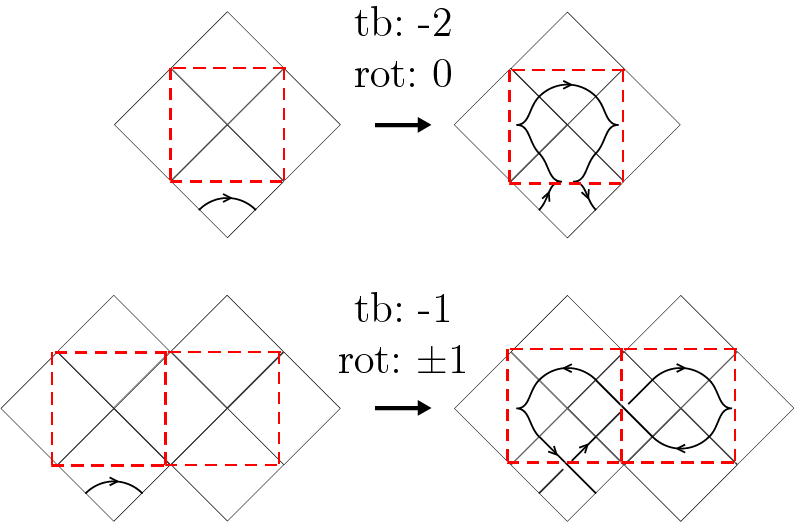}
    \caption{The Kraken (top) and fish (bottom) moves and their effects on the classical invariants.}
    \label{fig:barn-effects}
\end{figure}

Next, we define two moves which will ``fill" the non-soil barn tiles. The resulting structures will appear to ``grow out of the soil" vertically.
The first of these moves is inspired by the Kraken sequence of \cite{pezzimenti}. Consider an empty barn tile inscribed in four mosaic tiles. Among these four mosaic tiles, suppose that either the bottom tile is $T_1$ or the top tile is $T_3$. Then, we define the \textbf{Kraken move} as the move depicted in the first row of Figure \ref{fig:barn-effects}, up to reflection about the horizontal axis.

Now, consider two adjacent empty barn tiles inscribed in seven mosaic tiles. Among these seven mosaic tiles, suppose that at least one of the following is true: one or both of the bottommost tiles is $T_1$, or one or both of the topmost tiles is $T_3$. Then, we define the \textbf{fish move} (named for its fish-like appearance) as the move depicted in the second row of Figure \ref{fig:barn-effects}, up to reflections about the horizontal axis, the vertical axis, or both axes. Note that the fish move changes one of the aforementioned $T_1$ or $T_3$ tiles to a $T_{10}$ tile. 

By Figure \ref{fig:barn-effects}, performing a fish move on two empty barn tiles changes both of the mosaic tiles containing the top (resp.\ bottom) edges of those barn tiles to $T_1$ (resp.\ $T_3$). A similar statement holds for Kraken moves. In a soil setup, the mosaic tile containing the top (resp.\ bottom) edge of any soil tile is also $T_1$ (resp.\ $T_3$).
Therefore, we have the following:

\begin{obs} \label{sufficient}
    If $M$ is a Legendrian $n$-mosaic with a soil setup and $M'$ is obtained by performing any number of fish and Kraken moves on $M$, then we can perform a Kraken move on any empty barn tile in $M'$ that lies directly above or beneath a nonempty barn tile. Similarly, we can perform a fish move on any pair of adjacent empty tiles in $M'$ that lie directly above or beneath a nonempty barn tile.
\end{obs}

\begin{figure}
    \centering
    \includegraphics[width=0.95\linewidth]{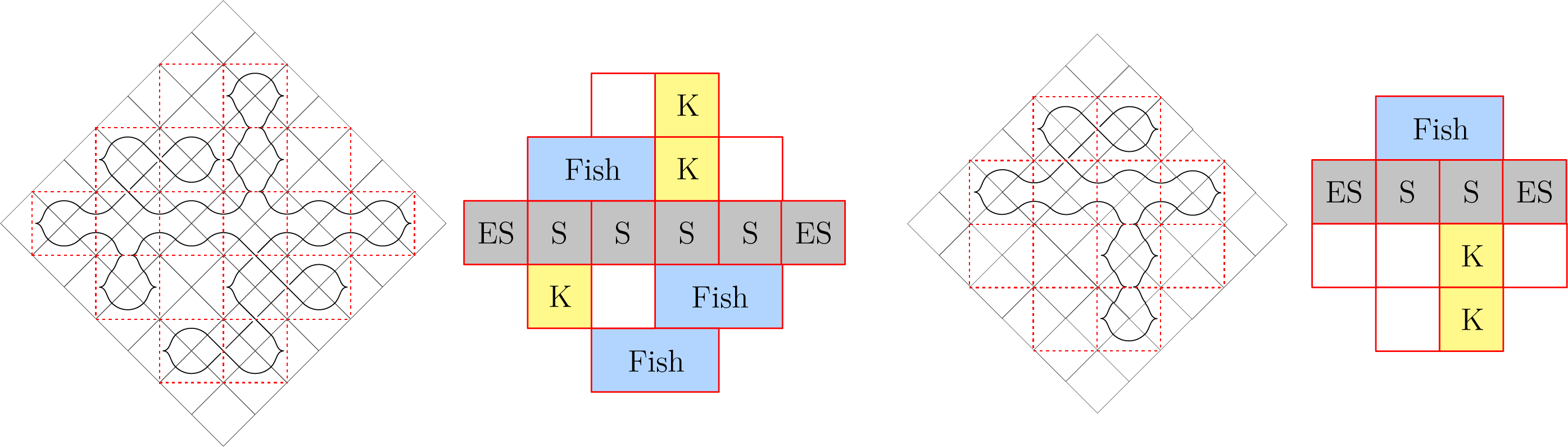}
    \caption{Two Legendrian $n$-mosaics, each created by performing Kraken and fish moves on an a mosaic having a soil setup. Note that these mosaics depict stabilized Legendrian unknots.}
    \label{fig:barn-setup}
\end{figure}

In other words, we can ``stack" fish and Kraken moves on vertically adjacent barn tiles.
Figures \ref{fig:barn-setup}, \ref{fig:barn-proof-ex}, and \ref{fig:unknot29} give examples of Legendrian $n$-mosaics produced by performing fish and Kraken moves on a soil setup in this fashion.

This observation also implies a necessary and sufficient condition for $f$ fish moves and $k$ Kraken moves to be possible on a mosaic having a soil setup:

\begin{lemma} \label{lem:enough-tiles}
    Let $M$ be a Legendrian $n$-mosaic having a soil setup. Then it is possible to perform $f$ fish moves and $k$ Kraken moves on $M$ if and only if
    \[
    2f+k\leq N,
    \]
    where
    \[
    N:=\begin{cases}
        \frac{(n-1)(n-3)}{2}, & 2\nmid n\\
        \frac{(n-2)^2}{2}, & 2\mid n.
    \end{cases}
    \]
\end{lemma}

\begin{proof}
    Since $M$ has a soil setup, $N$ is the number of non-soil barn tiles (equivalently, the number of empty barn tiles) appearing in $M$. Since fish moves fill two adjacent empty barn tiles and Kraken moves fill one empty barn tile, the necessity of the condition $2f+k\leq N$ is clear.
    
    Sufficiency follows from Observation \ref{sufficient} and the fact that every row of barn tiles in $M$ contains an even number of barn tiles. Note that if we label the barn tiles in each row of barn tiles as $B_1,\dots,B_m$ going left to right, then we perform fish moves only on pairs $B_i,B_{i+1}$ where $i$ is odd.
\end{proof}

%To facilitate our proofs, we put an additional stipulation on fish moves.
%Given a mosaic having the default barn configuration, label the barn tiles in each row as $B_1,\dots,B_m$, going from left to right. We stipulate that if one applies a fish move to any two adjacent barn tiles, then those tiles must be labeled $B_i$ and $B_{i+1}$ for some odd integer $i$. Figures \ref{fig:barn-setup}, \ref{fig:barn-proof-ex}, and \ref{fig:unknot29} give examples of fish moves performed under this stipulation.

Moreover, the Kraken move introduces two upward-oriented cusps and two downward-oriented cusps, while the fish move introduces two cusps of the same orientation, one positive trivial crossing, and one negative trivial crossing.  Figure \ref{fig:barn-effects} shows the effects of Kraken and fish moves on the Thurston-Bennequin number and rotation number. Furthermore, we have the following:
\begin{lemma} \label{barn-lemma}
    Let $M$ be a Legendrian $n$-mosaic having a soil setup. Let $M'$ be a mosaic obtained by performing $k$ Kraken moves and $f$ fish moves on $M$. Then $M'$ depicts a Legendrian unknot $\Lambda_U$ whose Thurston-Bennequin number is
    \[
    \tb(\Lambda_U) = -2k-f-|M|_{T_{10}}-1.
    \]
    Moreover, if $|M|_{T_{10}}$ is odd, then
    \[
    |\rot(\Lambda_U)| = f+1.
    \]
\end{lemma}

\begin{proof}
    All crossings introduced by Kraken and fish moves are trivial, and neither move introduces any ``loose strands." It follows from Observation \ref{obs:default-soil} that $M'$ depicts a Legendrian unknot $\Lambda_U$. Combining Observation \ref{obs:default-soil} with Figure \ref{fig:barn-effects} yields the desired Thurston-Bennequin number. Now, if $|M|_{T_{10}}$ is odd, then all cusps introduced by fish moves share the same orientation as the two cusps in the edge-soil tiles of $M'$. Therefore, combining Observation \ref{obs:default-soil} with Figure \ref{fig:barn-effects} gives us the desired expression for $|\rot(\Lambda_U)|$.
\end{proof}

\subsection{Results}

In the proof of the bound in Theorem \ref{fish-bound} and Corollary \ref{cor:fishbound} below, we present an algorithm to construct a mosaic for any Legendrian unknot. The algorithm begins with a specified soil setup. Then, it performs specified numbers of Kraken and fish moves to achieve the desired classical invariants.

Note that the expressions under the radicals in (\ref{barn-nonzero}) and (\ref{improved-fish-ladder}) are always positive since $\tb(\Lambda_U)\leq -2$ for any Legendrian unknot $\Lambda_U$ satisfying $\rot(\Lambda_U)\neq 0$. Similarly, the expression under the radical in (\ref{barn-zero}) is always positive since the maximum Thurston-Bennequin number across all Legendrian unknots is $-1$.

\begin{theorem} \label{fish-bound}
   Let $\Lambda_U$ be a Legendrian unknot. If  $\rot(\Lambda_U)\neq 0$, then
    \begin{equation}
        \label{barn-nonzero}
        m(\Lambda_U) \leq \left\lceil \sqrt{3|\rot(\Lambda_U)|-\tb(\Lambda_U)-\frac{11}{4}} + \frac{3}{2} \right\rceil.
    \end{equation}
    If  $\rot(\Lambda_U)=0$, then
    \begin{equation}
        \label{barn-zero}
        m(\Lambda_U)\leq \left\lceil \sqrt{-\tb(\Lambda_U)+\frac{5}{4}} + \frac{3}{2} \right\rceil.
    \end{equation}
\end{theorem}

\begin{proof}
    Let $r:=|\rot(\Lambda_U)|$.
    For $r\neq 0$, we will construct a Legendrian $n$-mosaic for $\Lambda_U$ with 
    \[n:=\left\lceil \sqrt{3r-\tb(\Lambda_U)-\frac{11}{4}} + \frac{3}{2} \right\rceil.\]
    To that end, let $M$ be the unique Legendrian $n$-mosaic having a soil setup with no $T_7$ tiles. In other words, all non-edge soil tiles of $M$ contain crossings, so
    \[
    |M|_{T_{10}} = \begin{cases}
        n-2,&2\nmid n\\
        n-3,&2\mid n.
    \end{cases}
    \]
    In particular, $|M|_{T_{10}}$ is odd.
    
    Now, set $f:=r-1$, and set
    \begin{equation} \label{k}
        k :=
    \begin{cases}
        -\frac{1}{2}(\tb(\Lambda_U)+r+n-2), & 2 \nmid n\\
        -\frac{1}{2}(\tb(\Lambda_U)+r+n-3), & 2 \mid n.
    \end{cases}
    \end{equation}
    Note that since $\tb(\Lambda_U)$ and $r$ have different parities, $\tb(\Lambda_U)+r$ is odd. Thus, in either case, $k\in\Z$. Furthermore,
    \[
     \tb(\Lambda_U) =
    \begin{cases}
        -2k-r-n+2, & 2 \nmid n\\
        -2k-r-n+3, & 2 \mid n.
    \end{cases}
    \]
    Now, note that the inequality $n\geq \sqrt{3r-\tb(\Lambda_U)-11/4}+3/2$ implies the inequality \[n^2-3n+\frac{9}{4}\geq 3r-\tb(\Lambda_U)-\frac{11}{4}.\] If $n$ is odd, algebraically manipulating this inequality yields
    \[
    \frac{(n-1)(n-3)}{2} \geq 2(r - 1) -\frac{1}{2}(\tb(\Lambda_U)+r+n-2) = 2f+k.
    \]
    If instead $n$ is even, algebraic manipulation  yields
    \[
    \frac{(n-2)^2}{2} \geq   2(r - 1) -\frac{1}{2}(\tb(\Lambda_U)+r+n-3) = 2f+k.
    \]
    It follows from Lemma \ref{lem:enough-tiles} that we can perform $k$ Kraken moves and $f$ fish moves on $M$.
    By Lemma \ref{barn-lemma}, the resulting mosaic depicts a stabilized Legendrian unknot $\Lambda_{U'}$.
    Figure \ref{fig:barn-proof-ex} gives an example of such a construction.

    \begin{figure}
        \centering
        \includegraphics[width=0.95\linewidth]{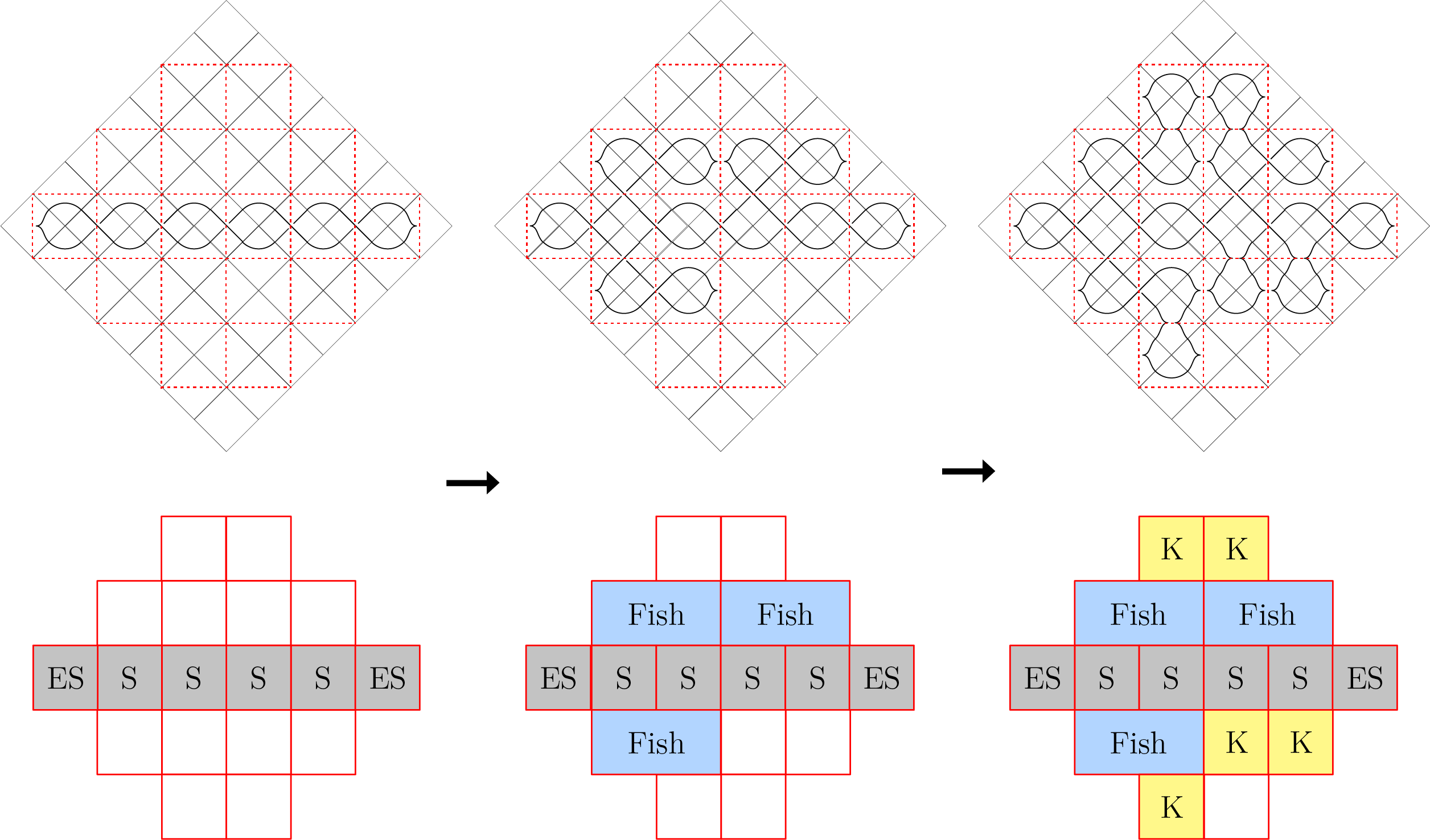}
        \caption{Using the algorithm from the proof of Theorem \ref{fish-bound} to construct a mosaic for the Legendrian unknot ${\Lambda_U}$ with $\tb({\Lambda_U})=-19$ and ${\rot({\Lambda_U})=4}$.}
        \label{fig:barn-proof-ex}
    \end{figure}
        
    We claim that $\Lambda_{U'}$ and $\Lambda_{U}$ are Legendrian isotopic. To show this, it will be enough to show that these unknots share the same Thurston-Bennequin invariant and rotation number (cf.\ \cite{eliashberg}).
    To that end, note by Lemma \ref{barn-lemma} that
    \[
    |\rot({\Lambda_{U'}})| = f+1 = r = |\rot({\Lambda_U})|,
    \]
    as desired. (This equality is sufficient since if $\rot({\Lambda_{U'}})=-\rot({\Lambda_U})$, then we can simply choose the opposite orientation of ${\Lambda_{U'}}$ to invert the sign of $\rot({\Lambda_{U'}})$.)
    Lemma \ref{barn-lemma} also implies that if $n$ is odd, then
    \begin{align*}
        \tb({\Lambda_{U'}})& -2k-f-(n-2)-1\\
        &= -2k-(r-1)-(n-2)-1\\
        &= -2k-r-n+2\\
        &=\tb({\Lambda_U}),
    \end{align*}
    as desired. If $n$ is even, a similar calculation shows that
    \[
    \tb({\Lambda_{U'}}) = -2k-r-n+3 =\tb({\Lambda_U}),
    \]
    so in either case, ${\Lambda_{U'}}$ and ${\Lambda_U}$ are Legendrian isotopic.

    If instead $r=0$, let 
    \[n:= \left\lceil \sqrt{-\tb({\Lambda_U})+\frac{5}{4}} + \frac{3}{2} \right\rceil,\]
    and use this value of $n$ to define $k$ exactly as we did in (\ref{k}).
    This time, let $M$ be a Legendrian $n$-mosaic having a soil setup with exactly one $T_7$ tile, so that
    \[
    |M|_{T_{10}}= \begin{cases}
        n-3,&2\nmid n\\
        n-4,&2\mid n.
    \end{cases}
    \]
    In particular, $|M|_{T_{10}}$ is even. It follows from Observation \ref{obs:default-soil} that the Legendrian unknot depicted in $M$ has rotation number 0.

    \begin{figure}
        \centering
        \includegraphics[width=0.75\linewidth]{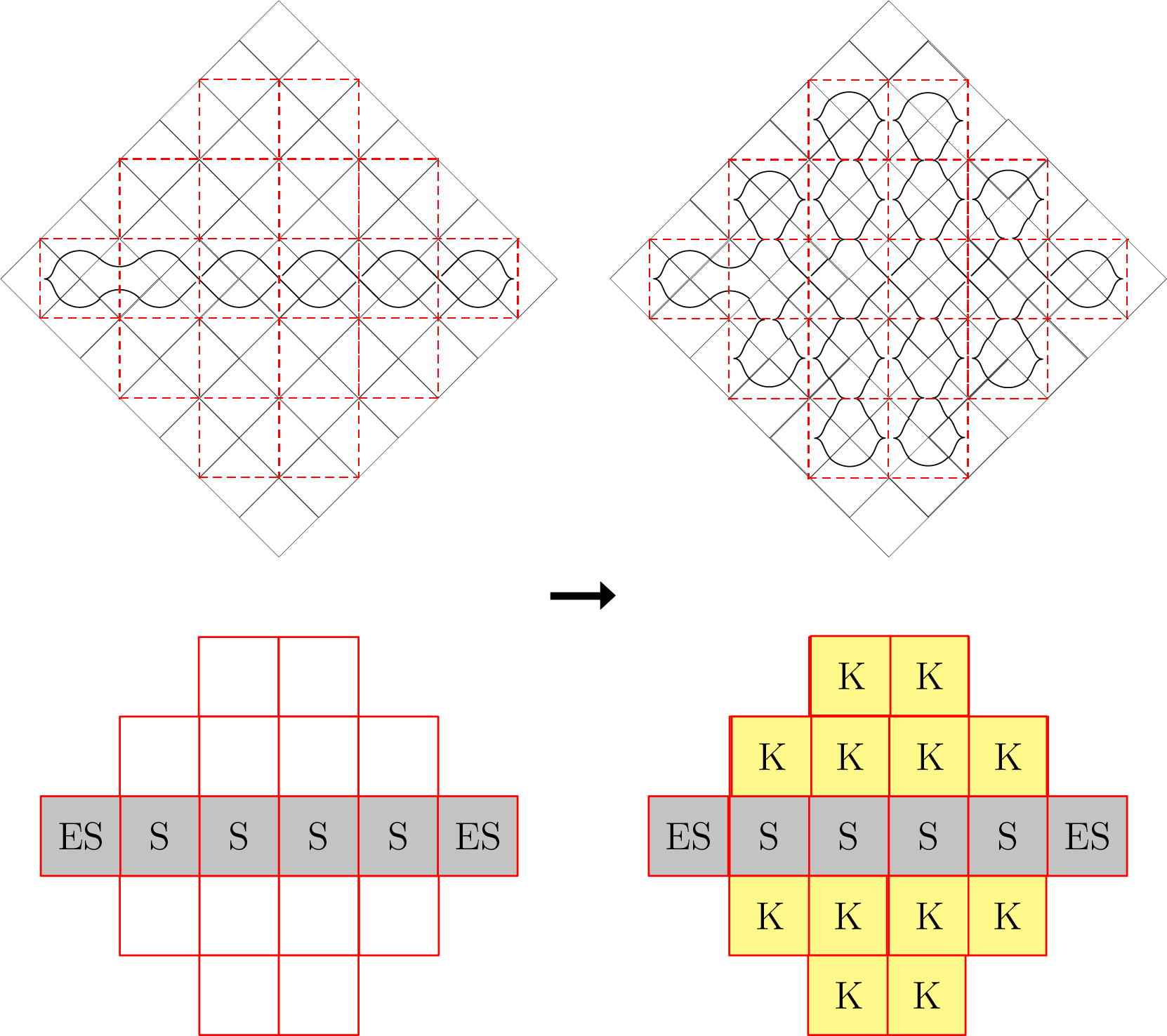}
        \caption{Using the algorithm from the proof of Theorem \ref{fish-bound} to construct a mosaic for the Legendrian unknot ${\Lambda_U}$ with $\tb({\Lambda_U})=-29$ and ${\rot({\Lambda_U})=0}$.}
        \label{fig:unknot29}
    \end{figure}
    
    Similarly to before, note that the inequality $n\geq \sqrt{-\tb({\Lambda_U})+5/4}+3/2$ implies the inequality \[n^2-3n+\frac{9}{4}\geq -\tb({\Lambda_U})+\frac{5}{4}.\] If $n$ is odd, algebraically manipulating this inequality yields
    \[
    \frac{(n-1)(n-3)}{2} \geq -\frac{1}{2}(\tb({\Lambda_U})+n-2) = k.
    \]
    If instead $n$ is even, algebraic manipulation yields
    \[
    \frac{(n-2)^2}{2} \geq  -\frac{1}{2}(\tb({\Lambda_U})+n-3) = k.
    \]
    It follows from Lemma \ref{lem:enough-tiles} that we can perform $k$ Kraken moves on $M$. (We will not perform any fish moves on $M$.) The resulting mosaic depicts a Legendrian unknot $\Lambda_{U'}$. Figure \ref{fig:unknot29} gives an example of such a construction.

    Once again, it will be enough to show that the classical invariants of $\Lambda_{U'}$ equal those of $\Lambda_{U}$.
    By Figure \ref{fig:barn-effects}, the fact that we did not perform any fish moves on $M$ implies that 
    \[
    \rot({\Lambda_U'})=0=\rot({\Lambda_U}),
    \]
    as desired.
    To compute the Thurston-Bennequin number, we once again appeal to Lemma \ref{barn-lemma}. If $n$ is odd, then
    \[
    \tb({\Lambda_{U'}}) = -2k-(n-3)-1 = -2k-r-n+2= \tb({\Lambda_U}),
    \]
    as desired. If $n$ is even, then a similar calculation yields
    \[
    \tb({\Lambda_{U'}}) = -2k-r-n+3 = \tb({\Lambda_U}),
    \]
    so in either case, ${\Lambda_U'}$ is Legendrian isotopic to ${\Lambda_U}$. This completes the proof.
\end{proof}

In particular, we have the following for Legendrian unknots on the ``outermost boundary" of their mountain range:

\begin{cor}\label{cor:fishbound}
    Let $\Lambda_U$ be a Legendrian unknot with $\rot(\Lambda_U)\neq 0$, and suppose $\tb(\Lambda_U)=-|\rot(\Lambda_U)|-1$ (i.e., $\rot(\Lambda_U)$ is either minimal or maximal among the rotation numbers of all other Legendrian unknots sharing the same Thurston-Bennequin invariant as $\Lambda_U$). Then
    \begin{equation} \label{improved-fish-ladder}
        m(\Lambda_U)\leq \left\lceil \sqrt{-4\tb(\Lambda_U)-\frac{23}{4}} + \frac{3}{2} \right\rceil = \left\lceil \sqrt{4|\rot(\Lambda_U)|-\frac{7}{4}} + \frac{3}{2} \right\rceil.
    \end{equation}
\end{cor}

\section{The Number of Legendrian Link Mosaics} \label{sec:counting}
In Theorem 1 of \cite{mosaic-count}, Oh, Hong, Lee, and Lee give an algorithm to count the number, which they call $D^{(m,n)}$, of $m\times n$ classical link mosaics. In this section, we adapt their result to the Legendrian setting. Let $\mathbb{L}^{(m,n)}$ denote the set of $m\times n$ Legendrian link mosaics (i.e., suitably connected mosaics using only the mosaic tiles in Figure \ref{fig:Legtiles}), and let $D_L^{(m,n)}$ denote the total number of elements in $\mathbb{L}^{(m,n)}$. 
We will prove the following algorithm to compute $D_L^{(m,n)}$ for any $m,n\in\Z^+$. Here, $||M||$ denotes the sum of all entries of a matrix $M$.

\begin{theorem} \label{thm:counting}
    Let $m,n\in\Z^+$. If $m=1$ or $n=1$, then the total number $D_L^{(m,n)}$ of all $m\times n $ Legendrian link mosaics is 1. Otherwise,
    \[
    D_L^{(m,n)} = 2||(X_{m-2}+O_{m-2})^{n-2}||,
    \]
    where $X_{m-2}$ and $O_{m-2}$ are $2^{m-2}\times 2^{m-2}$ matrices defined recursively by
    \[
    X_{k+1}:=\begin{bmatrix}
        X_k & O_k \\ O_k & X_k
    \end{bmatrix} \text{ and }\; O_{k+1}:=\begin{bmatrix}
        O_k & X_k \\ X_k & 3O_k
    \end{bmatrix}
    \]
    for $k=0,1,\dots,m-3$, with $1\times 1$ matrices $X_0,O_0:=\begin{bmatrix}
        1
    \end{bmatrix}$.
    For $n=2$, by $(X_{m-2}+O_{m-2})^{0}$ we mean the $2^{m-2}\times 2^{m-2}$ identity matrix. 
\end{theorem}

\begin{proof}
    If $m=1$ or $n=1$, then $D_L^{(m,n)}=1$ since the only way for a single row or column of Legendrian mosaic tiles to be suitably connected is for all of its tiles to be $T_0$. So, suppose $m,n\geq 2$. Let $\mathbb{K}_{T_9}^{(m,n)}$ be the set of $m\times n$ classical link mosaics that do not contain any $T_9$ tiles, and let $D^{(m,n)}_{T_9}$ be the total number of elements of $\mathbb{K}_{T_9}^{(m,n)}$.
    We claim that $D_L^{(m,n)}=D^{(m,n)}_{T_9}$.
    Indeed, by the construction of the Legendrian mosaic tiles in \cite{pezzimenti}, rotating counterclockwise by $45^\circ$ and replacing vertical tangencies with cusps gives a bijection from $\mathbb{K}_{T_9}^{(m,n)}$ to $\mathbb{L}^{(m,n)}$.  
    
    So, it suffices to compute $D^{(m,n)}_{T_9}$. We first refer the reader to the proof of Theorem 1 in \cite{mosaic-count} that computes $D^{(m,n)}$.
    In that proof, the authors assign \emph{states} to each suitably connected classical $m\times n$ mosaic that record the number and locations of connection points on the edges of each mosaic tile. Their algorithm for computing $D^{(m,n)}$ relies on recurrence relations of \emph{state matrices}, which are block matrices whose entries count the number of suitably connected mosaics of a certain size and with a given state. Due to this proof's length and complexity, we will not reproduce every detail here.
    
    Instead, we adjust the proof slightly by removing $T_9$ from the set of allowable tiles. In particular, $T_7$, $T_8$, and $T_{10}$ are the only usable tiles with at least three connection points. Therefore, the ``four" in the statement of the \emph{choice rule} stated in Section 2 of \cite{mosaic-count} changes to ``three." Similarly, these three tiles are now the only tiles counted by the $(2,2)$-entry of the state matrix $O_1$ defined in this section. Thus, the state matrices defined in this section are now
    \[
    X_{1}=\begin{bmatrix}
        1 & 1 \\ 1 & 1
    \end{bmatrix},\; O_{1}=\begin{bmatrix}
        1 & 1 \\ 1 & 3
    \end{bmatrix},\;  \text{ and }\; N^{(1,1)}=X_1 + O_1=\begin{bmatrix}
        2 & 2 \\ 2 & 4 \end{bmatrix}.
    \]
    In particular, the bottom-right submatrix of $O_{k+1}$ in the inductive step of Proposition 2 in \cite{mosaic-count} becomes $3O_k$ rather than $4O_k$. Since these are the only points in the proof where $T_9$ tiles have any bearing, the rest of the  proof in \cite{mosaic-count} shows that $D^{(m,n)}_{T_9}$ is given by the formula in the statement of Theorem \ref{thm:counting}.
\end{proof}

\begin{figure}
    \centering
    \begin{tikzpicture}
        \begin{axis}[
          xlabel=$n$,
          ylabel=$\ln D_L^{(n,n)}$]
        \addplot table [y=a, x=n]{Appendices/mosaics.dat};
        \node[] at (axis cs: 3.5,90) {$\ln D_L^{(n,n)}\approx 1.0745n^2$};
        \node[] at (axis cs: 6.5,80) {$-3.1057n+2.3933$};
        \node[] at (axis cs: 3.25,72) {$R^2=1$};
        \end{axis}
    \end{tikzpicture}
    \caption{Quadratic exponential growth of the number $D_L^{(n,n)}$ of suitably connected Legendrian $n$-mosaics.}
    \label{app-fig:b1}
\end{figure}
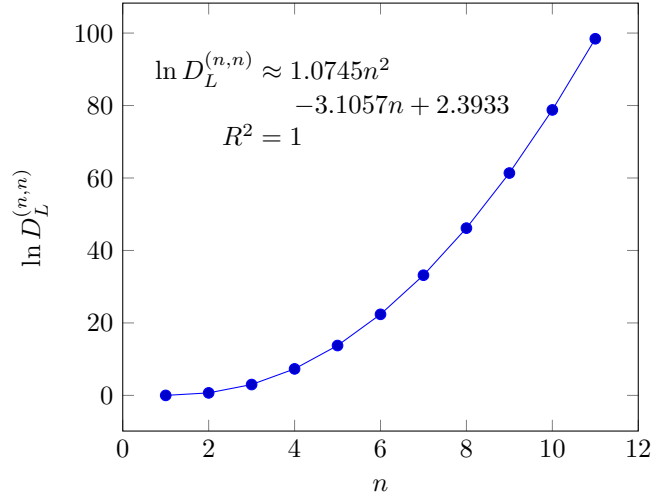

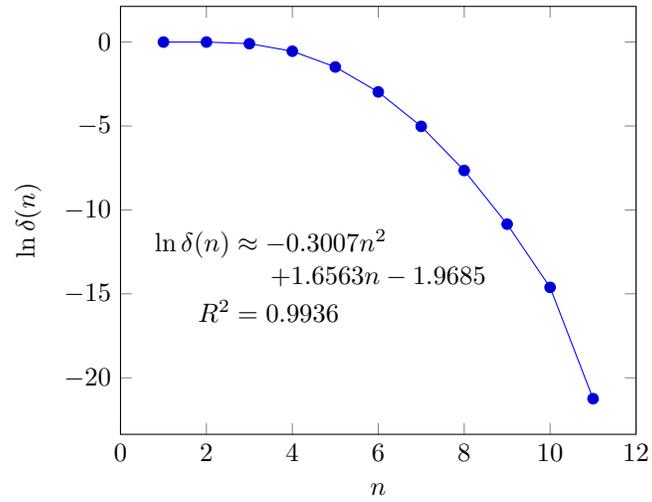
\begin{figure}
    \centering
    \begin{tikzpicture}
        \begin{axis}[
          xlabel=$n$,
          ylabel=$\ln\delta(n)$]
        \addplot table [y=b, x=n]{Appendices/ratios.dat};
        \node[] at (axis cs: 3.5,-12) {$\ln\delta(n)\approx -0.3007n^2$};
        \node[] at (axis cs: 6,-14) {$+1.6563n-1.9685$};
        \node[] at (axis cs: 3.45,-16) {$R^2=0.9936$};
        \end{axis}
    \end{tikzpicture}
    \caption{Negative quadratic exponential growth of the ratios $\delta(n)$ between the number $D_L^{(n,n)}$ of suitably connected Legendrian $n$-mosaics and the number $D^{(n,n)}$ of suitably connected classical $n$-mosaics.}
    \label{app-fig:b2}
\end{figure}

In particular, $D_L^{(m,2)}=D^{(m,2)}$ for all $m\in\Z^+$ since $T_9$ tiles do not appear in classical link mosaics with only two rows or columns. As noted in \cite{mosaic-count}, this quantity is precisely $2^{m-1}$.

By implementing the algorithm of Theorem \ref{thm:counting} in Mathematica, we were able to compute $D_L^{(m,n)}$ for all $1\leq n\leq m\leq 11$. We present our code and data in Appendix \ref{app:counting}.

In \cite{mosaic-count}, the authors computed $D^{(n,n)}$ for all $n\leq 13$. This allows us to consider the ratios
\[
\delta(n):= \frac{D_L^{(n,n)}}{D^{(n,n)}}
\]
for all $n\leq 11$. Similarly to $D^{(n,n)}$, it appears that $D_L^{(n,n)}$ grows at a quadratic exponential rate. Moreover, $\delta(n)$ seems to monotonically converge to 0 at a quadratic exponential rate. We provide graphs and best-fit models of $\ln D_L^{(n,n)}$ and $\ln \delta(n)$ in Figures \ref{app-fig:b1} and \ref{app-fig:b2}, respectively.

\section{Updated Censuses} \label{algos}

%When investigating Legendrian knot mosaics, it is often helpful to study the mosaic numbers of different Legendrian representations of the same smooth knot.
In \cite{pezzimenti}, Pezzimenti and Pandey give a preliminary census containing computed mosaic numbers of certain Legendrian knots, organized in their mountain ranges, as introduced in Section \ref{sec:basics}. Using an exhaustive computer search, we provide an updated census, which covers a much larger class of Legendrian knots. The updated mountain ranges are give in Appendix \ref{census-list}.
%This information can be represented in the form of a census, first introduced by Pezzimenti and Pandey in \cite{pezzimenti}. 
%A census can be thought of as an extension of the mountain ranges introduced in Section \ref{sec:basics},

While the exact mosaic number of a Legendrian knot is generally hard to determine, precise bounds can be obtained via an exhaustive search. Suppose $L_n$ is the set containing every Legendrian knot realizable on a Legendrian $n$-mosaic. If $\Lambda \in L_n$, then $m(\Lambda) \leq n$, and, likewise, if $\Lambda \notin L_n$, then $m(\Lambda) > n$. These lists would be impractical to produce by hand, but the discrete nature of this problem makes it a natural fit for a computer search. In this section, we detail such a computer search and summarize its results, using the mountain ranges provided in \cite{atlas} as a base for our censuses.

\subsection{Computer Search}\label{sec:search_details}

Our search was conducted using two programs: \texttt{mosaic-gen.rs}, a \texttt{rust} program that produces a list of all suitably connected Legendrian $n$-mosaics for a given $n$, and \texttt{mosaic-cat.py}, a \texttt{python} program that catalogs those mosaics by $\tb$, $\rot$, and smooth knot type. Full implementations and documentation for these programs can be found in \cite{code}. 

When working digitally, it is convenient to represent mosaics as integers in base-10, obtained by reading a mosaic's tile numbers from left to right and top to bottom, starting from the left corner as shown in in Figure \ref{fig:number_representation}. Here, we use 9 instead of 10 to represent the crossing tile $T_{10}$ for ease of notation, as $T_9$ is unused in Legendrian mosaics.
\begin{figure}[H]
    \centering
    \begin{equation*}
        \vcenter{\hbox{\includegraphics[scale=0.12]{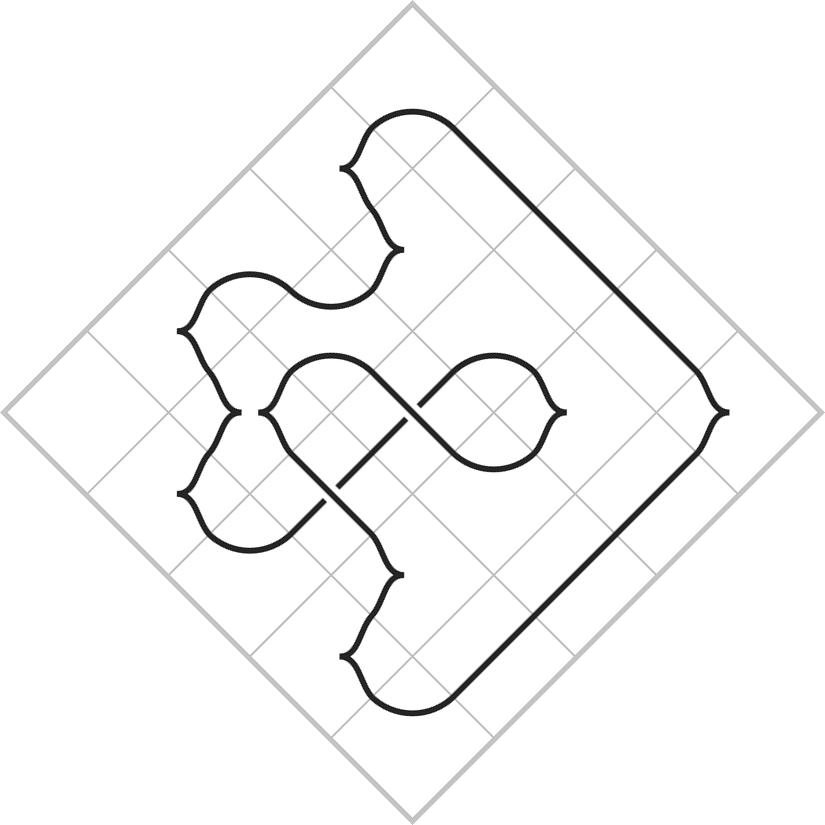}}} \Leftrightarrow 
        \begin{bmatrix} 
        0 & 2 & 1 & 2 & 1 \\
        2 & 8 & 7 & 4 & 6 \\
        3 & 10 & 10 & 1 & 6 \\
        2 & 4 & 3 & 4 & 6 \\
        3 & 5 & 5 & 5 & 4
        \end{bmatrix} 
        \Leftrightarrow 0212128746399162434635554
    \end{equation*}
    \caption{Representing a Legendrian mosaic as an integer.}
    \label{fig:number_representation}
\end{figure}
Using this representation, a list of all suitably connected $n$-mosaics can be obtained by simply iterating through all positive integers in base-10 of maximum length $n^2$, and noting those mosaics that are suitably connected. This process is handled by \texttt{mosiac-gen.rs}, which produces lists matching the numbers of Legendrian link mosaics presented in Section \ref{sec:counting} and Appendix \ref{app:counting}. 

From this list, \texttt{mosaic-cat.py} determines which mosaics correspond to knots (rather than links). If so, the classical invariants of the associated knot are noted, and the knot is encoded as an extended Gauss code to be used with \texttt{SageMath}'s \texttt{Links} package \cite{sagemath}. This package is used to calculate the HOMFLY-PT polynomial of the resulting knots, which, for knots up to 8 crossings, uniquely determines the smooth knot type \cite{homfly}. This was a valid assumption for the mosaic sizes that we cataloged. Our mosaics are then organized into a list according to their smooth knot type, Thurston-Bennequin number, and rotation number.

\subsection{Results}\label{sec:census_results}
Using the programs described in Section \ref{sec:search_details}, we produced lists of all Legendrian knots, distinct up to smooth knot type and classical invariants, realizable on mosaics through size 6. The results from these lists regarding the knots (both smooth and Legendrian) realizable on a given mosaic size are summarized in Tables \ref{tab:bounds_by_size}, \ref{tab:knots_by_size}, and \ref{tab:mosaics_by_size}.

\begin{table}
    \centering
    \begin{tabular}{|c|c|c|c|}
    \hline
    Mosaic Size & Maximum $\tb$ & Minimum $\tb$ & Maximum $|\rot|$ \\
    \hline
    $n=2$ & -1 & -1 & 0 \\
    $n=3$ & -1 & -3 & 2 \\
    $n=4$ & -1 & -7 & 2 \\
    $n=5$ & 1 & -12 & 3 \\
    $n=6$ & 3 & -21 & 6 \\
    \hline
    \end{tabular}
    \caption{Computationally verified bounds on $\tb$ and $\rot$.}
    \label{tab:bounds_by_size}
\end{table}

\begin{table}
    \centering
    \begin{tabular}{|c|p{8cm}|}
    \hline
    Mosaic Size & Realizable Smooth Knots \\
    \hline
    $n=2$ & $0_1$ \\
    $n=3$ & $0_1$ \\
    $n=4$ & $0_1$ \\
    $n=5$ & $0_1, 3_1, m(3_1)$ \\
    $n=6$ & $0_1, 3_1, m(3_1), 3_1\#3_1, m(3_1)\#m(3_1), 4_1, 5_1, m(5_1), 5_2,$ $m(5_2), 6_1, m(6_1), 7_1, m(7_1), 7_2, m(7_2), 8_1, m(8_1)$ \\
    \hline
    \end{tabular}
    \caption{Smooth knot types realizable on a Legendrian $n$-mosaic.}
    \label{tab:knots_by_size}
\end{table}

\begin{table}
    \centering
    \begin{tabular}{|c|c|c|}
    \hline
    Mosaic Size & Mosaics corresponding to knots & Distinct Legendrian knots\\
    \hline
    $n=2$ & 1 & 1 \\
    $n=3$ & 17 & 4 \\
    $n=4$ & 793 & 9 \\
    $n=5$ & 275557 & 40 \\
    $n=6$ & 831699599 & 328 \\
    \hline
    \end{tabular}
    \caption{Number of distinct suitably connected square Legendrian $n$-mosaics corresponding to knots (rather than links), and number of distinct Legendrian knots within these mosaics.}
    \label{tab:mosaics_by_size}
\end{table}

The bounds from these lists were also used to create updated censuses for every smooth knot type found, with one caveat. While the lower bound indicated in a census applies to every Legendrian representative with the given $\tb$ and $\rot$, upper bounds only indicate that at least one Legendrian representative with those classical invariants is realizable on a mosaic of that size. For certain smooth knot types, for instance $5_1$, multiple non-isotopic Legendrian representatives may have the same classical invariants, but the upper bound may not apply to every representation. More information on these smooth knot types can be found in \cite{atlas}.

The full updated censuses are available in Appendix \ref{census-list}, and minimal mosaics for the knots presented in these censuses can be found in Appendix \ref{min-mosaics}. 

The lower bounds provided by this search answered many of the questions presented by Pezzimenti and Pandey in \cite{pezzimenti} and reintroduced in Section \ref{sec:initial_questions}, particularly those regarding how stabilization affects mosaic number:

\begin{obs}\label{obs:many_stabilizations}
    There exist (several) Legendrian knots for which stabilization reduces mosaic number.
\end{obs}

 Examples of occurrences of Legendrian knots featuring this property are highlighted in yellow in the mountain ranges of Appendix \ref{census-list}.
One of the simplest examples of a mosaic-number-reducing stabilization, and the only example of a knot with mosaic number 5 being stabilized to a knot with mosaic number 4, can be found in the census for the unknot (Figure \ref{fig:new-unknots}) and is depicted in Figure \ref{fig:unknot_stabilization}.

\begin{figure}
    \begin{equation*}
        \vcenter{\hbox{\includegraphics[scale=0.168]{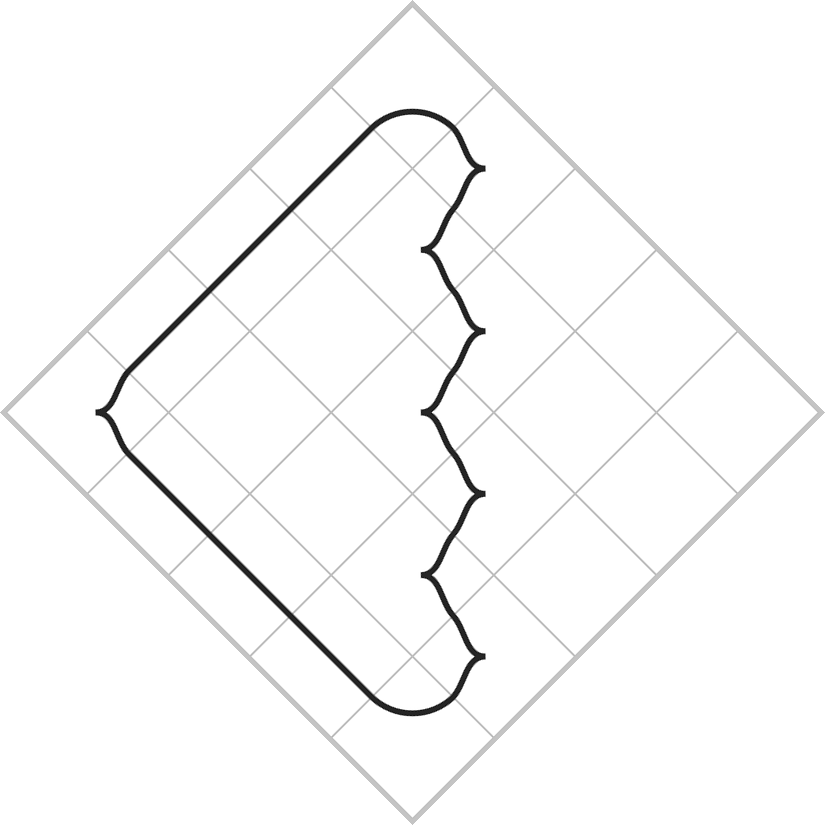}}}
        ~\xrightarrow{~S_-~}~        
        \vcenter{\hbox{\includegraphics[scale=0.21]{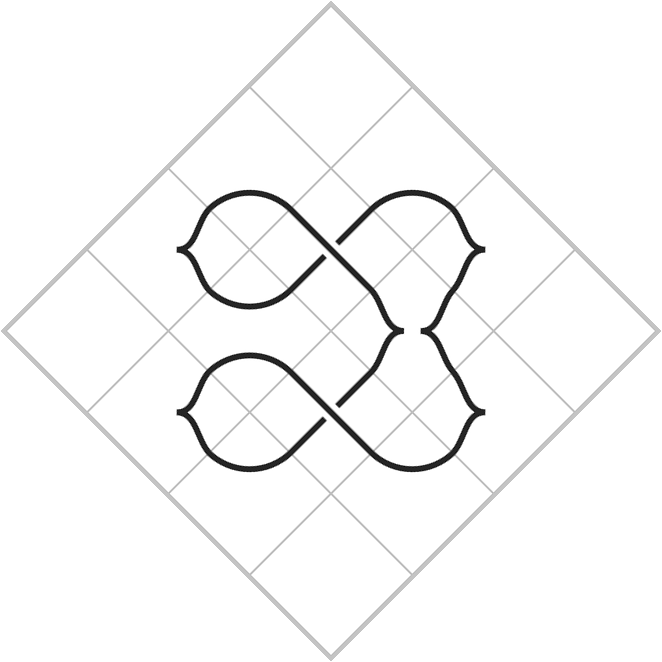}}}
    \end{equation*}
    \caption{Stabilization reducing mosaic number in the unknot.}
    \label{fig:unknot_stabilization}
\end{figure}

The crab buckets introduced in Section \ref{combo_bounds} are inspired by an unusual example from the negative trefoil census (Figure \ref{fig:trefoils}). While most stabilizations reduce the magnitude of the a knot's rotation number and occur towards the edges of the census, the crab bucket $\beta_5$ is the result of a stabilization that increases the magnitude of the rotation number, depicted in Figure \ref{fig:crab_stabilization}. In Figure \ref{fig:trefoils}, $\beta_5$ is the knot with $\tb(\beta_5)=-12$ and $\rot(\beta_5) = 1$, near the middle of the census. 

\begin{figure}
    \begin{equation*}
        \vcenter{\hbox{\includegraphics[scale=0.14]{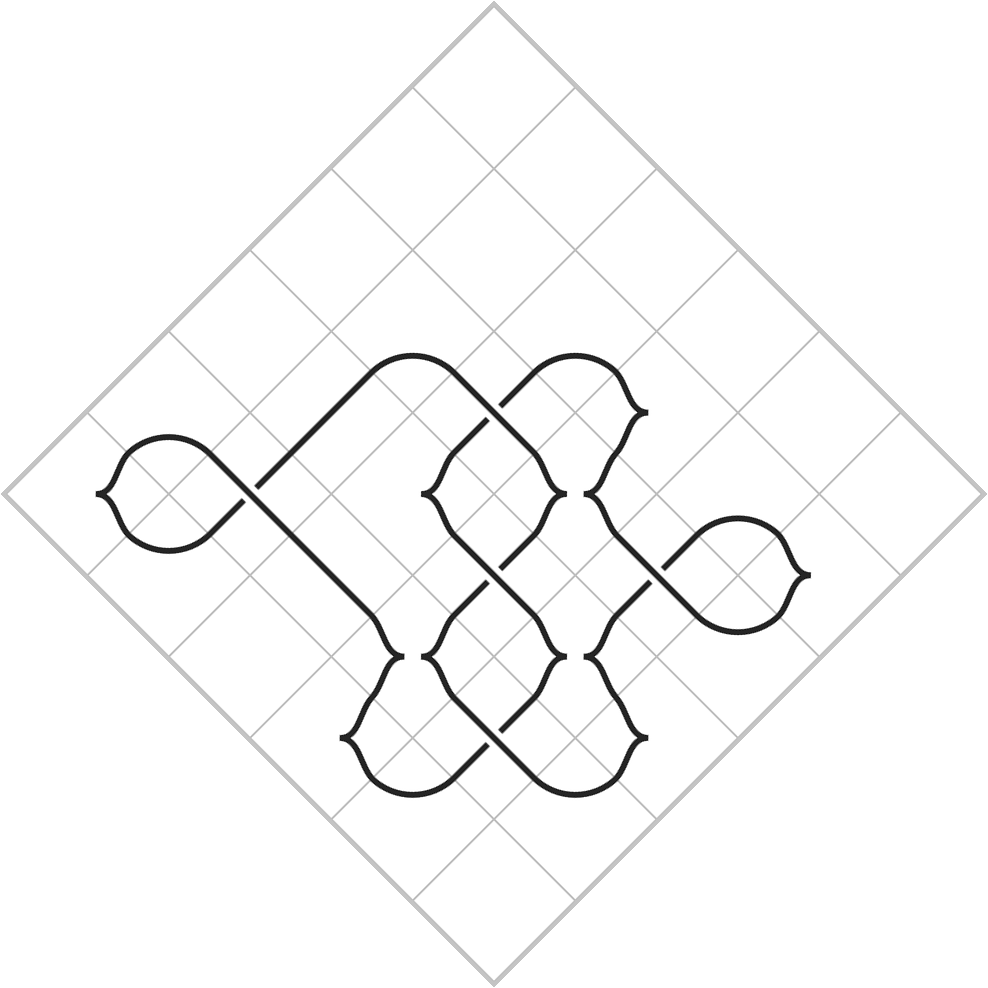}}}
        ~\xrightarrow{~S_+~}~        
        \vcenter{\hbox{\includegraphics[scale=0.168]{crab5.png}}}
    \end{equation*}
    \caption{Stabilization producing $\beta_5$.}
    \label{fig:crab_stabilization}
\end{figure}
    
\begin{obs}\label{obs:non_maximal_tb}
    There exists a smooth knot $K$ for which $m_L(K)$ is only realized by a Legendrian representative with non-maximal Thurston-Bennequin number, namely $8_1$. 
\end{obs}
In other words, there exists a smooth knot whose Legendrian mosaic number is not obtained by a ``peak'' on its mountain range. In fact, the Legendrian knot at the peak of the mountain range for $8_1$ requires two stabilizations, one positive and one negative, to realize  Legendrian mosaic number of $8_1$, as shown in Figure \ref{fig:8_1stabilization}. 
Our search concluded that $m_L(8_1)=6$, while $m(\Lambda)=7$, where $\Lambda$ is the Legendrian representative of $8_1$ with maximal Thurston-Bennequin invariant.  

This can also be seen in the censuses for $5_1$ and $7_1$ (Figures \ref{fig:5_1} and \ref{fig:7_1}), both of which contain multiple peaks in their mountain ranges. 
%Here, there exist  representatives of both knots with maximal Turston-Bennequin invariant that do not realize their respective Legendrian mosaic numbers (although, other peaks do). 
Both knots  have at least one Legendrian representative of with maximal Thurston-Bennequin invariant that does not realize its Legendrian mosaic number (although, other peaks do). 
%In fact, $7_1$ has two such representatives. 

\begin{figure}
    \centering
        \begin{equation*}
        \vcenter{\hbox{\includegraphics[scale=0.12]{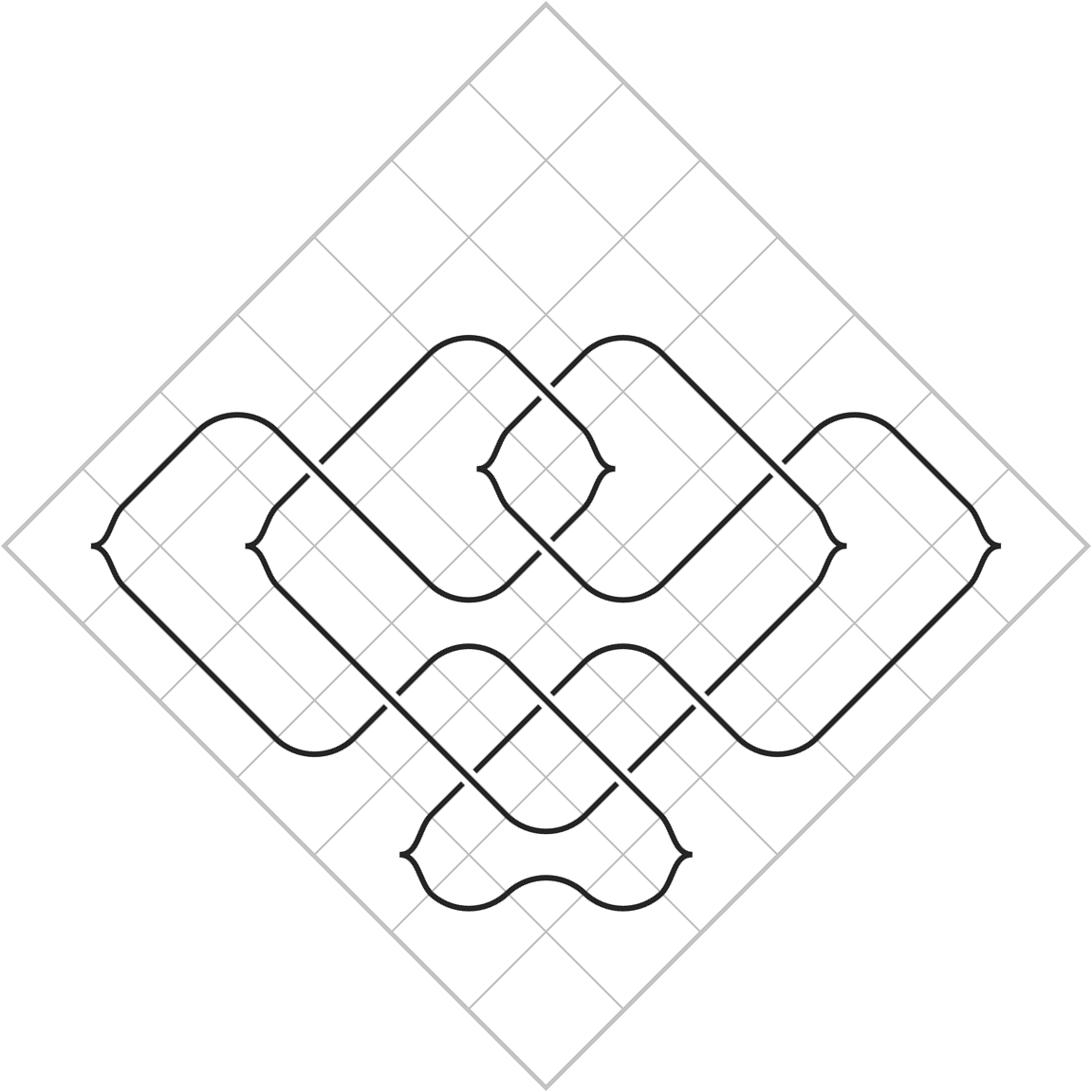}}}
        ~\xrightarrow{~S_+~\circ~S_-~}~        
        \vcenter{\hbox{\includegraphics[scale=0.14]{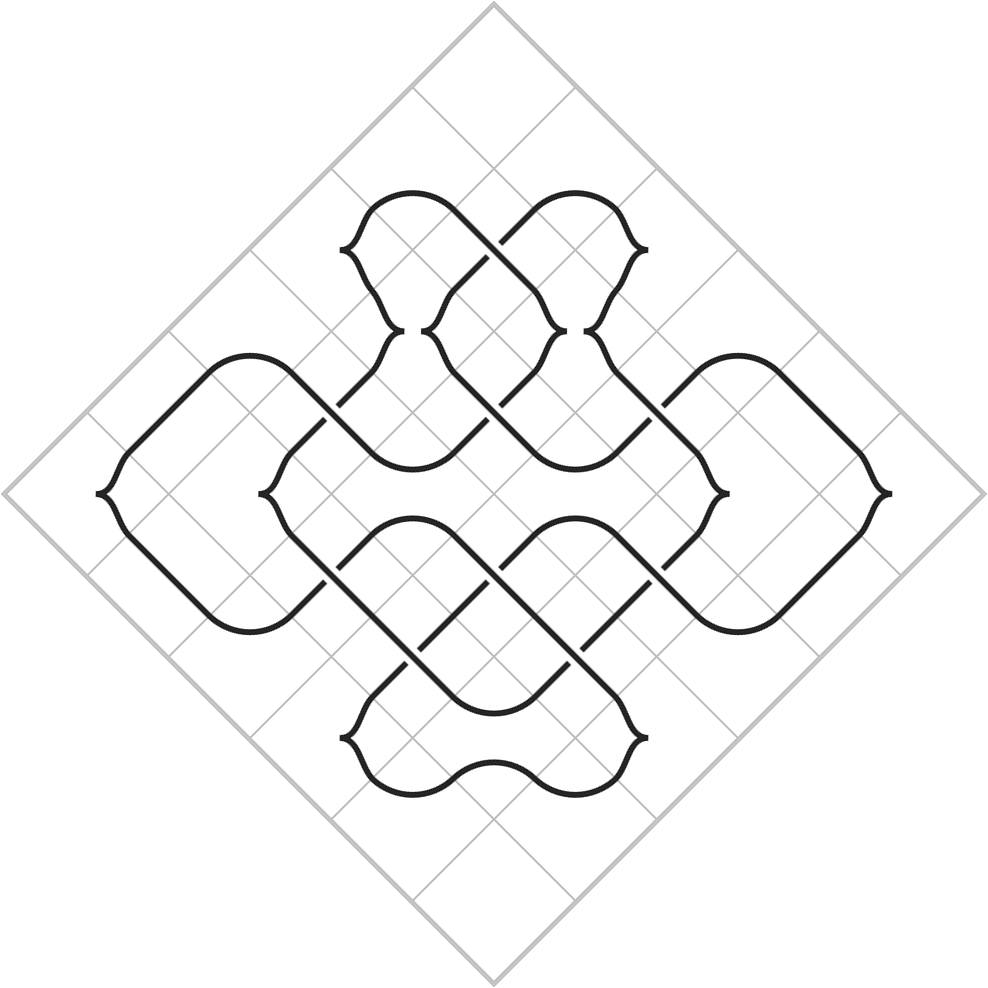}}}
        \end{equation*}
    \caption{Maximum tb and minimum mosaic number representations of $8_1$.}
    \label{fig:8_1stabilization}
\end{figure}

\section{Further Questions}\label{questions}
We conclude by suggesting some additional related areas to explore.
Since our computer search showed that the Legendrian mosaic number of $8_1$ is not attained by any Legendrian representative with maximal Thurston-Bennequin invariant, it is natural to ask whether infinitely many such knot types exist:
\begin{question}
    Are there infinitely many smooth knot types whose Legendrian mosaic number is only attained by stabilized Legendrian representatives?
\end{question}
Similarly, our computer search produced examples of stabilization decreasing the mosaic number. This leads us to ask the (weaker) question of whether infinitely many such examples exist:
\begin{question}
    Are there infinitely many Legendrian knots for which stabilization decreases the mosaic number?
\end{question}
Our original motivation behind the crab bucket construction at the end of Section \ref{combo_bounds} was to answer this question by providing an infinite family of Legendrian knots for which stabilization decreases the mosaic number.
The name of our construction alludes to the so-called ``crabs in a bucket effect." For all odd $n\geq 5$, any attempt to stabilize a cusp of $\beta_n$ seems to be prevented by its neighboring cusps and crossings, like crabs in a bucket pulling each other back in whenever one attempts to escape:
\begin{conjecture}
    For all odd $n\geq 5$, fix an orientation of the crab bucket $\beta_{n}$ so that $\rot(\beta_n)=1$. Then $\beta_n$ cannot be negatively destabilized without increasing its mosaic number.
\end{conjecture}
As discussed in Section \ref{algos}, this holds when $n=5$.
Since the crossing number and classification of connected sums of Legendrian torus knots are well-understood (see, for example, \cite{torus-additivity} or \cite{connect-sum}), answering the following question may aid in proving our conjecture:

\begin{question}
    Can we improve upon the bounds in Section \ref{bounds} by considering the crossing number or other invariants of Legendrian knots?
\end{question}
Of course, this question is intriguing in its own right.
In particular, it would be interesting to incorporate the crossing number into the linear algebraic approach discussed in Section \ref{linear}.

Furthermore, our upper bound construction in Section \ref{unknot-section} is limited only to Legendrian unknots. It would be desirable to generalize the resulting bounds to Legendrian representatives of other smooth knot types:

\begin{question}
    Can we obtain upper bounds for the mosaic number of Legendrian nontrivial knots, by reworking the barn tile construction or otherwise?
\end{question}

One way to rework the soil setup construction for a given Legendrian nontrivial knot $\Lambda$ might be to fix some pre-stabilized copy $\Lambda'$ of $\Lambda$ in the corner of a mosaic and place soil tiles adjacent to $\Lambda'$ before performing Kraken and fish moves.
Another way might be to consider moves other than Kraken and fish moves on barn tiles.

However, this construction may become less optimal as the rotation number increases. Namely, the only way to increase the rotation number in our construction is to use fish moves, each of which fills two barn tiles. As a consequence, Theorem \ref{fish-bound} places the four Legendrian unknots with rotation number 6 and Thurston-Bennequin invariants $-7,-9,-11,$ and $-13$ on a Legendrian 7-mosaic, while our computer search produced a Legendrian 6-mosaic for these unknots. This motivates the following question:

\begin{question}
 Can our bound in Theorem \ref{fish-bound} be improved using a different construction?
\end{question}

In particular, it would be desirable to decrease the coefficient on the term $3|\rot(\Lambda_U)|$ in our current bound. One approach could be to consider moves other than fish moves that affect the rotation number.

Finally, it is reasonable to alter the computer algorithm from Section \ref{algos} to generate random Legendrian knot and link mosaics. The following question, then, extends the growing literature on random knots (see, for example, \cite{random2} or \cite{random3}) to the study of Legendrian knots and knot mosaics:

\begin{question}
    What can random Legendrian knot mosaic generation tell us about the distributions of Legendrian knots and their invariants across all Legendrian knot mosaics of a given size?
\end{question}

\section*{Acknowledgments}
We would like to thank Eugene Fiorini for encouraging us to publish some of our findings on the Online Encyclopedia of Integer Sequences.
We would also like to thank the developers of the computer game \emph{Minecraft}, in which using a custom resource pack greatly facilitated our conception of the constructions in Section \ref{unknot-section}.

This research was done at Moravian University as part of the Research Challenges of Computational Methods in Discrete Mathematics REU; it was funded by the National Science Foundation (MPS-2150299).

\bibliography{references}
\pagebreak

\appendix

\renewcommand{\thefigure}{\thesection.\arabic{figure}}
\renewcommand{\theHfigure}{\thesection.\arabic{figure}}
\setcounter{figure}{0}

\renewcommand{\thetable}{\thesection.\arabic{table}}
\renewcommand{\theHtable}{\thesection.\arabic{table}}
\setcounter{table}{0}
    
\section{Calculation in Theorem \ref{thm:lin-alg}}\label{app:lin-alg-details}

Here we provide the details of the calculation used to obtain the inequalities at the end of the proof of Theorem \ref{thm:lin-alg}. The following Mathematica code computes the intersection of
$$V=\{(x_1,x_2,x_3,x_4,x_5)\in\R^5:x_3=x_4=0\}$$
with the image of $\R^{25}_{\geq0}$ under the transformation $P:\R^{25}\to\R^5$. There are a few tricks used to make the calculation more manageable, including removing duplicate columns from the starting matrix P1, and using null space matrices to restrict the image to $V$.
\definecolor{mygray}{RGB}{240,240,240}
\lstset{
    backgroundcolor=\color{mygray}, % Set background color
    basicstyle=\ttfamily,              % Use a monospaced font
    breaklines=true,                   % Automatic line breaking
}
\begin{lstlisting}[language=Mathematica]
P1={
    {0,0,-1/2,0,0,0,-1/2,0,0,0,-1,1,-1,-1,0,1,-1,-1},
    {0,0,1/2,0,0,0,-1/2,0,0,0,0,0,1,0,0,0,0,-1},
    {-1,-1,0,1,1,1,0,-1,0,-2,0,-2,0,0,2,2,0,0},
    {0,-1,-1,0,-1,1,1,1,0,0,-2,0,-2,0,0,0,2,2},
    {1,1,1,1,1,1,1,1,1,1,1,1,1,1,1,1,1,1}
};

E1={{0,0,1,0,0},{0,0,0,1,0}};
E2={{1,0,0,0,0},{0,1,0,0,0},{0,0,0,0,1}};

R1=Transpose[NullSpace[E1.P1]];
P2=E2.P1.R1;
R2=Transpose[NullSpace[P2]];

yVec=Transpose[{Array[y,Length[P2]]}];
vVec=Transpose[{Array[v,Length[R2[[1]]]]}];

Q=Simplify[R1.(PseudoInverse[P2].yVec+R2.vVec)];
inequalities=Thread[Transpose[Q][[1]]>=0];

image=Reduce[Exists[{v[1],v[2],v[3],v[4],v[5],v[6],v[7],v[8],v[9],v[10],v[11],v[12],v[13]},inequalities],{y[1],y[2],y[3]}];
inequalities=Simplify[image/.{y[1]->tb,y[2]->rot,y[3]->n2}]
\end{lstlisting}

The output of this program gives a condition \texttt{inequalities} which constrains $P(\R^{25}_{\geq0})\cap V$ in terms of \texttt{tb}, \texttt{rot}, and \texttt{n2}. Here, \texttt{tb} represents $\tb(\Lambda)$, \texttt{rot} represents $\rot(\Lambda)$, and \texttt{n2} represents $n^2$. For the rest of this calculation, to keep the notation from becoming unwieldy, we will write $\tb$ and $\rot$ to mean $\tb(\Lambda)$ and $\rot(\Lambda)$ respectively. In this way, \texttt{inequalities} can be written as
\begin{align*}
    &(\tb \leq 0 \wedge ((2 \rot \leq \tb \wedge 4 \rot + n^2 \geq \tb)\\
    &\vee (\tb < 2 \rot \leq -\tb \wedge \tb + n^2 \geq 0)\tag{$*$}\\
    &\vee (2 \rot + \tb > 0 \wedge n^2 \geq 4 \rot + \tb)))\\
    &\vee ((4 \rot + n^2 \geq \tb \vee 
     \rot > 0) \wedge (n^2 \geq 4 \rot + \tb \vee \rot \leq 0) \wedge \tb > 0).
\end{align*}
 Mathematica is unable to simplify this statement any further, but it does have a simpler form. We claim that $(*)$ is equivalent to the statement
\begin{equation*}
    (\tb+n^2\geq0)\wedge(n^2\geq4|\rot|+\tb),\tag{$\dagger$}
\end{equation*}
which we will show via the cases that follow. We will argue that in each case, the condition given by $(*)$ is equivalent to $(\dagger)$.
%This can be understood in the following way. The condition $(*)$ can be seen as including two cases, with the first case having three subcases. The cases are as follows.

\begin{itemize}
%    \item[(1)] $\tb\leq0$
%     \begin{itemize}
%         \item[(1.1)] $2 \rot \leq \tb$
%         \item[(1.2)] $\tb < 2 \rot \leq -\tb$
%         \item[(1.3)] $2 \rot > -\tb$
%     \end{itemize}
%     \item[(2)] $\tb > 0$
% \end{itemize}

\item[\textbf{Case (1.1)}]: $\tb\leq0$ and $2 \rot \leq \tb$.

% We will argue that in each case, the condition given by $(*)$ within that case is equivalent to $(\dagger)$. 
We must show that $4 \rot + n^2 \geq \tb$ is equivalent to $(\dagger)$. Because $2\rot\leq\tb\leq0$, we know $|\rot|=-\rot$. So $(\dagger)$ can be written as $$(\tb+n^2\geq0)\wedge(n^2\geq-4\rot+\tb),$$
which is equivalent to $$(\tb+n^2\geq0)\wedge(4\rot+n^2\geq\tb).$$
The second inequality of this statement, $4\rot+n^2\geq\tb$, is the condition given in case (1.1). So we just need to show that case (1.1) also implies that $\tb+n^2\geq0$. Combining $2\rot\leq\tb$ and $4\rot+n^2\geq\tb$ tells us that $2\tb+n^2\geq \tb$. Rearranging this gives that $\tb+n^2\geq0$. Therefore, $(*)$ is equivalent to $(\dagger)$.

\item[\textbf{Case (1.2)}]:  $\tb\leq0$ and $\tb<2\rot\leq-\tb$.

In this case, the condition given by $(*)$ is $\tb+n^2\geq0$. The statement $\tb+n^2\geq0$ is exactly the first half of $(\dagger)$, so we just need to show that we can also recover the second half, which states $n^2\geq4|\rot|+\tb$. 

First consider if $\rot\geq0$. Because $2\rot\leq-\tb$, we have $2\rot+\tb\leq0$. Multiplying by 2 gives $4\rot+2\tb\leq0$. We then have
$$\tb+n^2\geq0\geq4\rot+2\tb.$$
Since $\rot\geq0$, $|\rot|=\rot$, and thus, 
$\tb+n^2\geq4|\rot|
+2\tb$. Subtracting $\tb$ from both sides gives $n^2\geq4|\rot|+\tb$, as desired.

If $\rot<0$, we can make a similar argument using the fact that $\tb<2\rot$. We find that
$$2\tb-4\rot<0\leq\tb+n^2.$$
Subtracting $\tb$ from both sides then gives $n^2\geq-4\rot+\tb$. Since $\rot<0$, we have $|\rot|=-\rot$, and thus, $n^2\geq4|\rot|+\tb$. This completes case (1.2).

\item[\textbf{Case (1.3)}]: $\tb\leq0$ and $2\rot>-\tb$.

We must show that the statement $n^2\geq4\rot+\tb$ is equivalent to $(\dagger)$. In this case, our assumptions imply that $2\rot>-\tb\geq0$. Thus, $\rot>0$, which means $|\rot|=\rot$. So the statement $n^2\geq4\rot+\tb$ is equivalent to $n^2\geq4|\rot|+\tb$, which is just the second part of $(\dagger)$. Thus, we just need to show that in this case, $n^2\geq4\rot+\tb$ implies the first part of $(\dagger)$, which is $\tb+n^2\geq0$. Because $2\rot>-\tb$, we have $4\rot+2\tb>0$. Thus, we can subtract $4\rot+2\tb$ from the right side of $n^2\geq4\rot+\tb$ to find that $n^2\geq-\tb$. Thus, $n^2+\tb\geq0$, as desired.

\end{itemize}

In the remaining two cases, we want to show that if $\tb>0$, then the statement 
\begin{align*}
    &(4 \rot + n^2 \geq \tb \vee 
     \rot > 0) \wedge (n^2 \geq 4 \rot + \tb \vee \rot \leq 0)\tag{$\prime$}
\end{align*}
is equivalent to $(\dagger)$. Note that exactly one of $\rot>0$ and $\rot\leq0$ will be true, so $(\prime)$ can be rewritten as
\begin{align*}
    &(\rot>0\wedge n^2 \geq 4 \rot + \tb)\vee(\rot\leq0\wedge4 \rot + n^2 \geq \tb).\tag{$\prime\prime$}
\end{align*}

\begin{itemize}

\item[\textbf{Case (2.1)}]: $\tb>0$ and $\rot>0$.

%This means that we can split case (2) into two subcases, one where $\rot>0$, and one where $\rot\leq0$. 

The condition given by $(\prime\prime)$ in this case is $n^2 \geq 4 \rot + \tb$, which we will show is equivalent to $(\dagger)$. Since $\rot>0$, $|\rot|=\rot$, so we have $n^2 \geq 4|\rot| + \tb$, which  is the same as the second half of $(\dagger)$.  It remains to show that this case also implies $\tb+n^2\geq0$, which is true because both $\tb$ and $n^2$ are nonnegative.

\item[\textbf{Case (2.2)}]: $\tb>0$ and $\rot<0$.

The condition given by $(\prime\prime)$ in this case is $n^2+4\rot \geq \tb$, which we will show is equivalent to $(\dagger)$. Since $\rot\leq0$, $|\rot|=-\rot$, and we obtain  $n^2 \geq 4|\rot| + \tb$. Again, this is the same as the second half of $(\dagger)$, so it suffices to show that this case also implies $\tb+n^2\geq0$. This is true because both $\tb$ and $n^2$ are nonnegative.

\end{itemize}

Since $(*)$ is equivalent to $(\dagger)$ in every case, we have shown that they are equivalent conditions. Therefore, for any Legendrian knot $\Lambda$ with a representation on an $n\times n$ mosaic, both of the following inequalities must be satisfied:
\begin{align*}
    n^2&\geq4|\rot(\Lambda)|+\tb(\Lambda);\\
    n^2&\geq-\tb(\Lambda).
\end{align*}

\pagebreak

\section{Counting Legendrian Link Mosaics} \label{app:counting}

Let $D_L^{(m,n)}$ be the number of $m\times n$ Legendrian link mosaics. The following is an implementation in Mathematica of the algorithm given in Section \ref{sec:counting} to compute $D_L^{(m,n)}$:
\begin{lstlisting}[language=Mathematica]
x[0] = o[0] = {{1}};
x[n_] := ArrayFlatten[{{x[n - 1], o[n - 1]}, {o[n - 1], x[n - 1]}}];
o[n_] := ArrayFlatten[{{o[n - 1], x[n - 1]}, {x[n - 1], 3 * o[n - 1]}}];
legendrian[m_, n_] :=
   If[m > 1 && n > 1,
   2 * Total[MatrixPower[x[m - 2] + o[m - 2], n - 2], 2], 1];
ParallelTable[legendrian[m, n], {m, 1, 11}, {n, 1, m}]
\end{lstlisting}

Tables \ref{tab:11-mosaics}, \ref{tab:mosaic-counts-2}, and \ref{tab:mosaic-counts-3} below list all values of $D_L^{(m,n)}$ for $1\leq n\leq m\leq 11$. 

\begin{table}[H]
\centering
\begin{tabular}{|l|l|}
\hline
& $D_L^{(11,n)}$                                                      \\ \hline
$n=1$                                                   & 1                                                         \\ \hline
$n=2$                                                & 1,024                                                     \\ \hline
$n=3$                                             & 11,208,704                                                \\ \hline
$n=4$                                & 236,710,854,656                                           \\ \hline
$n=5$                               & 6,252,734,836,998,144                                     \\ \hline
$n=6$                    & 181,637,396,577,778,663,424                               \\ \hline
$n=7$                  & 5,528,978,020,972,513,728,659,456                         \\ \hline
$n=8$           & 172,538,354,498,746,094,406,296,666,112                   \\ \hline
$n=9$         & 5,458,581,384,707,531,006,284,534,386,786,304             \\ \hline
$n=10$   &   174,020,259,805,998,593,378,475,801,984,133,758,976       \\ \hline
$n=11$    & 5,571,666,891,811,926,168,753,521,842,383,673,521,864,704 \\ \hline

\end{tabular}
\caption{The numbers $D_L^{(11,n)}$ of $11\times n$ Legendrian link mosaics, $1\leq n\leq 11$.}
\label{tab:11-mosaics}
\end{table}

\begin{sidewaystable}
\begin{center}
    \begin{tabular}{|l|l|l|l|l|l|l|}
\hline
$D_T^{(m,n)}$  & $m=3$ & $m=4$   & $m=5$     & $m=6$           & $m=7$                 & $m=8$                    \\ \hline
$n=1$          & 1   & 1     & 1       & 1             & 1                   & 1                               \\ \hline
$n=2$          & 4   & 8     & 16      & 32            & 64                  & 128                                        \\ \hline
$n=3$          & 20  & 104   & 544     & 2,848         & 14,912              & 78,080                                             \\ \hline
$n=4$           &       & 1,504 & 22,208  & 329,216       & 4,883,968           & 72,464,384           \                             \\ \hline
$n=5$           &       &         & 948,032 & 40,930,304    & 1,772,261,888       & 76,795,762,688                                      \\ \hline
$n=6$             &       &         &           & 5,204,262,912 & 666,548,862,976     & 85,575,149,027,328                               \\ \hline
$n=7$              &       &         &           &                 & 254,112,496,082,944 & 97,392,800,416,399,360                             \\ \hline
$n=8$           &       &         &           &                 &                       & 111,879,597,850,371,293,184                 \\ \hline
\end{tabular}
    \caption{$D_L^{(m,n)}$ for all $1\leq n\leq m\leq 8$, omitting $D_L^{(1,1)}=D_L^{(2,1)}=1$ and $D_L^{(2,2)}=2$.}
    \label{tab:mosaic-counts-2}
\end{center}

\begin{center}
\begin{tabular}{|l|l|l|}
\hline
$D_L^{(m,n)}$ & $m=9$                                 & $m=10$                                                         \\ \hline
$n=1$          & 1                                   & 1                                                                                                      \\ \hline
$n=2$         & 256                                 & 512                                                          \\ \hline
$n=3$         &  408,832                             & 2,140,672                                                   \\ \hline
$n=4$                             & 1,075,195,904                       & 15,953,379,328\\ \hline
$n=5$                     & 3,328,369,229,824                   & 144,260,644,012,032\\ \hline
$n=6$       & 10,995,378,691,637,248              & 1,413,150,484,120,731,648                                  \\ \hline
$n=7$       & 37,406,060,364,420,743,168          & 14,378,743,075,001,419,694,080                             \\ \hline
$n=8$   & 129,059,960,656,754,513,018,880     & 149,160,169,737,929,856,250,806,272                     \\ \hline
$n=9$     & 448,381,477,417,976,615,986,528,256 & 1,563,124,967,138,785,231,707,530,330,112                  \\ \hline
$n=10$   &     & 16,469,260,582,635,747,355,818,375,736,459,264       \\ \hline
\end{tabular}
    \caption{$D_L^{(m,n)}$ for all $9\leq m\leq 10$ and $1\leq n\leq m$.}
    \label{tab:mosaic-counts-3}
\end{center}
\end{sidewaystable}

\pagebreak

\section{Census Diagrams} \label{census-list}
    In this section, we provide censuses of known mosaic numbers of Legendrian knots.
    Recall that for any smooth knot type, its Legendrian representatives can be depicted in a mountain range organized by their classical invariants.
    We label each Legendrian representative in each mountain range with its mosaic number or our best bounds on it (indicated with a ``$\geq$" symbol), as determined by the computer search described in Section \ref{algos}, the bounds given in Sections \ref{bounds} and \ref{unknot-section}, and our own constructions. Yellow vertices denote examples of representatives where stabilization decreases the mosaic number.
    As the sign of a Legendrian knot's rotation number depends only on its orientation, the censuses presented in this paper include only the right half of the mountain diagram where knots have non-negative rotation number. A complete census can be obtained by mirroring these censuses horizontally about the column with $\rot = 0$. 

    \vspace{-10pt}
    % \documentclass{article}
% \usepackage{graphicx} % Required for inserting images
% \usepackage{fancyhdr}
% \usepackage{extramarks}
% \usepackage{amsmath}
% \usepackage{amsthm}
% \usepackage{amssymb}
% \usepackage{amsfonts}
% \usepackage{tikz}
% \usepackage[plain]{algorithm}
% \usepackage{algpseudocode}
%  \usepackage{relsize}
%  \usepackage[shortlabels]{enumitem}
%  \usepackage{url}
%  \usepackage{hyperref}

% \usepackage{tikzscale}
% \usetikzlibrary{positioning}

% \newcommand{\Q}{\mathbb{Q}}
% \newcommand{\R}{\mathbb{R}}
% \newcommand{\C}{\mathbb{C}}
% \newcommand{\Z}{\mathbb{Z}}
% \newcommand{\N}{\mathbb{N}}
% \newcommand{\F}{\mathbb{F}}

% \newtheorem{prop}{Proposition}
% \newtheorem{cor}{Corollary}
% \newtheorem{theorem}{Theorem}
% \newtheorem{conjecture}{Conjecture}
% \newtheorem{lemma}{Lemma}
% \newtheorem{obs}{Observation}

% \DeclareMathOperator{\tb}{tb}

% \title{24.02 Legendrian Knot Mosaics}
% \author{Margaret Kipe, Samantha Pezzimenti, Leif Schaumann, Luc Ta. Tony Wong}
% \date{May 2024}

%\begin{document}

\newcommand{\tinygeq}{\mathrel{\raisebox{.4ex}{\scalebox{0.7}{$\geq$}}\mkern-5mu}}
\begin{figure}[H]
    \centering
    \resizebox{0.4\linewidth}{!}{
    %\scalebox{0.6}{
    \input{new_censuses/unknots_new.tikz}
    }
    \caption{Unknot census}
    \label{fig:new-unknots}
\end{figure}

\begin{figure}[H]
    \centering
    %\scalebox{0.6}{
    \resizebox{\linewidth}{!}{
    \tikzstyle{black}=[fill=white, draw=black, shape=circle, minimum size=0.8cm]
\tikzstyle{bruh}=[fill=yellow, draw=black, shape=circle, minimum size=0.8cm]
\tikzstyle{none}=[fill=none, draw=none, shape=circle]

\pgfdeclarelayer{nodelayer}
\pgfdeclarelayer{edgelayer}
\pgfsetlayers{edgelayer,nodelayer}

\begin{tikzpicture}
	\begin{pgfonlayer}{nodelayer}
		\node [style=black] (1) at (-0.25, 0.75) {5};
		\node [style=black] (2) at (0.5, 1.5) {5};
		\node [style=black] (5) at (1.25, 0.75) {5};
		\node [style=bruh] (6) at (2, 0) {6};
		\node [style=black] (8) at (0.5, 0) {5};
		\node [style=bruh] (11) at (2.75, -0.75) {6};
		\node [style=bruh] (13) at (1.25, -0.75) {5};
		\node [style=black] (14) at (-0.25, -0.75) {5};
		\node [style=black] (18) at (0.5, -1.5) {5};
		\node [style=black] (19) at (1.25, -2.25) {5};
		\node [style=black] (20) at (2, -3) {6};
		\node [style=bruh] (21) at (-0.25, -2.25) {6};
		\node [style=bruh] (23) at (0.5, -3) {5};
		\node [style=bruh] (25) at (2, -1.5) {5};
		\node [style=bruh] (26) at (2.75, -2.25) {6};
		\node [style=black] (30) at (2.75, -3.75) {6};
		\node [style=black] (34) at (-0.25, -3.75) {6};
		\node [style=black] (35) at (1.25, -3.75) {6};
		\node [style=black] (37) at (2, -4.5) {6};
		\node [style=black] (39) at (0.5, -4.5) {6};
		\node [style=black] (41) at (2.75, -5.25) {6};
		\node [style=black] (42) at (1.25, -5.25) {6};
		\node [style=black] (43) at (2, -6) {6};
		\node [style=bruh] (47) at (3.5, -4.5) {6};
		\node [style=bruh] (48) at (3.5, -3) {6};
		\node [style=bruh] (49) at (3.5, -1.5) {$\tinygeq7$};
		\node [style=none] (50) at (-1.25, 0.75) {-7};
		\node [style=none] (51) at (-1.25, 1.5) {-6};
		\node [style=none] (52) at (-2, -4.125) {tb};
		\node [style=none] (53) at (-1.25, 0) {-8};
		\node [style=none] (54) at (-1.25, -0.75) {-9};
		\node [style=none] (55) at (-1.25, -1.5) {-10};
		\node [style=none] (56) at (-1.25, -2.25) {-11};
		\node [style=none] (57) at (-1.25, -3) {-12};
		\node [style=none] (58) at (-1.25, -3.75) {-13};
		\node [style=none] (59) at (-1.25, -4.5) {-14};
		\node [style=none] (60) at (-1.25, -5.25) {-15};
		\node [style=none] (61) at (-1.25, -6) {-16};
		\node [style=none] (62) at (-0.25, 2.25) {0};
		\node [style=none] (63) at (0.5, 2.25) {1};
		\node [style=none] (64) at (1.25, 2.25) {2};
		\node [style=none] (65) at (2, 2.25) {3};
		\node [style=none] (66) at (2.75, 2.25) {4};
		\node [style=none] (67) at (2.25, 2.75) {rot};
		\node [style=none] (68) at (-1.25, -6.75) {-17};
		\node [style=none] (69) at (-1.25, -7.5) {-18};
		\node [style=none] (70) at (-1.25, -8.25) {-19};
		\node [style=none] (71) at (-1.25, -9) {-20};
		\node [style=none] (72) at (-1.25, -9.75) {-21};
		\node [style=none] (73) at (3.5, 2.25) {5};
		\node [style=none] (74) at (4.25, 2.25) {6};
		\node [style=bruh] (75) at (4.25, -2.25) {$\tinygeq7$};
		\node [style=bruh] (76) at (4.25, -3.75) {$\tinygeq7$};
		\node [style=black] (77) at (-0.25, -5.25) {6};
		\node [style=black] (78) at (-0.25, -9.75) {$\tinygeq7$};
		\node [style=black] (79) at (2.75, -9.75) {$\tinygeq7$};
		\node [style=black] (80) at (1.25, -8.25) {6};
		\node [style=bruh] (81) at (2, -9) {6};
		\node [style=black] (82) at (0.5, -9) {6};
		\node [style=black] (83) at (1.25, -9.75) {6};
		\node [style=black] (84) at (3.5, -9) {$\tinygeq7$};
		\node [style=black] (85) at (0.5, -6) {6};
		\node [style=black] (86) at (1.25, -6.75) {6};
		\node [style=black] (87) at (-0.25, -6.75) {6};
		\node [style=black] (88) at (0.5, -7.5) {6};
		\node [style=black] (89) at (3.5, -6) {6};
		\node [style=black] (91) at (2.75, -6.75) {6};
		\node [style=bruh] (93) at (5, -4.5) {$\tinygeq7$};
		\node [style=bruh] (94) at (4.25, -5.25) {6};
		\node [style=none] (95) at (5, 2.25) {7};
		\node [style=bruh] (96) at (2.75, -8.25) {$\tinygeq7$};
		\node [style=black] (97) at (3.5, -7.5) {$\tinygeq7$};
		\node [style=black] (98) at (4.25, -6.75) {$\tinygeq7$};
		\node [style=black] (99) at (2, -7.5) {6};
		\node [style=black] (100) at (-0.25, -8.25) {6};
	\end{pgfonlayer}
	\begin{pgfonlayer}{edgelayer}
		\draw (2) to (1);
		\draw (2) to (5);
		\draw (1) to (8);
		\draw (5) to (6);
		\draw (5) to (8);
		\draw (8) to (14);
		\draw (8) to (13);
		\draw (6) to (13);
		\draw (6) to (11);
		\draw (14) to (18);
		\draw (13) to (18);
		\draw (13) to (25);
		\draw (11) to (25);
		\draw (18) to (19);
		\draw (18) to (21);
		\draw (25) to (26);
		\draw (25) to (19);
		\draw (19) to (23);
		\draw (21) to (23);
		\draw (23) to (34);
		\draw (23) to (35);
		\draw (34) to (39);
		\draw (19) to (20);
		\draw (20) to (35);
		\draw (20) to (30);
		\draw (35) to (39);
		\draw (35) to (37);
		\draw (26) to (20);
		\draw (30) to (37);
		\draw (39) to (42);
		\draw (37) to (42);
		\draw (37) to (41);
		\draw (42) to (43);
		\draw (41) to (43);
		\draw (11) to (49);
		\draw (49) to (26);
		\draw (26) to (48);
		\draw (48) to (30);
		\draw (30) to (47);
		\draw (47) to (41);
		\draw (80) to (82);
		\draw (80) to (81);
		\draw (82) to (83);
		\draw (81) to (83);
		\draw (85) to (87);
		\draw (85) to (86);
		\draw (87) to (88);
		\draw (86) to (88);
		\draw (89) to (91);
		\draw (77) to (39);
		\draw (85) to (77);
		\draw (85) to (42);
		\draw (86) to (43);
		\draw (48) to (75);
		\draw (75) to (49);
		\draw (48) to (76);
		\draw (76) to (47);
		\draw (94) to (47);
		\draw (76) to (93);
		\draw (94) to (89);
		\draw (89) to (41);
		\draw (43) to (91);
		\draw (94) to (93);
		\draw (89) to (98);
		\draw (98) to (97);
		\draw (97) to (91);
		\draw (91) to (99);
		\draw (99) to (86);
		\draw (88) to (80);
		\draw (80) to (99);
		\draw (99) to (96);
		\draw (96) to (81);
		\draw (96) to (97);
		\draw (84) to (96);
		\draw (81) to (79);
		\draw (82) to (78);
		\draw (100) to (82);
		\draw (100) to (88);
	\end{pgfonlayer}
\end{tikzpicture}
    \tikzstyle{black}=[fill=white, draw=black, shape=circle, minimum size=0.8cm]
\tikzstyle{bruh}=[fill=yellow, draw=black, shape=circle, minimum size=0.8cm]
\tikzstyle{none}=[fill=none, draw=none, shape=circle]

\pgfdeclarelayer{nodelayer}
\pgfdeclarelayer{edgelayer}
\pgfsetlayers{edgelayer,nodelayer}

\begin{tikzpicture}
	\begin{pgfonlayer}{nodelayer}
		\node [style=black] (1) at (-0.25, 1.5) {5};
		\node [style=black] (8) at (0.5, 0.75) {5};
		\node [style=black] (13) at (1.25, 0) {5};
		\node [style=black] (14) at (-0.25, 0) {5};
		\node [style=black] (18) at (0.5, -0.75) {5};
		\node [style=bruh] (19) at (1.25, -1.5) {5};
		\node [style=black] (20) at (2, -2.25) {6};
		\node [style=black] (21) at (-0.25, -1.5) {5};
		\node [style=black] (23) at (0.5, -2.25) {5};
		\node [style=bruh] (25) at (2, -0.75) {6};
		\node [style=black] (26) at (2.75, -1.5) {6};
		\node [style=black] (30) at (2.75, -3) {6};
		\node [style=black] (34) at (-0.25, -3) {5};
		\node [style=black] (35) at (1.25, -3) {6};
		\node [style=black] (37) at (2, -3.75) {6};
		\node [style=black] (39) at (0.5, -3.75) {6};
		\node [style=black] (42) at (1.25, -4.5) {6};
		\node [style=black] (43) at (2, -5.25) {6};
		\node [style=bruh] (47) at (3.5, -3.75) {6};
		\node [style=bruh] (48) at (4.25, -3) {$\tinygeq7$};
		\node [style=black] (49) at (3.5, -2.25) {6};
		\node [style=none] (50) at (-1.25, 0.75) {0};
		\node [style=none] (51) at (-1.25, 1.5) {1};
		\node [style=none] (52) at (-2, -4) {tb};
		\node [style=none] (53) at (-1.25, 0) {-1};
		\node [style=none] (54) at (-1.25, -0.75) {-2};
		\node [style=none] (55) at (-1.25, -1.5) {-3};
		\node [style=none] (56) at (-1.25, -2.25) {-4};
		\node [style=none] (57) at (-1.25, -3) {-5};
		\node [style=none] (58) at (-1.25, -3.75) {-6};
		\node [style=none] (59) at (-1.25, -4.5) {-7};
		\node [style=none] (60) at (-1.25, -5.25) {-8};
		\node [style=none] (61) at (-1.25, -6) {-9};
		\node [style=none] (62) at (-0.25, 2.25) {0};
		\node [style=none] (63) at (0.5, 2.25) {1};
		\node [style=none] (64) at (1.25, 2.25) {2};
		\node [style=none] (65) at (2, 2.25) {3};
		\node [style=none] (66) at (2.75, 2.25) {4};
		\node [style=none] (67) at (2, 2.75) {rot};
		\node [style=black] (68) at (2.75, -4.5) {6};
		\node [style=black] (69) at (-0.25, -4.5) {6};
		\node [style=black] (70) at (-0.25, -6) {6};
		\node [style=black] (71) at (-0.25, -7.5) {6};
		\node [style=black] (72) at (0.5, -6.75) {6};
		\node [style=black] (73) at (1.25, -6) {6};
		\node [style=black] (74) at (0.5, -5.25) {6};
		\node [style=none] (75) at (-1.25, -6.75) {-10};
		\node [style=none] (76) at (-1.25, -7.5) {-11};
		\node [style=none] (77) at (3.5, 2.25) {5};
		\node [style=bruh] (78) at (3.5, -5.25) {6};
		\node [style=bruh] (79) at (4.25, -4.5) {$\tinygeq7$};
		\node [style=none] (80) at (-1.25, -8.25) {-12};
		\node [style=none] (81) at (-1.25, -9) {-13};
		\node [style=black] (82) at (0.5, -8.25) {6};
		\node [style=black] (83) at (1.25, -7.5) {6};
		\node [style=black] (84) at (-0.25, -9) {6};
		\node [style=black] (85) at (1.25, -9) {$\tinygeq7$};
		\node [style=black] (86) at (2, -8.25) {$\tinygeq7$};
		\node [style=black] (87) at (2.75, -7.5) {$\tinygeq7$};
		\node [style=black] (88) at (3.5, -6.75) {$\tinygeq7$};
		\node [style=black] (89) at (4.25, -6) {$\tinygeq7$};
		\node [style=black] (90) at (2.75, -6) {6};
		\node [style=black] (91) at (2, -6.75) {6};
		\node [style=none] (92) at (4.25, 2.25) {6};
	\end{pgfonlayer}
	\begin{pgfonlayer}{edgelayer}
		\draw (1) to (8);
		\draw (8) to (14);
		\draw (8) to (13);
		\draw (14) to (18);
		\draw (13) to (18);
		\draw (13) to (25);
		\draw (18) to (19);
		\draw (18) to (21);
		\draw (25) to (26);
		\draw (25) to (19);
		\draw (19) to (23);
		\draw (21) to (23);
		\draw (23) to (34);
		\draw (23) to (35);
		\draw (34) to (39);
		\draw (19) to (20);
		\draw (20) to (35);
		\draw (20) to (30);
		\draw (35) to (39);
		\draw (35) to (37);
		\draw (26) to (20);
		\draw (30) to (37);
		\draw (39) to (42);
		\draw (37) to (42);
		\draw (42) to (43);
		\draw (49) to (26);
		\draw (49) to (48);
		\draw (30) to (49);
		\draw (47) to (30);
		\draw (47) to (48);
		\draw (37) to (68);
		\draw (43) to (68);
		\draw (47) to (68);
		\draw (39) to (69);
		\draw (69) to (74);
		\draw (74) to (70);
		\draw (70) to (72);
		\draw (72) to (71);
		\draw (72) to (73);
		\draw (73) to (74);
		\draw (43) to (73);
		\draw (78) to (79);
		\draw (83) to (82);
		\draw (90) to (91);
		\draw (71) to (82);
		\draw (82) to (84);
		\draw (82) to (85);
		\draw (72) to (83);
		\draw (83) to (91);
		\draw (91) to (73);
		\draw (43) to (90);
		\draw (90) to (78);
		\draw (78) to (68);
		\draw (47) to (79);
		\draw (78) to (89);
		\draw (88) to (89);
		\draw (88) to (90);
		\draw (91) to (87);
		\draw (87) to (88);
		\draw (87) to (86);
		\draw (86) to (83);
		\draw (85) to (86);
	\end{pgfonlayer}
\end{tikzpicture}
    }
    \caption{Negative (i.e., $3_1$) and positive (i.e., $m(3_1)$) trefoil censuses}
    \label{fig:trefoils}
\end{figure}
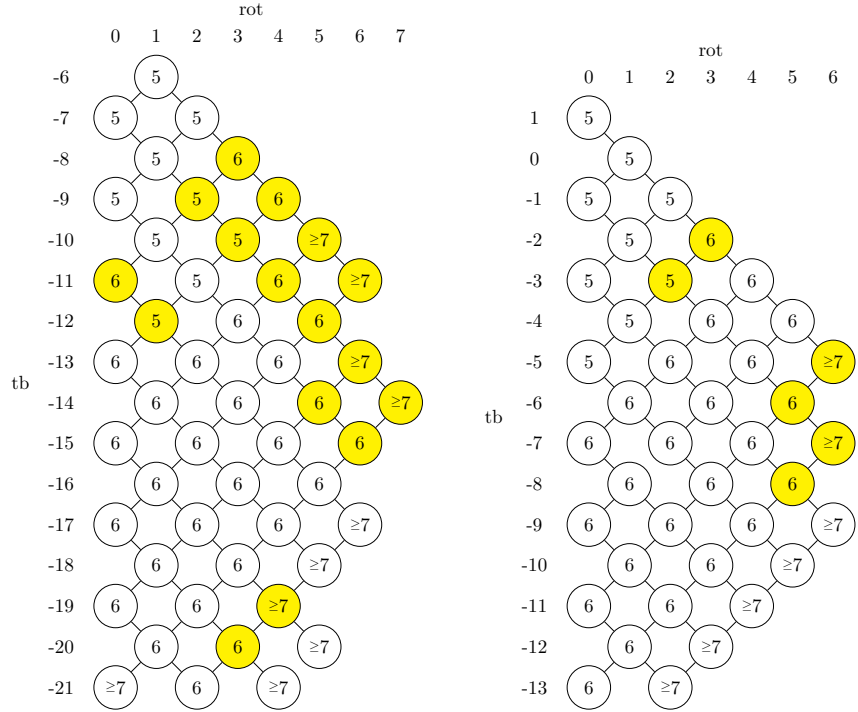

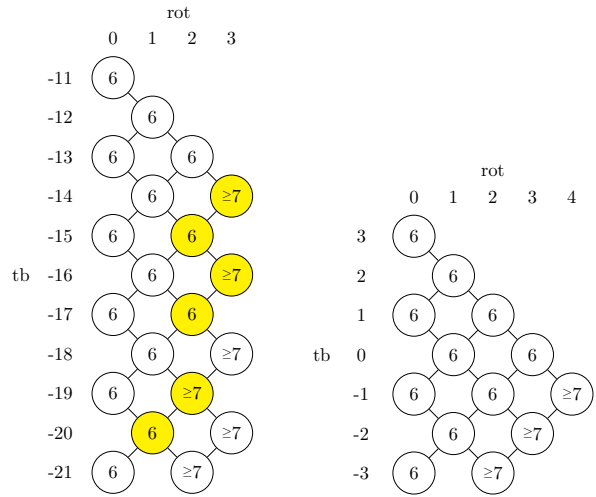
\begin{figure}[H]
    \centering
    \scalebox{0.7}{
    %\resizebox{\linewidth}{!}{
    \tikzstyle{black}=[fill=white, draw=black, shape=circle, minimum size=0.8cm]
\tikzstyle{bruh}=[fill=yellow, draw=black, shape=circle, minimum size=0.8cm]
\tikzstyle{none}=[fill=none, draw=none, shape=circle]

\pgfdeclarelayer{nodelayer}
\pgfdeclarelayer{edgelayer}
\pgfsetlayers{edgelayer,nodelayer}

\begin{tikzpicture}
	\begin{pgfonlayer}{nodelayer}
		\node [style=black] (21) at (-0.25, -2.25) {6};
		\node [style=black] (23) at (0.5, -3) {6};
		\node [style=black] (34) at (-0.25, -3.75) {6};
		\node [style=black] (35) at (1.25, -3.75) {6};
		\node [style=bruh] (37) at (2, -4.5) {$\tinygeq7$};
		\node [style=black] (39) at (0.5, -4.5) {6};
		\node [style=bruh] (42) at (1.25, -5.25) {6};
		\node [style=bruh] (43) at (2, -6) {$\tinygeq7$};
		\node [style=none] (52) at (-2, -6) {tb};
		\node [style=none] (56) at (-1.25, -2.25) {-11};
		\node [style=none] (57) at (-1.25, -3) {-12};
		\node [style=none] (58) at (-1.25, -3.75) {-13};
		\node [style=none] (59) at (-1.25, -4.5) {-14};
		\node [style=none] (60) at (-1.25, -5.25) {-15};
		\node [style=none] (61) at (-1.25, -6) {-16};
		\node [style=none] (62) at (-0.25, -1.5) {0};
		\node [style=none] (63) at (0.5, -1.5) {1};
		\node [style=none] (64) at (1.25, -1.5) {2};
		\node [style=none] (65) at (2, -1.5) {3};
		\node [style=none] (67) at (1, -1) {rot};
		\node [style=none] (68) at (-1.25, -6.75) {-17};
		\node [style=none] (69) at (-1.25, -7.5) {-18};
		\node [style=none] (70) at (-1.25, -8.25) {-19};
		\node [style=none] (71) at (-1.25, -9) {-20};
		\node [style=none] (72) at (-1.25, -9.75) {-21};
		\node [style=black] (77) at (-0.25, -5.25) {6};
		\node [style=black] (78) at (-0.25, -9.75) {6};
		\node [style=bruh] (80) at (1.25, -8.25) {$\tinygeq7$};
		\node [style=black] (81) at (2, -9) {$\tinygeq7$};
		\node [style=bruh] (82) at (0.5, -9) {6};
		\node [style=black] (83) at (1.25, -9.75) {$\tinygeq7$};
		\node [style=black] (85) at (0.5, -6) {6};
		\node [style=bruh] (86) at (1.25, -6.75) {6};
		\node [style=black] (87) at (-0.25, -6.75) {6};
		\node [style=black] (88) at (0.5, -7.5) {6};
		\node [style=black] (99) at (2, -7.5) {$\tinygeq7$};
		\node [style=black] (100) at (-0.25, -8.25) {6};
	\end{pgfonlayer}
	\begin{pgfonlayer}{edgelayer}
		\draw (21) to (23);
		\draw (23) to (34);
		\draw (23) to (35);
		\draw (34) to (39);
		\draw (35) to (39);
		\draw (35) to (37);
		\draw (39) to (42);
		\draw (37) to (42);
		\draw (42) to (43);
		\draw (80) to (82);
		\draw (80) to (81);
		\draw (82) to (83);
		\draw (81) to (83);
		\draw (85) to (87);
		\draw (85) to (86);
		\draw (87) to (88);
		\draw (86) to (88);
		\draw (77) to (39);
		\draw (85) to (77);
		\draw (85) to (42);
		\draw (86) to (43);
		\draw (99) to (86);
		\draw (88) to (80);
		\draw (80) to (99);
		\draw (82) to (78);
		\draw (100) to (82);
		\draw (100) to (88);
	\end{pgfonlayer}
\end{tikzpicture}
    \tikzstyle{black}=[fill=white, draw=black, shape=circle, minimum size=0.8cm]
\tikzstyle{bruh}=[fill=yellow, draw=black, shape=circle, minimum size=0.8cm]
\tikzstyle{none}=[fill=none, draw=none, shape=circle]

\pgfdeclarelayer{nodelayer}
\pgfdeclarelayer{edgelayer}
\pgfsetlayers{edgelayer,nodelayer}

\begin{tikzpicture}
	\begin{pgfonlayer}{nodelayer}
		\node [style=black] (21) at (-0.25, -2.25) {6};
		\node [style=black] (23) at (0.5, -3) {6};
		\node [style=black] (34) at (-0.25, -3.75) {6};
		\node [style=black] (35) at (1.25, -3.75) {6};
		\node [style=black] (37) at (2, -4.5) {6};
		\node [style=black] (39) at (0.5, -4.5) {6};
		\node [style=black] (42) at (1.25, -5.25) {6};
		\node [style=black] (43) at (2, -6) {$\tinygeq7$};
		\node [style=none] (52) at (-2, -4.5) {tb};
		\node [style=none] (56) at (-1.25, -2.25) {3};
		\node [style=none] (57) at (-1.25, -3) {2};
		\node [style=none] (58) at (-1.25, -3.75) {1};
		\node [style=none] (59) at (-1.25, -4.5) {0};
		\node [style=none] (60) at (-1.25, -5.25) {-1};
		\node [style=none] (61) at (-1.25, -6) {-2};
		\node [style=none] (62) at (-0.25, -1.5) {0};
		\node [style=none] (63) at (0.5, -1.5) {1};
		\node [style=none] (64) at (1.25, -1.5) {2};
		\node [style=none] (65) at (2, -1.5) {3};
		\node [style=none] (67) at (1.25, -1) {rot};
		\node [style=none] (68) at (-1.25, -6.75) {-3};
		\node [style=black] (77) at (-0.25, -5.25) {6};
		\node [style=black] (85) at (0.5, -6) {6};
		\node [style=black] (86) at (1.25, -6.75) {$\tinygeq7$};
		\node [style=black] (87) at (-0.25, -6.75) {6};
		\node [style=none] (101) at (2.75, -1.5) {4};
		\node [style=black] (102) at (2.75, -5.25) {$\tinygeq7$};
	\end{pgfonlayer}
	\begin{pgfonlayer}{edgelayer}
		\draw (21) to (23);
		\draw (23) to (34);
		\draw (23) to (35);
		\draw (34) to (39);
		\draw (35) to (39);
		\draw (35) to (37);
		\draw (39) to (42);
		\draw (37) to (42);
		\draw (42) to (43);
		\draw (85) to (87);
		\draw (85) to (86);
		\draw (77) to (39);
		\draw (85) to (77);
		\draw (85) to (42);
		\draw (86) to (43);
		\draw (37) to (102);
		\draw (102) to (43);
	\end{pgfonlayer}
\end{tikzpicture}
    }
    \caption{Negative (i.e., $3_1\# 3_1$) and positive (i.e., $m(3_1)\#m(3_1)$) granny knot censuses}
    \label{fig:granny}
\end{figure}

\begin{figure}[H]
    \centering
    \scalebox{0.7}{
    \tikzstyle{black}=[fill=white, draw=black, shape=circle, minimum size=0.8cm]
\tikzstyle{bruh}=[fill=yellow, draw=black, shape=circle, minimum size=0.8cm]
\tikzstyle{none}=[fill=none, draw=none, shape=circle]

\pgfdeclarelayer{nodelayer}
\pgfdeclarelayer{edgelayer}
\pgfsetlayers{edgelayer,nodelayer}

\begin{tikzpicture}
	\begin{pgfonlayer}{nodelayer}
		\node [style=black] (14) at (-0.25, -0.75) {6};
		\node [style=black] (18) at (0.5, -1.5) {6};
		\node [style=black] (19) at (1.25, -2.25) {6};
		\node [style=black] (20) at (2, -3) {6};
		\node [style=black] (21) at (-0.25, -2.25) {6};
		\node [style=black] (23) at (0.5, -3) {6};
		\node [style=black] (30) at (2.75, -3.75) {6};
		\node [style=black] (34) at (-0.25, -3.75) {6};
		\node [style=black] (35) at (1.25, -3.75) {6};
		\node [style=black] (37) at (2, -4.5) {6};
		\node [style=black] (39) at (0.5, -4.5) {6};
		\node [style=black] (41) at (2.75, -5.25) {6};
		\node [style=black] (42) at (1.25, -5.25) {6};
		\node [style=black] (43) at (2, -6) {6};
		\node [style=black] (47) at (3.5, -4.5) {6};
		\node [style=none] (50) at (-1.25, -3.75) {-7};
		\node [style=none] (52) at (-2, -6) {tb};
		\node [style=none] (53) at (-1.25, -4.5) {-8};
		\node [style=none] (54) at (-1.25, -5.25) {-9};
		\node [style=none] (55) at (-1.25, -6) {-10};
		\node [style=none] (56) at (-1.25, -6.75) {-11};
		\node [style=none] (57) at (-1.25, -7.5) {-12};
		\node [style=none] (58) at (-1.25, -8.25) {-13};
		\node [style=none] (59) at (-1.25, -9) {-14};
		\node [style=none] (60) at (-1.25, -9.75) {-15};
		\node [style=none] (61) at (-1.25, -10.5) {-16};
		\node [style=none] (62) at (-0.25, 0) {0};
		\node [style=none] (63) at (0.5, 0) {1};
		\node [style=none] (64) at (1.25, 0) {2};
		\node [style=none] (65) at (2, 0) {3};
		\node [style=none] (66) at (2.75, 0) {4};
		\node [style=none] (67) at (2, 0.5) {rot};
		\node [style=none] (73) at (3.5, 0) {5};
		\node [style=black] (77) at (-0.25, -5.25) {6};
		\node [style=black] (78) at (-0.25, -9.75) {6};
		\node [style=black] (79) at (2.75, -9.75) {$\tinygeq 7$};
		\node [style=black] (80) at (1.25, -8.25) {6};
		\node [style=black] (81) at (2, -9) {$\tinygeq 7$};
		\node [style=black] (82) at (0.5, -9) {6};
		\node [style=black] (83) at (1.25, -9.75) {$\tinygeq 7$};
		\node [style=black] (84) at (3.5, -9) {$\tinygeq 7$};
		\node [style=black] (85) at (0.5, -6) {6};
		\node [style=black] (86) at (1.25, -6.75) {6};
		\node [style=black] (87) at (-0.25, -6.75) {6};
		\node [style=black] (88) at (0.5, -7.5) {6};
		\node [style=bruh] (89) at (3.5, -6) {6};
		\node [style=black] (91) at (2.75, -6.75) {6};
		\node [style=black] (96) at (2.75, -8.25) {6};
		\node [style=bruh] (97) at (3.5, -7.5) {6};
		\node [style=black] (99) at (2, -7.5) {6};
		\node [style=black] (100) at (-0.25, -8.25) {6};
		\node [style=none] (101) at (-1.25, -2.25) {-5};
		\node [style=none] (102) at (-1.25, -1.5) {-4};
		\node [style=none] (103) at (-1.25, -0.75) {-3};
		\node [style=none] (104) at (-1.25, -3) {-6};
		\node [style=bruh] (105) at (4.25, -5.25) {$\tinygeq 7$};
		\node [style=none] (106) at (4.25, 0) {6};
		\node [style=bruh] (107) at (4.25, -6.75) {$\tinygeq 7$};
		\node [style=black] (108) at (4.25, -8.25) {$\tinygeq 7$};
		\node [style=black] (109) at (0.5, -10.5) {$\tinygeq 7$};
		\node [style=black] (110) at (2, -10.5) {$\tinygeq 7$};
	\end{pgfonlayer}
	\begin{pgfonlayer}{edgelayer}
		\draw (14) to (18);
		\draw (18) to (19);
		\draw (18) to (21);
		\draw (19) to (23);
		\draw (21) to (23);
		\draw (23) to (34);
		\draw (23) to (35);
		\draw (34) to (39);
		\draw (19) to (20);
		\draw (20) to (35);
		\draw (20) to (30);
		\draw (35) to (39);
		\draw (35) to (37);
		\draw (30) to (37);
		\draw (39) to (42);
		\draw (37) to (42);
		\draw (37) to (41);
		\draw (42) to (43);
		\draw (41) to (43);
		\draw (30) to (47);
		\draw (47) to (41);
		\draw (80) to (82);
		\draw (80) to (81);
		\draw (82) to (83);
		\draw (81) to (83);
		\draw (85) to (87);
		\draw (85) to (86);
		\draw (87) to (88);
		\draw (86) to (88);
		\draw (89) to (91);
		\draw (77) to (39);
		\draw (85) to (77);
		\draw (85) to (42);
		\draw (86) to (43);
		\draw (89) to (41);
		\draw (43) to (91);
		\draw (97) to (91);
		\draw (91) to (99);
		\draw (99) to (86);
		\draw (88) to (80);
		\draw (80) to (99);
		\draw (99) to (96);
		\draw (96) to (81);
		\draw (96) to (97);
		\draw (84) to (96);
		\draw (81) to (79);
		\draw (82) to (78);
		\draw (100) to (82);
		\draw (100) to (88);
		\draw (105) to (89);
		\draw (47) to (105);
		\draw (89) to (107);
		\draw (107) to (97);
		\draw (97) to (108);
		\draw (108) to (84);
		\draw (79) to (84);
		\draw (83) to (110);
		\draw (110) to (79);
		\draw (78) to (109);
		\draw (109) to (83);
	\end{pgfonlayer}
\end{tikzpicture}
    }
    \caption{$4_1=m(4_1)$ census}
    \label{fig:4_1}
\end{figure}

\begin{figure}[H]
    \centering
    \resizebox{\linewidth}{!}{
    \tikzstyle{black}=[fill=white, draw=black, shape=circle, minimum size=0.8cm]
\tikzstyle{bruh}=[fill=yellow, draw=black, shape=circle, minimum size=0.8cm]
\tikzstyle{none}=[fill=none, draw=none, shape=circle]

\pgfdeclarelayer{nodelayer}
\pgfdeclarelayer{edgelayer}
\pgfsetlayers{edgelayer,nodelayer}

\begin{tikzpicture}
	\begin{pgfonlayer}{nodelayer}
		\node [style=bruh] (6) at (2, -1.5) {7};
		\node [style=black] (8) at (0.5, -1.5) {6};
		\node [style=bruh] (11) at (2.75, -2.25) {$\tinygeq7$};
		\node [style=bruh] (13) at (1.25, -2.25) {6};
		\node [style=black] (14) at (-0.25, -2.25) {6};
		\node [style=black] (18) at (0.5, -3) {6};
		\node [style=black] (19) at (1.25, -3.75) {6};
		\node [style=bruh] (20) at (2, -4.5) {6};
		\node [style=black] (21) at (-0.25, -3.75) {6};
		\node [style=black] (23) at (0.5, -4.5) {6};
		\node [style=bruh] (25) at (2, -3) {6};
		\node [style=bruh] (26) at (2.75, -3.75) {$\tinygeq7$};
		\node [style=bruh] (30) at (2.75, -5.25) {6};
		\node [style=black] (34) at (-0.25, -5.25) {6};
		\node [style=black] (35) at (1.25, -5.25) {6};
		\node [style=black] (37) at (2, -6) {6};
		\node [style=black] (39) at (0.5, -6) {6};
		\node [style=bruh] (41) at (2.75, -6.75) {6};
		\node [style=black] (42) at (1.25, -6.75) {6};
		\node [style=black] (43) at (2, -7.5) {6};
		\node [style=bruh] (47) at (3.5, -6) {$\tinygeq7$};
		\node [style=bruh] (48) at (3.5, -4.5) {$\tinygeq7$};
		\node [style=black] (49) at (3.5, -3) {$\tinygeq7$};
		\node [style=none] (52) at (-2, -5.25) {tb};
		\node [style=none] (55) at (-1.25, -1.5) {-10};
		\node [style=none] (56) at (-1.25, -2.25) {-11};
		\node [style=none] (57) at (-1.25, -3) {-12};
		\node [style=none] (58) at (-1.25, -3.75) {-13};
		\node [style=none] (59) at (-1.25, -4.5) {-14};
		\node [style=none] (60) at (-1.25, -5.25) {-15};
		\node [style=none] (61) at (-1.25, -6) {-16};
		\node [style=none] (62) at (-0.25, -0.75) {0};
		\node [style=none] (63) at (0.5, -0.75) {1};
		\node [style=none] (64) at (1.25, -0.75) {2};
		\node [style=none] (65) at (2, -0.75) {3};
		\node [style=none] (66) at (2.75, -0.75) {4};
		\node [style=none] (67) at (2.25, -0.25) {rot};
		\node [style=none] (68) at (-1.25, -6.75) {-17};
		\node [style=none] (69) at (-1.25, -7.5) {-18};
		\node [style=none] (70) at (-1.25, -8.25) {-19};
		\node [style=none] (71) at (-1.25, -9) {-20};
		\node [style=none] (73) at (3.5, -0.75) {5};
		\node [style=none] (74) at (4.25, -0.75) {6};
		\node [style=black] (75) at (4.25, -3.75) {$\tinygeq7$};
		\node [style=black] (76) at (4.25, -5.25) {$\tinygeq7$};
		\node [style=black] (77) at (-0.25, -6.75) {6};
		\node [style=black] (85) at (0.5, -7.5) {6};
		\node [style=black] (86) at (1.25, -8.25) {$\tinygeq7$};
		\node [style=black] (87) at (-0.25, -8.25) {6};
		\node [style=black] (88) at (0.5, -9) {$\tinygeq7$};
		\node [style=black] (89) at (3.5, -7.5) {$\tinygeq7$};
		\node [style=black] (91) at (2.75, -8.25) {$\tinygeq7$};
		\node [style=black] (94) at (4.25, -6.75) {$\tinygeq7$};
		\node [style=none] (95) at (5, -0.75) {7};
		\node [style=black] (97) at (3.5, -9) {$\tinygeq7$};
		\node [style=black] (98) at (4.25, -8.25) {$\tinygeq7$};
		\node [style=black] (99) at (2, -9) {$\tinygeq7$};
	\end{pgfonlayer}
	\begin{pgfonlayer}{edgelayer}
		\draw (8) to (14);
		\draw (8) to (13);
		\draw (6) to (13);
		\draw (6) to (11);
		\draw (14) to (18);
		\draw (13) to (18);
		\draw (13) to (25);
		\draw (11) to (25);
		\draw (18) to (19);
		\draw (18) to (21);
		\draw (25) to (26);
		\draw (25) to (19);
		\draw (19) to (23);
		\draw (21) to (23);
		\draw (23) to (34);
		\draw (23) to (35);
		\draw (34) to (39);
		\draw (19) to (20);
		\draw (20) to (35);
		\draw (20) to (30);
		\draw (35) to (39);
		\draw (35) to (37);
		\draw (26) to (20);
		\draw (30) to (37);
		\draw (39) to (42);
		\draw (37) to (42);
		\draw (37) to (41);
		\draw (42) to (43);
		\draw (41) to (43);
		\draw (11) to (49);
		\draw (49) to (26);
		\draw (26) to (48);
		\draw (48) to (30);
		\draw (30) to (47);
		\draw (47) to (41);
		\draw (85) to (87);
		\draw (85) to (86);
		\draw (87) to (88);
		\draw (86) to (88);
		\draw (89) to (91);
		\draw (77) to (39);
		\draw (85) to (77);
		\draw (85) to (42);
		\draw (86) to (43);
		\draw (48) to (75);
		\draw (75) to (49);
		\draw (48) to (76);
		\draw (76) to (47);
		\draw (94) to (47);
		\draw (94) to (89);
		\draw (89) to (41);
		\draw (43) to (91);
		\draw (89) to (98);
		\draw [in=45, out=-135] (98) to (97);
		\draw (97) to (91);
		\draw (91) to (99);
		\draw (99) to (86);
	\end{pgfonlayer}
\end{tikzpicture}
    \tikzstyle{black}=[fill=white, draw=black, shape=circle, minimum size=0.8cm]
\tikzstyle{bruh}=[fill=yellow, draw=black, shape=circle, minimum size=0.8cm]
\tikzstyle{none}=[fill=none, draw=none, shape=circle]

\pgfdeclarelayer{nodelayer}
\pgfdeclarelayer{edgelayer}
\pgfsetlayers{edgelayer,nodelayer}

\begin{tikzpicture}
	\begin{pgfonlayer}{nodelayer}
		\node [style=black] (21) at (-0.25, -4.5) {6};
		\node [style=black] (23) at (0.5, -5.25) {6};
		\node [style=black] (34) at (-0.25, -6) {6};
		\node [style=black] (35) at (1.25, -6) {6};
		\node [style=black] (37) at (2, -6.75) {6};
		\node [style=black] (39) at (0.5, -6.75) {6};
		\node [style=black] (41) at (2.75, -7.5) {6};
		\node [style=black] (42) at (1.25, -7.5) {6};
		\node [style=black] (43) at (2, -8.25) {6};
		\node [style=none] (52) at (-2, -9.25) {tb};
		\node [style=none] (53) at (-1.25, -6) {1};
		\node [style=none] (54) at (-1.25, -6.75) {0};
		\node [style=none] (55) at (-1.25, -7.5) {-1};
		\node [style=none] (56) at (-1.25, -8.25) {-2};
		\node [style=none] (57) at (-1.25, -9) {-3};
		\node [style=none] (58) at (-1.25, -9.75) {-4};
		\node [style=none] (59) at (-1.25, -10.5) {-5};
		\node [style=none] (60) at (-1.25, -11.25) {-6};
		\node [style=none] (61) at (-1.25, -12) {-7};
		\node [style=none] (62) at (-0.25, -3.75) {0};
		\node [style=none] (63) at (0.5, -3.75) {1};
		\node [style=none] (64) at (1.25, -3.75) {2};
		\node [style=none] (65) at (2, -3.75) {3};
		\node [style=none] (66) at (2.75, -3.75) {4};
		\node [style=none] (67) at (1.5, -3.25) {rot};
		\node [style=black] (77) at (-0.25, -7.5) {6};
		\node [style=black] (78) at (-0.25, -12) {6};
		\node [style=black] (79) at (2.75, -12) {$\tinygeq7$};
		\node [style=black] (80) at (1.25, -10.5) {6};
		\node [style=black] (81) at (2, -11.25) {$\tinygeq7$};
		\node [style=black] (82) at (0.5, -11.25) {6};
		\node [style=black] (83) at (1.25, -12) {$\tinygeq7$};
		\node [style=black] (85) at (0.5, -8.25) {6};
		\node [style=black] (86) at (1.25, -9) {6};
		\node [style=black] (87) at (-0.25, -9) {6};
		\node [style=black] (88) at (0.5, -9.75) {6};
		\node [style=bruh] (91) at (2.75, -9) {6};
		\node [style=black] (96) at (2.75, -10.5) {$\tinygeq7$};
		\node [style=black] (99) at (2, -9.75) {6};
		\node [style=black] (100) at (-0.25, -10.5) {6};
		\node [style=none] (111) at (-1.25, -12.75) {-8};
		\node [style=none] (112) at (-1.25, -5.25) {2};
		\node [style=none] (113) at (-1.25, -4.5) {3};
		\node [style=none] (114) at (3.5, -3.75) {5};
		\node [style=bruh] (115) at (3.5, -8.25) {$\tinygeq7$};
		\node [style=black] (116) at (3.5, -9.75) {$\tinygeq7$};
		\node [style=black] (117) at (3.5, -11.25) {$\tinygeq7$};
		\node [style=black] (118) at (0.5, -12.75) {$\tinygeq7$};
		\node [style=black] (119) at (2, -12.75) {$\tinygeq7$};
		\node [style=black] (120) at (3.5, -12.75) {$\tinygeq7$};
	\end{pgfonlayer}
	\begin{pgfonlayer}{edgelayer}
		\draw (21) to (23);
		\draw (23) to (34);
		\draw (23) to (35);
		\draw (34) to (39);
		\draw (35) to (39);
		\draw (35) to (37);
		\draw (39) to (42);
		\draw (37) to (42);
		\draw (37) to (41);
		\draw (42) to (43);
		\draw (41) to (43);
		\draw (80) to (82);
		\draw (80) to (81);
		\draw (82) to (83);
		\draw (81) to (83);
		\draw (85) to (87);
		\draw (85) to (86);
		\draw (87) to (88);
		\draw (86) to (88);
		\draw (77) to (39);
		\draw (85) to (77);
		\draw (85) to (42);
		\draw (86) to (43);
		\draw (43) to (91);
		\draw (91) to (99);
		\draw (99) to (86);
		\draw (88) to (80);
		\draw (80) to (99);
		\draw (99) to (96);
		\draw (96) to (81);
		\draw (81) to (79);
		\draw (82) to (78);
		\draw (100) to (82);
		\draw (100) to (88);
		\draw (41) to (115);
		\draw (115) to (91);
		\draw (91) to (116);
		\draw (96) to (116);
		\draw (117) to (96);
		\draw (79) to (117);
		\draw (120) to (79);
		\draw (79) to (119);
		\draw (119) to (83);
		\draw (83) to (118);
		\draw (118) to (78);
	\end{pgfonlayer}
\end{tikzpicture}
    }
    \caption{$5_1$ and $m(5_1)$ censuses}
    \label{fig:5_1}
\end{figure}

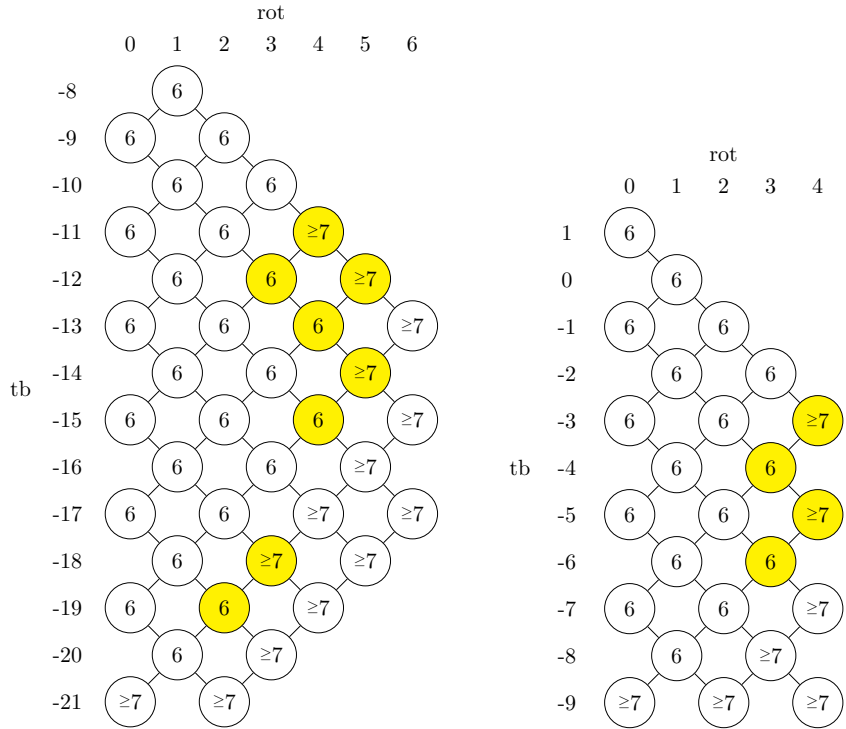
\begin{figure}[H]
    \centering
    \resizebox{\linewidth}{!}{
    \tikzstyle{black}=[fill=white, draw=black, shape=circle, minimum size=0.8cm]
\tikzstyle{bruh}=[fill=yellow, draw=black, shape=circle, minimum size=0.8cm]
\tikzstyle{none}=[fill=none, draw=none, shape=circle]

\pgfdeclarelayer{nodelayer}
\pgfdeclarelayer{edgelayer}
\pgfsetlayers{edgelayer,nodelayer}

\begin{tikzpicture}
	\begin{pgfonlayer}{nodelayer}
		\node [style=black] (18) at (0.5, -4.5) {6};
		\node [style=black] (19) at (1.25, -5.25) {6};
		\node [style=black] (20) at (2, -6) {6};
		\node [style=black] (21) at (-0.25, -5.25) {6};
		\node [style=black] (23) at (0.5, -6) {6};
		\node [style=bruh] (30) at (2.75, -6.75) {$\tinygeq7$};
		\node [style=black] (34) at (-0.25, -6.75) {6};
		\node [style=black] (35) at (1.25, -6.75) {6};
		\node [style=bruh] (37) at (2, -7.5) {6};
		\node [style=black] (39) at (0.5, -7.5) {6};
		\node [style=bruh] (41) at (2.75, -8.25) {6};
		\node [style=black] (42) at (1.25, -8.25) {6};
		\node [style=black] (43) at (2, -9) {6};
		\node [style=bruh] (47) at (3.5, -7.5) {$\tinygeq7$};
		\node [style=none] (52) at (-2, -9.25) {tb};
		\node [style=none] (53) at (-1.25, -4.5) {-8};
		\node [style=none] (54) at (-1.25, -5.25) {-9};
		\node [style=none] (55) at (-1.25, -6) {-10};
		\node [style=none] (56) at (-1.25, -6.75) {-11};
		\node [style=none] (57) at (-1.25, -7.5) {-12};
		\node [style=none] (58) at (-1.25, -8.25) {-13};
		\node [style=none] (59) at (-1.25, -9) {-14};
		\node [style=none] (60) at (-1.25, -9.75) {-15};
		\node [style=none] (61) at (-1.25, -10.5) {-16};
		\node [style=none] (62) at (-0.25, -3.75) {0};
		\node [style=none] (63) at (0.5, -3.75) {1};
		\node [style=none] (64) at (1.25, -3.75) {2};
		\node [style=none] (65) at (2, -3.75) {3};
		\node [style=none] (66) at (2.75, -3.75) {4};
		\node [style=none] (67) at (2, -3.25) {rot};
		\node [style=none] (73) at (3.5, -3.75) {5};
		\node [style=black] (77) at (-0.25, -8.25) {6};
		\node [style=black] (78) at (-0.25, -12.75) {6};
		\node [style=black] (79) at (2.75, -12.75) {$\tinygeq7$};
		\node [style=black] (80) at (1.25, -11.25) {6};
		\node [style=bruh] (81) at (2, -12) {$\tinygeq7$};
		\node [style=black] (82) at (0.5, -12) {6};
		\node [style=bruh] (83) at (1.25, -12.75) {6};
		\node [style=black] (84) at (3.5, -12) {$\tinygeq7$};
		\node [style=black] (85) at (0.5, -9) {6};
		\node [style=black] (86) at (1.25, -9.75) {6};
		\node [style=black] (87) at (-0.25, -9.75) {6};
		\node [style=black] (88) at (0.5, -10.5) {6};
		\node [style=bruh] (89) at (3.5, -9) {$\tinygeq7$};
		\node [style=bruh] (91) at (2.75, -9.75) {6};
		\node [style=black] (96) at (2.75, -11.25) {$\tinygeq7$};
		\node [style=black] (97) at (3.5, -10.5) {$\tinygeq7$};
		\node [style=black] (99) at (2, -10.5) {6};
		\node [style=black] (100) at (-0.25, -11.25) {6};
		\node [style=black] (105) at (4.25, -8.25) {$\tinygeq7$};
		\node [style=none] (106) at (4.25, -3.75) {6};
		\node [style=black] (107) at (4.25, -9.75) {$\tinygeq7$};
		\node [style=black] (108) at (4.25, -11.25) {$\tinygeq7$};
		\node [style=black] (109) at (0.5, -13.5) {6};
		\node [style=black] (110) at (2, -13.5) {$\tinygeq7$};
		\node [style=none] (111) at (-1.25, -11.25) {-17};
		\node [style=none] (112) at (-1.25, -12) {-18};
		\node [style=none] (113) at (-1.25, -12.75) {-19};
		\node [style=none] (114) at (-1.25, -13.5) {-20};
		\node [style=none] (115) at (-1.25, -14.25) {-21};
		\node [style=black] (116) at (-0.25, -14.25) {$\tinygeq7$};
		\node [style=black] (117) at (1.25, -14.25) {$\tinygeq7$};
	\end{pgfonlayer}
	\begin{pgfonlayer}{edgelayer}
		\draw (18) to (19);
		\draw (18) to (21);
		\draw (19) to (23);
		\draw (21) to (23);
		\draw (23) to (34);
		\draw (23) to (35);
		\draw (34) to (39);
		\draw (19) to (20);
		\draw (20) to (35);
		\draw (20) to (30);
		\draw (35) to (39);
		\draw (35) to (37);
		\draw (30) to (37);
		\draw (39) to (42);
		\draw (37) to (42);
		\draw (37) to (41);
		\draw (42) to (43);
		\draw (41) to (43);
		\draw (30) to (47);
		\draw (47) to (41);
		\draw (80) to (82);
		\draw (80) to (81);
		\draw (82) to (83);
		\draw (81) to (83);
		\draw (85) to (87);
		\draw (85) to (86);
		\draw (87) to (88);
		\draw (86) to (88);
		\draw (89) to (91);
		\draw (77) to (39);
		\draw (85) to (77);
		\draw (85) to (42);
		\draw (86) to (43);
		\draw (89) to (41);
		\draw (43) to (91);
		\draw (97) to (91);
		\draw (91) to (99);
		\draw (99) to (86);
		\draw (88) to (80);
		\draw (80) to (99);
		\draw (99) to (96);
		\draw (96) to (81);
		\draw (96) to (97);
		\draw (84) to (96);
		\draw (81) to (79);
		\draw (82) to (78);
		\draw (100) to (82);
		\draw (100) to (88);
		\draw (105) to (89);
		\draw (47) to (105);
		\draw (89) to (107);
		\draw (107) to (97);
		\draw (97) to (108);
		\draw (108) to (84);
		\draw (79) to (84);
		\draw (83) to (110);
		\draw (110) to (79);
		\draw (78) to (109);
		\draw (109) to (83);
		\draw (109) to (116);
		\draw (109) to (117);
		\draw (117) to (110);
	\end{pgfonlayer}
\end{tikzpicture}
    \tikzstyle{black}=[fill=white, draw=black, shape=circle, minimum size=0.8cm]
\tikzstyle{bruh}=[fill=yellow, draw=black, shape=circle, minimum size=0.8cm]
\tikzstyle{none}=[fill=none, draw=none, shape=circle]

\pgfdeclarelayer{nodelayer}
\pgfdeclarelayer{edgelayer}
\pgfsetlayers{edgelayer,nodelayer}

\begin{tikzpicture}
	\begin{pgfonlayer}{nodelayer}
		\node [style=black] (21) at (-0.25, -4.5) {6};
		\node [style=black] (23) at (0.5, -5.25) {6};
		\node [style=black] (34) at (-0.25, -6) {6};
		\node [style=black] (35) at (1.25, -6) {6};
		\node [style=black] (37) at (2, -6.75) {6};
		\node [style=black] (39) at (0.5, -6.75) {6};
		\node [style=bruh] (41) at (2.75, -7.5) {$\tinygeq7$};
		\node [style=black] (42) at (1.25, -7.5) {6};
		\node [style=bruh] (43) at (2, -8.25) {6};
		\node [style=none] (52) at (-2, -8.25) {tb};
		\node [style=none] (53) at (-1.25, -4.5) {1};
		\node [style=none] (54) at (-1.25, -5.25) {0};
		\node [style=none] (55) at (-1.25, -6) {-1};
		\node [style=none] (56) at (-1.25, -6.75) {-2};
		\node [style=none] (57) at (-1.25, -7.5) {-3};
		\node [style=none] (58) at (-1.25, -8.25) {-4};
		\node [style=none] (59) at (-1.25, -9) {-5};
		\node [style=none] (60) at (-1.25, -9.75) {-6};
		\node [style=none] (61) at (-1.25, -10.5) {-7};
		\node [style=none] (62) at (-0.25, -3.75) {0};
		\node [style=none] (63) at (0.5, -3.75) {1};
		\node [style=none] (64) at (1.25, -3.75) {2};
		\node [style=none] (65) at (2, -3.75) {3};
		\node [style=none] (66) at (2.75, -3.75) {4};
		\node [style=none] (67) at (1.25, -3.25) {rot};
		\node [style=black] (77) at (-0.25, -7.5) {6};
		\node [style=black] (78) at (-0.25, -12) {$\tinygeq7$};
		\node [style=black] (79) at (2.75, -12) {$\tinygeq7$};
		\node [style=black] (80) at (1.25, -10.5) {6};
		\node [style=black] (81) at (2, -11.25) {$\tinygeq7$};
		\node [style=black] (82) at (0.5, -11.25) {6};
		\node [style=black] (83) at (1.25, -12) {$\tinygeq7$};
		\node [style=black] (85) at (0.5, -8.25) {6};
		\node [style=black] (86) at (1.25, -9) {6};
		\node [style=black] (87) at (-0.25, -9) {6};
		\node [style=black] (88) at (0.5, -9.75) {6};
		\node [style=bruh] (91) at (2.75, -9) {$\tinygeq7$};
		\node [style=black] (96) at (2.75, -10.5) {$\tinygeq7$};
		\node [style=bruh] (99) at (2, -9.75) {6};
		\node [style=black] (100) at (-0.25, -10.5) {6};
		\node [style=none] (111) at (-1.25, -11.25) {-8};
		\node [style=none] (112) at (-1.25, -12) {-9};
	\end{pgfonlayer}
	\begin{pgfonlayer}{edgelayer}
		\draw (21) to (23);
		\draw (23) to (34);
		\draw (23) to (35);
		\draw (34) to (39);
		\draw (35) to (39);
		\draw (35) to (37);
		\draw (39) to (42);
		\draw (37) to (42);
		\draw (37) to (41);
		\draw (42) to (43);
		\draw (41) to (43);
		\draw (80) to (82);
		\draw (80) to (81);
		\draw (82) to (83);
		\draw (81) to (83);
		\draw (85) to (87);
		\draw (85) to (86);
		\draw (87) to (88);
		\draw (86) to (88);
		\draw (77) to (39);
		\draw (85) to (77);
		\draw (85) to (42);
		\draw (86) to (43);
		\draw (43) to (91);
		\draw (91) to (99);
		\draw (99) to (86);
		\draw (88) to (80);
		\draw (80) to (99);
		\draw (99) to (96);
		\draw (96) to (81);
		\draw (81) to (79);
		\draw (82) to (78);
		\draw (100) to (82);
		\draw (100) to (88);
	\end{pgfonlayer}
\end{tikzpicture}
    }
    \caption{$5_2$ and $m(5_2)$ censuses}
    \label{fig:5_2}
\end{figure}

\begin{figure}[H]
    \centering
    \resizebox{\linewidth}{!}{
    \tikzstyle{black}=[fill=white, draw=black, shape=circle, minimum size=0.8cm]
\tikzstyle{bruh}=[fill=yellow, draw=black, shape=circle, minimum size=0.8cm]
\tikzstyle{none}=[fill=none, draw=none, shape=circle]

\pgfdeclarelayer{nodelayer}
\pgfdeclarelayer{edgelayer}
\pgfsetlayers{edgelayer,nodelayer}

\begin{tikzpicture}
	\begin{pgfonlayer}{nodelayer}
		\node [style=black] (14) at (-0.25, -2.25) {6};
		\node [style=black] (18) at (0.5, -3) {6};
		\node [style=black] (19) at (1.25, -3.75) {6};
		\node [style=bruh] (20) at (2, -4.5) {$\tinygeq7$};
		\node [style=black] (21) at (-0.25, -3.75) {6};
		\node [style=black] (23) at (0.5, -4.5) {6};
		\node [style=bruh] (30) at (2.75, -5.25) {$\tinygeq7$};
		\node [style=black] (34) at (-0.25, -5.25) {6};
		\node [style=bruh] (35) at (1.25, -5.25) {6};
		\node [style=bruh] (37) at (2, -6) {6};
		\node [style=black] (39) at (0.5, -6) {6};
		\node [style=black] (41) at (2.75, -6.75) {$\tinygeq7$};
		\node [style=black] (42) at (1.25, -6.75) {6};
		\node [style=bruh] (43) at (2, -7.5) {$\tinygeq7$};
		\node [style=black] (47) at (3.5, -6) {$\tinygeq7$};
		\node [style=none] (50) at (-1.25, -3.75) {-7};
		\node [style=none] (52) at (-2, -5.5) {tb};
		\node [style=none] (53) at (-1.25, -4.5) {-8};
		\node [style=none] (54) at (-1.25, -5.25) {-9};
		\node [style=none] (55) at (-1.25, -6) {-10};
		\node [style=none] (56) at (-1.25, -6.75) {-11};
		\node [style=none] (57) at (-1.25, -7.5) {-12};
		\node [style=none] (58) at (-1.25, -8.25) {-13};
		\node [style=none] (59) at (-1.25, -9) {-14};
		\node [style=none] (62) at (-0.25, -1.5) {0};
		\node [style=none] (63) at (0.5, -1.5) {1};
		\node [style=none] (64) at (1.25, -1.5) {2};
		\node [style=none] (65) at (2, -1.5) {3};
		\node [style=none] (66) at (2.75, -1.5) {4};
		\node [style=none] (67) at (2, -1) {rot};
		\node [style=none] (73) at (3.5, -1.5) {5};
		\node [style=black] (77) at (-0.25, -6.75) {6};
		\node [style=black] (85) at (0.5, -7.5) {6};
		\node [style=bruh] (86) at (1.25, -8.25) {6};
		\node [style=black] (87) at (-0.25, -8.25) {6};
		\node [style=black] (88) at (0.5, -9) {$\tinygeq7$};
		\node [style=bruh] (89) at (3.5, -7.5) {6};
		\node [style=black] (91) at (2.75, -8.25) {$\tinygeq7$};
		\node [style=black] (97) at (3.5, -9) {$\tinygeq7$};
		\node [style=black] (99) at (2, -9) {$\tinygeq7$};
		\node [style=none] (101) at (-1.25, -2.25) {-5};
		\node [style=none] (104) at (-1.25, -3) {-6};
		\node [style=bruh] (105) at (4.25, -6.75) {$\tinygeq7$};
		\node [style=none] (106) at (4.25, -1.5) {6};
		\node [style=black] (107) at (4.25, -8.25) {$\tinygeq7$};
	\end{pgfonlayer}
	\begin{pgfonlayer}{edgelayer}
		\draw (14) to (18);
		\draw (18) to (19);
		\draw (18) to (21);
		\draw (19) to (23);
		\draw (21) to (23);
		\draw (23) to (34);
		\draw (23) to (35);
		\draw (34) to (39);
		\draw (19) to (20);
		\draw (20) to (35);
		\draw (20) to (30);
		\draw (35) to (39);
		\draw (35) to (37);
		\draw (30) to (37);
		\draw (39) to (42);
		\draw (37) to (42);
		\draw (37) to (41);
		\draw (42) to (43);
		\draw (41) to (43);
		\draw (30) to (47);
		\draw (47) to (41);
		\draw (85) to (87);
		\draw (85) to (86);
		\draw (87) to (88);
		\draw (86) to (88);
		\draw (89) to (91);
		\draw (77) to (39);
		\draw (85) to (77);
		\draw (85) to (42);
		\draw (86) to (43);
		\draw (89) to (41);
		\draw (43) to (91);
		\draw (97) to (91);
		\draw (91) to (99);
		\draw (99) to (86);
		\draw (105) to (89);
		\draw (47) to (105);
		\draw (89) to (107);
		\draw (107) to (97);
	\end{pgfonlayer}
\end{tikzpicture}
    \tikzstyle{black}=[fill=white, draw=black, shape=circle, minimum size=0.8cm]
\tikzstyle{bruh}=[fill=yellow, draw=black, shape=circle, minimum size=0.8cm]
\tikzstyle{none}=[fill=none, draw=none, shape=circle]

\pgfdeclarelayer{nodelayer}
\pgfdeclarelayer{edgelayer}
\pgfsetlayers{edgelayer,nodelayer}

\begin{tikzpicture}
	\begin{pgfonlayer}{nodelayer}
		\node [style=black] (14) at (-0.25, -2.25) {6};
		\node [style=black] (18) at (0.5, -3) {6};
		\node [style=black] (19) at (1.25, -3.75) {6};
		\node [style=black] (20) at (2, -4.5) {6};
		\node [style=black] (21) at (-0.25, -3.75) {6};
		\node [style=black] (23) at (0.5, -4.5) {6};
		\node [style=black] (30) at (2.75, -5.25) {6};
		\node [style=black] (34) at (-0.25, -5.25) {6};
		\node [style=black] (35) at (1.25, -5.25) {6};
		\node [style=black] (37) at (2, -6) {6};
		\node [style=black] (39) at (0.5, -6) {6};
		\node [style=bruh] (41) at (2.75, -6.75) {6};
		\node [style=black] (42) at (1.25, -6.75) {6};
		\node [style=black] (43) at (2, -7.5) {$\tinygeq7$};
		\node [style=bruh] (47) at (3.5, -6) {$\tinygeq7$};
		\node [style=none] (50) at (-1.25, -5.25) {-7};
		\node [style=none] (52) at (-2, -4.5) {tb};
		\node [style=none] (53) at (-1.25, -6) {-8};
		\node [style=none] (54) at (-1.25, -6.75) {-9};
		\node [style=none] (55) at (-1.25, -7.5) {-10};
		\node [style=none] (62) at (-0.25, -1.5) {0};
		\node [style=none] (63) at (0.5, -1.5) {1};
		\node [style=none] (64) at (1.25, -1.5) {2};
		\node [style=none] (65) at (2, -1.5) {3};
		\node [style=none] (66) at (2.75, -1.5) {4};
		\node [style=none] (67) at (1.5, -1) {rot};
		\node [style=none] (73) at (3.5, -1.5) {5};
		\node [style=black] (77) at (-0.25, -6.75) {$\tinygeq7$};
		\node [style=black] (85) at (0.5, -7.5) {$\tinygeq7$};
		\node [style=black] (89) at (3.5, -7.5) {$\tinygeq7$};
		\node [style=none] (101) at (-1.25, -3.75) {-5};
		\node [style=none] (104) at (-1.25, -4.5) {-6};
		\node [style=none] (107) at (-1.25, -3) {-4};
		\node [style=none] (108) at (-1.25, -2.25) {-3};
	\end{pgfonlayer}
	\begin{pgfonlayer}{edgelayer}
		\draw (14) to (18);
		\draw (18) to (19);
		\draw (18) to (21);
		\draw (19) to (23);
		\draw (21) to (23);
		\draw (23) to (34);
		\draw (23) to (35);
		\draw (34) to (39);
		\draw (19) to (20);
		\draw (20) to (35);
		\draw (20) to (30);
		\draw (35) to (39);
		\draw (35) to (37);
		\draw (30) to (37);
		\draw (39) to (42);
		\draw (37) to (42);
		\draw (37) to (41);
		\draw (42) to (43);
		\draw (41) to (43);
		\draw (30) to (47);
		\draw (47) to (41);
		\draw (77) to (39);
		\draw (85) to (77);
		\draw (85) to (42);
		\draw (89) to (41);
	\end{pgfonlayer}
\end{tikzpicture}
    }
    \caption{$6_1$ and $m(6_1)$ censuses}
    \label{fig:6_1}
\end{figure}

\begin{figure}[H]
    \centering
    \scalebox{0.7}{
    %\resizebox{\linewidth}{!}{
    \tikzstyle{black}=[fill=white, draw=black, shape=circle, minimum size=0.8cm]
\tikzstyle{bruh}=[fill=yellow, draw=black, shape=circle, minimum size=0.8cm]
\tikzstyle{none}=[fill=none, draw=none, shape=circle]

\pgfdeclarelayer{nodelayer}
\pgfdeclarelayer{edgelayer}
\pgfsetlayers{edgelayer,nodelayer}

\begin{tikzpicture}
	\begin{pgfonlayer}{nodelayer}
		\node [style=black] (23) at (0.5, -2.25) {6};
		\node [style=black] (34) at (-0.25, -3) {6};
		\node [style=bruh] (35) at (1.25, -3) {6};
		\node [style=bruh] (37) at (2, -3.75) {$\tinygeq7$};
		\node [style=black] (39) at (0.5, -3.75) {6};
		\node [style=bruh] (42) at (1.25, -4.5) {6};
		\node [style=black] (43) at (2, -5.25) {$\tinygeq7$};
		\node [style=none] (52) at (-2, -4) {tb};
		\node [style=none] (56) at (-1.25, -2.25) {-14};
		\node [style=none] (57) at (-1.25, -3) {-15};
		\node [style=none] (58) at (-1.25, -3.75) {-16};
		\node [style=none] (59) at (-1.25, -4.5) {-17};
		\node [style=none] (60) at (-1.25, -5.25) {-18};
		\node [style=none] (61) at (-1.25, -6) {-19};
		\node [style=none] (62) at (-0.25, -1.5) {0};
		\node [style=none] (63) at (0.5, -1.5) {1};
		\node [style=none] (64) at (1.25, -1.5) {2};
		\node [style=none] (65) at (2, -1.5) {3};
		\node [style=none] (67) at (1.5, -1) {rot};
		\node [style=black] (77) at (-0.25, -4.5) {6};
		\node [style=black] (85) at (0.5, -5.25) {6};
		\node [style=black] (86) at (1.25, -6) {$\tinygeq7$};
		\node [style=black] (87) at (-0.25, -6) {$\tinygeq7$};
		\node [style=none] (101) at (2.75, -1.5) {4};
		\node [style=black] (102) at (2.75, -4.5) {$\tinygeq7$};
		\node [style=none] (103) at (3.5, -1.5) {5};
		\node [style=bruh] (104) at (2, -2.25) {7-9};
		\node [style=black] (105) at (3.5, -2.25) {7-9};
		\node [style=black] (106) at (2.75, -3) {$\tinygeq7$};
		\node [style=black] (107) at (2.75, -6) {$\tinygeq7$};
		\node [style=black] (108) at (3.5, -3.75) {$\tinygeq7$};
		\node [style=black] (109) at (3.5, -5.25) {$\tinygeq7$};
	\end{pgfonlayer}
	\begin{pgfonlayer}{edgelayer}
		\draw (23) to (34);
		\draw (23) to (35);
		\draw (34) to (39);
		\draw (35) to (39);
		\draw (35) to (37);
		\draw (39) to (42);
		\draw (37) to (42);
		\draw (42) to (43);
		\draw (85) to (87);
		\draw (85) to (86);
		\draw (77) to (39);
		\draw (85) to (77);
		\draw (85) to (42);
		\draw (86) to (43);
		\draw (37) to (102);
		\draw (102) to (43);
		\draw (35) to (104);
		\draw (37) to (106);
		\draw (104) to (106);
		\draw (106) to (105);
		\draw (106) to (108);
		\draw (108) to (102);
		\draw (43) to (107);
		\draw (109) to (102);
		\draw (109) to (107);
	\end{pgfonlayer}
\end{tikzpicture}
    \tikzstyle{black}=[fill=white, draw=black, shape=circle, minimum size=0.8cm]
\tikzstyle{bruh}=[fill=yellow, draw=black, shape=circle, minimum size=0.8cm]
\tikzstyle{none}=[fill=none, draw=none, shape=circle]

\pgfdeclarelayer{nodelayer}
\pgfdeclarelayer{edgelayer}
\pgfsetlayers{edgelayer,nodelayer}

\begin{tikzpicture}
	\begin{pgfonlayer}{nodelayer}
		\node [style=black] (34) at (-0.25, -2.25) {6};
		\node [style=black] (39) at (0.5, -3) {6};
		\node [style=black] (42) at (1.25, -3.75) {6};
		\node [style=bruh] (43) at (2, -4.5) {$\tinygeq7$};
		\node [style=none] (52) at (-2, -5) {tb};
		\node [style=none] (56) at (-1.25, -2.25) {5};
		\node [style=none] (57) at (-1.25, -3) {4};
		\node [style=none] (58) at (-1.25, -3.75) {3};
		\node [style=none] (59) at (-1.25, -4.5) {2};
		\node [style=none] (60) at (-1.25, -5.25) {1};
		\node [style=none] (61) at (-1.25, -6) {0};
		\node [style=none] (62) at (-0.25, -1.5) {0};
		\node [style=none] (63) at (0.5, -1.5) {1};
		\node [style=none] (64) at (1.25, -1.5) {2};
		\node [style=none] (65) at (2, -1.5) {3};
		\node [style=none] (67) at (1.25, -1) {rot};
		\node [style=black] (77) at (-0.25, -3.75) {6};
		\node [style=black] (85) at (0.5, -4.5) {6};
		\node [style=bruh] (86) at (1.25, -5.25) {6};
		\node [style=black] (87) at (-0.25, -5.25) {6};
		\node [style=none] (101) at (2.75, -1.5) {4};
		\node [style=black] (107) at (2.75, -5.25) {$\tinygeq7$};
		\node [style=none] (110) at (-1.25, -6.75) {-1};
		\node [style=none] (111) at (-1.25, -7.5) {-2};
		\node [style=black] (112) at (0.5, -6) {6};
		\node [style=black] (113) at (2, -6) {$\tinygeq7$};
		\node [style=black] (114) at (-0.25, -6.75) {6};
		\node [style=black] (115) at (1.25, -6.75) {$\tinygeq7$};
		\node [style=black] (116) at (0.5, -7.5) {$\tinygeq7$};
		\node [style=black] (117) at (2, -7.5) {$\tinygeq7$};
		\node [style=black] (118) at (2.75, -6.75) {$\tinygeq7$};
	\end{pgfonlayer}
	\begin{pgfonlayer}{edgelayer}
		\draw (34) to (39);
		\draw (39) to (42);
		\draw (42) to (43);
		\draw (85) to (87);
		\draw (85) to (86);
		\draw (77) to (39);
		\draw (85) to (77);
		\draw (85) to (42);
		\draw (86) to (43);
		\draw (43) to (107);
		\draw (87) to (112);
		\draw (112) to (86);
		\draw (86) to (113);
		\draw (113) to (107);
		\draw (113) to (118);
		\draw (118) to (117);
		\draw (117) to (115);
		\draw (115) to (113);
		\draw (112) to (115);
		\draw (115) to (116);
		\draw (116) to (114);
		\draw (114) to (112);
	\end{pgfonlayer}
\end{tikzpicture}
    }
    \caption{$7_1$ and $m(7_1)$ censuses}
    \label{fig:7_1}
\end{figure}

\vspace{-0.5cm}

\begin{figure}[H]
    \centering
    \scalebox{0.7}{
    %\resizebox{\linewidth}{!}{
    \tikzstyle{black}=[fill=white, draw=black, shape=circle, minimum size=0.8cm]
\tikzstyle{bruh}=[fill=yellow, draw=black, shape=circle, minimum size=0.8cm]
\tikzstyle{none}=[fill=none, draw=none, shape=circle]

\pgfdeclarelayer{nodelayer}
\pgfdeclarelayer{edgelayer}
\pgfsetlayers{edgelayer,nodelayer}

\begin{tikzpicture}
	\begin{pgfonlayer}{nodelayer}
		\node [style=black] (18) at (0.5, -6) {6};
		\node [style=black] (19) at (1.25, -6.75) {6};
		\node [style=bruh] (20) at (2, -7.5) {$\tinygeq7$};
		\node [style=black] (21) at (-0.25, -6.75) {6};
		\node [style=black] (23) at (0.5, -7.5) {6};
		\node [style=black] (34) at (-0.25, -8.25) {6};
		\node [style=bruh] (35) at (1.25, -8.25) {6};
		\node [style=bruh] (37) at (2, -9) {$\tinygeq7$};
		\node [style=black] (39) at (0.5, -9) {6};
		\node [style=bruh] (42) at (1.25, -9.75) {6};
		\node [style=black] (43) at (2, -10.5) {$\tinygeq7$};
		\node [style=none] (50) at (-1.25, -11.25) {-17};
		\node [style=none] (52) at (-2, -9.25) {tb};
		\node [style=none] (53) at (-1.25, -12) {-18};
		\node [style=none] (54) at (-1.25, -12.75) {-19};
		\node [style=none] (55) at (-1.25, -6) {-10};
		\node [style=none] (56) at (-1.25, -6.75) {-11};
		\node [style=none] (57) at (-1.25, -7.5) {-12};
		\node [style=none] (58) at (-1.25, -8.25) {-13};
		\node [style=none] (59) at (-1.25, -9) {-14};
		\node [style=none] (62) at (-0.25, -5.25) {0};
		\node [style=none] (63) at (0.5, -5.25) {1};
		\node [style=none] (64) at (1.25, -5.25) {2};
		\node [style=none] (65) at (2, -5.25) {3};
		\node [style=none] (67) at (0.75, -4.75) {rot};
		\node [style=black] (77) at (-0.25, -9.75) {6};
		\node [style=black] (85) at (0.5, -10.5) {6};
		\node [style=bruh] (86) at (1.25, -11.25) {$\tinygeq7$};
		\node [style=black] (87) at (-0.25, -11.25) {6};
		\node [style=bruh] (88) at (0.5, -12) {6};
		\node [style=black] (99) at (2, -12) {$\tinygeq7$};
		\node [style=none] (101) at (-1.25, -9.75) {-15};
		\node [style=none] (104) at (-1.25, -10.5) {-16};
		\node [style=black] (108) at (-0.25, -12.75) {$\tinygeq7$};
		\node [style=black] (109) at (1.25, -12.75) {$\tinygeq7$};
	\end{pgfonlayer}
	\begin{pgfonlayer}{edgelayer}
		\draw (18) to (19);
		\draw (18) to (21);
		\draw (19) to (23);
		\draw (21) to (23);
		\draw (23) to (34);
		\draw (23) to (35);
		\draw (34) to (39);
		\draw (19) to (20);
		\draw (20) to (35);
		\draw (35) to (39);
		\draw (35) to (37);
		\draw (39) to (42);
		\draw (37) to (42);
		\draw (42) to (43);
		\draw (85) to (87);
		\draw (85) to (86);
		\draw (87) to (88);
		\draw (86) to (88);
		\draw (77) to (39);
		\draw (85) to (77);
		\draw (85) to (42);
		\draw (86) to (43);
		\draw (99) to (86);
		\draw (88) to (108);
		\draw (109) to (99);
		\draw (88) to (109);
	\end{pgfonlayer}
\end{tikzpicture}
    \tikzstyle{black}=[fill=white, draw=black, shape=circle, minimum size=0.8cm]
\tikzstyle{bruh}=[fill=yellow, draw=black, shape=circle, minimum size=0.8cm]
\tikzstyle{none}=[fill=none, draw=none, shape=circle]

\pgfdeclarelayer{nodelayer}
\pgfdeclarelayer{edgelayer}
\pgfsetlayers{edgelayer,nodelayer}

\begin{tikzpicture}
	\begin{pgfonlayer}{nodelayer}
		\node [style=black] (34) at (-0.25, -2.25) {6};
		\node [style=black] (39) at (0.5, -3) {6};
		\node [style=bruh] (42) at (1.25, -3.75) {$\tinygeq7$};
		\node [style=none] (52) at (-2, -4) {tb};
		\node [style=none] (57) at (-1.25, -6) {-4};
		\node [style=none] (58) at (-1.25, -5.25) {-3};
		\node [style=none] (60) at (-1.25, -2.25) {1};
		\node [style=none] (61) at (-1.25, -3) {0};
		\node [style=none] (62) at (-0.25, -1.5) {0};
		\node [style=none] (63) at (0.5, -1.5) {1};
		\node [style=none] (64) at (1.25, -1.5) {2};
		\node [style=none] (67) at (0.5, -1) {rot};
		\node [style=black] (77) at (-0.25, -3.75) {6};
		\node [style=bruh] (85) at (0.5, -4.5) {6};
		\node [style=black] (86) at (1.25, -5.25) {$\tinygeq7$};
		\node [style=black] (87) at (-0.25, -5.25) {6};
		\node [style=none] (110) at (-1.25, -3.75) {-1};
		\node [style=none] (111) at (-1.25, -4.5) {-2};
		\node [style=black] (112) at (0.5, -6) {$\tinygeq7$};
	\end{pgfonlayer}
	\begin{pgfonlayer}{edgelayer}
		\draw (34) to (39);
		\draw (39) to (42);
		\draw (85) to (87);
		\draw (85) to (86);
		\draw (77) to (39);
		\draw (85) to (77);
		\draw (85) to (42);
		\draw (87) to (112);
		\draw (112) to (86);
	\end{pgfonlayer}
\end{tikzpicture}
    }
    \caption{$7_2$ and $m(7_2)$ censuses}
    \label{fig:7_2}
\end{figure}
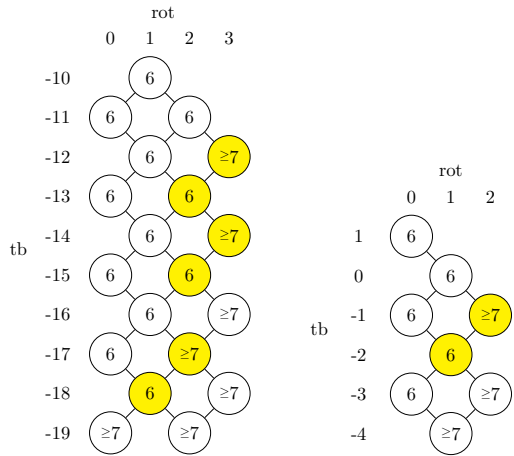

\vspace{-0.5cm}

\begin{figure}[H]
    \centering
    %\resizebox{\linewidth}{!}{
    \scalebox{0.7}{
    \tikzstyle{black}=[fill=white, draw=black, shape=circle, minimum size=0.8cm]
\tikzstyle{bruh}=[fill=yellow, draw=black, shape=circle, minimum size=0.8cm]
\tikzstyle{none}=[fill=none, draw=none, shape=circle]

\pgfdeclarelayer{nodelayer}
\pgfdeclarelayer{edgelayer}
\pgfsetlayers{edgelayer,nodelayer}

\begin{tikzpicture}
	\begin{pgfonlayer}{nodelayer}
		\node [style=black] (21) at (-0.25, -2.25) {7};
		\node [style=bruh] (23) at (0.5, -3) {7};
		\node [style=bruh] (34) at (-0.25, -3.75) {6};
		\node [style=bruh] (35) at (1.25, -3.75) {$\tinygeq7$};
		\node [style=bruh] (39) at (0.5, -4.5) {6};
		\node [style=black] (42) at (1.25, -5.25) {$\tinygeq7$};
		\node [style=none] (52) at (-2, -4) {tb};
		\node [style=none] (56) at (-1.25, -2.25) {-7};
		\node [style=none] (57) at (-1.25, -3) {-8};
		\node [style=none] (58) at (-1.25, -3.75) {-9};
		\node [style=none] (59) at (-1.25, -4.5) {-10};
		\node [style=none] (60) at (-1.25, -5.25) {-11};
		\node [style=none] (61) at (-1.25, -6) {-12};
		\node [style=none] (62) at (-0.25, -1.5) {0};
		\node [style=none] (63) at (0.5, -1.5) {1};
		\node [style=none] (64) at (1.25, -1.5) {2};
		\node [style=none] (67) at (0.5, -1) {rot};
		\node [style=black] (77) at (-0.25, -5.25) {6};
		\node [style=black] (85) at (0.5, -6) {$\tinygeq7$};
		\node [style=none] (103) at (1.25, -2.25) {?};
	\end{pgfonlayer}
	\begin{pgfonlayer}{edgelayer}
		\draw (21) to (23);
		\draw (23) to (34);
		\draw (23) to (35);
		\draw (34) to (39);
		\draw (35) to (39);
		\draw (39) to (42);
		\draw (77) to (39);
		\draw (85) to (77);
		\draw (85) to (42);
	\end{pgfonlayer}
\end{tikzpicture}
    \tikzstyle{black}=[fill=white, draw=black, shape=circle, minimum size=0.8cm]
\tikzstyle{bruh}=[fill=yellow, draw=black, shape=circle, minimum size=0.8cm]
\tikzstyle{none}=[fill=none, draw=none, shape=circle]

\pgfdeclarelayer{nodelayer}
\pgfdeclarelayer{edgelayer}
\pgfsetlayers{edgelayer,nodelayer}

\begin{tikzpicture}
	\begin{pgfonlayer}{nodelayer}
		\node [style=black] (21) at (-0.25, -2.25) {6};
		\node [style=black] (23) at (0.5, -3) {6};
		\node [style=black] (34) at (-0.25, -3.75) {$\tinygeq7$};
		\node [style=black] (35) at (1.25, -3.75) {$\tinygeq7$};
		\node [style=none] (52) at (-2, -3) {tb};
		\node [style=none] (56) at (-1.25, -2.25) {-3};
		\node [style=none] (57) at (-1.25, -3) {-4};
		\node [style=none] (58) at (-1.25, -3.75) {-5};
		\node [style=none] (62) at (-0.25, -1.5) {0};
		\node [style=none] (63) at (0.5, -1.5) {1};
		\node [style=none] (64) at (1.25, -1.5) {2};
		\node [style=none] (67) at (0.5, -1) {rot};
		\node [style=none] (103) at (1.25, -2.25) {?};
	\end{pgfonlayer}
	\begin{pgfonlayer}{edgelayer}
		\draw (21) to (23);
		\draw (23) to (34.center);
		\draw (23) to (35.center);
	\end{pgfonlayer}
\end{tikzpicture}
    }
    \caption{$8_1$ and $m(8_1)$ censuses}
    \label{fig:8_1}
\end{figure}

%\end{document}

\setcounter{table}{0}

\twocolumn[
\section{Minimal Mosaics}\label{min-mosaics}
In this section, we provide computationally derived mosaics demonstrating the upper bounds shown in Appendix \ref{census-list}. The mosaics listed are of the minimal size necessary to depict their respective Legendrian knots, and follow the format established in Section \ref{sec:census_results}. These mosaics can be converted into images using the python program \texttt{to\_image.py}, which can be found at \cite{code}. 
\vspace{20pt}
]

\tablecaption{Minimal Mosaics for Legendrian knots $\Lambda$ with $m(\Lambda) \le 6$.}

\setlength{\tabcolsep}{4pt}
\tiny
\tablehead{
\hline
\cellcolor{black!15}tb    & \cellcolor{black!15}rot  & \cellcolor{black!15}Minimal Mosaic \\
\hline \\[-4pt]
}
%\tabletail{\hline}
\begin{supertabular}[h]{c c r}
%\hline
%\endfoot
%\hline
%tb    & rot  & Minimal Mosaic                        \\
%\endhead
\hline
\multicolumn{3}{c}{$0_1$}                 \\[2pt]
\hline \\[-3pt]
-1    & 0  & 2134                                  \\
-2    & 1  & 021246354                             \\
-3    & 0  & 021284340                             \\
-3    & 2  & 0021024624063554                      \\
-4    & 1  & 0000021028913434                      \\
-4    & 3  & 0002100246024062400635554             \\
-5    & 0  & 0021028428403400                      \\
-5    & 2  & 0210279139840340                      \\
-5    & 4  & 0212128746399162434635554             \\
-6    & 1  & 0021028428913434                      \\
-6    & 3  & 0002100246028162439435540             \\
-6    & 5  & 000000255510621281389846243406355554  \\
-7    & 0  & 2121398428913434                      \\
-7    & 2  & 0000000210027912898434340             \\
-7    & 4  & 000000000021021246279916398846034354  \\
-7    & 6  & 002121028746284606395916240346355554  \\
-8    & 1  & 0000000210028812898434340             \\
-8    & 3  & 0021002791279843984003400             \\
-8    & 5  & 000021000246002816028466243594355540  \\
-9    & 0  & 0000002121288843989103434             \\
-9    & 2  & 0002100284027912898434340             \\
-9    & 4  & 000000000210002871028794243991355434  \\
-9    & 6  & 002121028746287916399846243406355554  \\
-10   & 1  & 0002100284028812898434340             \\
-10   & 3  & 0212127984399712878434340             \\
-10   & 5  & 000021000246002816028794243991355434  \\
-11   & 0  & 0002121284398812898434340             \\
-11   & 2  & 0212127984398712888434340             \\
-11   & 4  & 000000002121027984279971398784034340  \\
-11   & 6  & 002121028746287954379951246624354340  \\
-12   & 1  & 0212127984398812898434340             \\
-12   & 3  & 000000002121027984279871398884034340  \\
-12   & 5  & 000021021246287816379794289791343434  \\
-13   & 0  & 0212128884398812898434340             \\
-13   & 2  & 000000002121027984279881398984034340  \\
-13   & 4  & 000021021246279916398884288891343434  \\
-13   & 6  & 021210287791379984279971398984034340  \\
-14   & 1  & 000000000210212881398884289891343434  \\
-14   & 3  & 000021021246279716398884289891343434  \\
-14   & 5  & 002121027984279971399784287840343400  \\
-15   & 0  & 000000002121028884288881398984034340  \\
-15   & 2  & 000000021210279881398884289891343434  \\
-15   & 4  & 002121027984279871398894288891343434  \\
-16   & 1  & 000000021210288881398884289891343434  \\
-16   & 3  & 000251021624279881398884289891343434  \\
-16   & 5  & 212121398846288816398794243991355434  \\
-17   & 0  & 000021021284287881388884289891343434  \\
-17   & 2  & 000021021284279881398884289891343434  \\
-17   & 4  & 212121397984297881638884629891343434  \\
-18   & 1  & 000021021284288881398884289891343434  \\
-18   & 3  & 002121027984279881398884289891343434  \\
-19   & 0  & 002121028784288881388884289891343434  \\
-19   & 2  & 002121027984288881398884289891343434  \\
-20   & 1  & 002121028884288881398884289891343434  \\
-21   & 0  & 212121398884288881398884289891343434  \\
\\[-5pt]

\hline
\multicolumn{3}{c}{$3_1$}     \\
\hline \\[-3pt]
-6    & 1  & 0021002971294663759403540             \\
-7    & 0  & 0021002971298463791603434             \\
-7    & 2  & 0210028910399710379400340             \\
-8    & 1  & 0021002871298943794003400             \\
-8    & 3  & 000000002100029710243991355984000340  \\
-9    & 0  & 0021002971298843799103434             \\
-9    & 2  & 0021002871298943799103434             \\
-9    & 4  & 000210002460029160243991355984000340  \\
-10   & 1  & 0021025891629843989103434             \\
-10   & 3  & 0021002891289843984003400             \\
-11   & 0  & 000000000210002971029884287991343434  \\
-11   & 2  & 0021002891289843989103434             \\
-11   & 4  & 000000002100028710288991379984034340  \\
-12   & 1  & 2121039891289843989103434             \\
-12   & 3  & 000000000210002891028984289840343400  \\
-12   & 5  & 000210002460028160288991379984034340  \\
-13   & 0  & 000000002121029784298881379984034340  \\
-13   & 2  & 000000000210002891028984289891343434  \\
-13   & 4  & 000000002121028784288991379984034340  \\
-14   & 1  & 000000000210212891398984289891343434  \\
-14   & 3  & 000000002121028984289871398884034340  \\
-14   & 5  & 000210002791028984289971398784034340  \\
-15   & 0  & 000000021210279891398984289891343434  \\
-15   & 2  & 000000002121028984289881398984034340  \\
-15   & 4  & 000021021284287891379984289840343400  \\
-15   & 6  & 021210287791379984289971398784034340  \\
-16   & 1  & 000000021210288891398984289891343434  \\
-16   & 3  & 000021021284287891379984289891343434  \\
-16   & 5  & 002121027984289971399884287840343400  \\
-17   & 0  & 000021021284279891398984289891343434  \\
-17   & 2  & 000021021284287891388984289891343434  \\
-17   & 4  & 002121027984289971399884287891343434  \\
-18   & 1  & 000021021284288891398984289891343434  \\
-18   & 3  & 002100028910289881398884289891343434  \\
-19   & 0  & 002121027984288891398984289891343434  \\
-19   & 2  & 002121028784288891388984289891343434  \\
-20   & 1  & 002121028884288891398984289891343434  \\
-20   & 3  & 002121028984289881398884289891343434  \\
-21   & 2  & 212121398884288891398984289891343434  \\
\\[-5pt]

\hline
\multicolumn{3}{c}{$m(3_1)$}                     \\
\hline \\[-3pt]
1     & 0  & 0021025971629943943103554             \\
0     & 1  & 0021002891299463791603434             \\
-1    & 0  & 0021002891299843794003400             \\
-1    & 2  & 0210028910397912434635554             \\
-2    & 1  & 0021002891299843799103434             \\
-2    & 3  & 000000002100028910284391395546035554  \\
-3    & 0  & 0210028710397912898434340             \\
-3    & 2  & 0212128946397162894634354             \\
-3    & 4  & 000000021210289891397466243546355554  \\
-4    & 1  & 0212128784397912894634354             \\
-4    & 3  & 000000002121028946289716397946034354  \\
-4    & 5  & 002121028946299916639746628406343554  \\
-5    & 0  & 0212128784397912898434340             \\
-5    & 2  & 000000000210212891399984287991343434  \\
-5    & 4  & 000021021246289816397466289546343554  \\
-6    & 1  & 000000002121028784279791398984034340  \\
-6    & 3  & 000000021210289891397974243891355434  \\
-6    & 5  & 002121259746628816639794627991343434  \\
-7    & 0  & 000000002121028784288791398984034340  \\
-7    & 2  & 000000021210287891397974289891343434  \\
-7    & 4  & 000210212791399984279791398984034340  \\
-8    & 1  & 000000021210287881397984289891343434  \\
-8    & 3  & 000021021246289816397794289891343434  \\
-8    & 5  & 021210279791399984279791398984034340  \\
-9    & 0  & 000021021284287791397984289891343434  \\
-9    & 2  & 000021021246289816397884289891343434  \\
-9    & 4  & 002121028946288916397994289891343434  \\
-10   & 1  & 000021021284287881397984289891343434  \\
-10   & 3  & 002121028784287991397974289891343434  \\
-11   & 0  & 002121027984287881397984289891343434  \\
-11   & 2  & 002121028784287981397984289891343434  \\
-12   & 1  & 002121028784288791398884289891343434  \\
-13   & 0  & 212121398784288791398884289891343434  \\
\\[-5pt]

\hline
\multicolumn{3}{c}{$3_1\#3_1$} \\
\hline \\[-3pt]
-11   & 0  & 255100602910629891398946039406003554  \\
-12   & 1  & 002100258910629891398946039406003554  \\
-13   & 0  & 002100258910629891398946039424003540  \\
-13   & 2  & 002100028910289891398946039406003554  \\
-14   & 1  & 002100028910289891398946039424003540  \\
-15   & 0  & 002100028910289891398984039840003400  \\
-15   & 2  & 002100028910289891398946289406343554  \\
-16   & 1  & 002100028910289891398946289424343540  \\
-17   & 0  & 002100028910289891398984289840343400  \\
-17   & 2  & 002121028984289891398946289406343554  \\
-18   & 1  & 002100028910289891398984289891343434  \\
-19   & 0  & 002121028984289891398984289840343400  \\
-20   & 1  & 002121028984289891398984289891343434  \\
-21   & 0  & 212121398984289891398984289891343434  \\
\\[-5pt]

\hline 
\multicolumn{3}{c}{$m(3_1)\#m(3_1)$}                  \\
\hline \\[-3pt]
3     & 0  & 002100029710297971379794037940003400  \\
2     & 1  & 002100029710297971379794037991003434  \\
1     & 0  & 002100029710297971379794243991355434  \\
1     & 2  & 002100029710297971639794627991343434  \\
0     & 1  & 002100029710297971379794287991343434  \\
0     & 3  & 002551029124297971639794627991343434  \\
-1    & 0  & 002121029746297916379794287991343434  \\
-1    & 2  & 002121029784297971639794627991343434  \\
-2    & 1  & 002121029784297971379794287991343434  \\
-3    & 0  & 212121399784297971379794287991343434  \\
\\[-5pt]

\hline 
\multicolumn{3}{c}{$4_1$}                            \\
\hline \\[-3pt]
-3    & 0  & 000000021210297971379466037594003540  \\
-4    & 1  & 000000002100028910299971379794034340  \\
-5    & 0  & 000000002100028710298971379794034340  \\
-5    & 2  & 000000002100028910289791394346035554  \\
-6    & 1  & 000000002100028710289791398946034354  \\
-6    & 3  & 000000002100029710289991398984034340  \\
-7    & 0  & 000000002100028910289791398784034340  \\
-7    & 2  & 000000002100028710289791398984034340  \\
-7    & 4  & 000000002551029124289991398984034340  \\
-8    & 1  & 000000002121028784289791398946034354  \\
-8    & 3  & 000000002121029784289991398984034340  \\
-8    & 5  & 000210212891399984063991062984034340  \\
-9    & 0  & 000000002121028984289791398784034340  \\
-9    & 2  & 000000002121028784289791398984034340  \\
-9    & 4  & 000021021284289891399984243940355400  \\
-10   & 1  & 000000021210287891397984289891343434  \\
-10   & 3  & 000021021284289891399984243991355434  \\
-10   & 5  & 000210212891399984277991398984034340  \\
-11   & 0  & 000021021284287891397984289840343400  \\
-11   & 2  & 000021021284289891377984289840343400  \\
-11   & 4  & 000210212891399984287991388984034340  \\
-11   & 6  & 021210289791399984289991398984034340  \\
-12   & 1  & 000021021284287891397984289891343434  \\
-12   & 3  & 002121027984289891377984289840343400  \\
-12   & 5  & 021210288891399984277991398984034340  \\
-13   & 0  & 002121027984287891397984289891343434  \\
-13   & 2  & 002121027984289891377984289891343434  \\
-13   & 4  & 021210288891399984287991388984034340  \\
-14   & 1  & 002121028784288791398984289891343434  \\
-15   & 0  & 212121398784288791398984289891343434  \\
\\[-5pt]

\hline 
\multicolumn{3}{c}{$5_1$}                            \\
\hline \\[-3pt]
-10   & 1  & 000000002100029710298991379946034354  \\
-11   & 0  & 000000002551029124298991379946034354  \\
-11   & 2  & 000000002100029710298991379984034340  \\
-12   & 1  & 000000002121029784298991379946034354  \\
-12   & 3  & 000210002460029160298991379984034340  \\
-13   & 0  & 000000021210289971399894243991355434  \\
-13   & 2  & 000000002121029784298991379984034340  \\
-14   & 1  & 000000021210289971399894287991343434  \\
-14   & 3  & 000210002871028894298991379946034354  \\
-15   & 0  & 000021021246289916399894287991343434  \\
-15   & 2  & 000021021284289971399894287940343400  \\
-15   & 4  & 000210002871028894298991379984034340  \\
-16   & 1  & 000021021284289971399894287991343434  \\
-16   & 3  & 000210212871398894298991379984034340  \\
-17   & 0  & 002121028846289916399894287991343434  \\
-17   & 2  & 002121027984289971399894287991343434  \\
-17   & 4  & 021210287871388894298991379984034340  \\
-18   & 1  & 002121028884289971399894287991343434  \\
-18   & 3  & 021210288871398894298991379984034340  \\
-19   & 0  & 212121398884289971399894287991343434  \\
\\[-5pt]

\hline 
\multicolumn{3}{c}{$m(5_1)$}                         \\
\hline \\[-3pt]
3     & 0  & 000000025100062910299791379946034354  \\
2     & 1  & 000000002100028910299791379946034354  \\
1     & 0  & 000000002100028910299791379984034340  \\
1     & 2  & 000021251246629916397994039840003400  \\
0     & 1  & 000000002121028984299791379946034354  \\
0     & 3  & 000210255971625994662791398984034340  \\
-1    & 0  & 000000002121028984299791379984034340  \\
-1    & 2  & 000021251246629916397994289840343400  \\
-1    & 4  & 021210289791397994284391395546035554  \\
-2    & 1  & 000000021210289971397994289891343434  \\
-2    & 3  & 002121028984299791639946628406343554  \\
-3    & 0  & 000021021246289916397994289891343434  \\
-3    & 2  & 000210212791398984299791379984034340  \\
-3    & 4  & 021210289891397466289166397946034354  \\
-4    & 1  & 000021021284289971397994289891343434  \\
-4    & 3  & 021210279791398984299791379984034340  \\
-5    & 0  & 002121028846289916397994289891343434  \\
-5    & 2  & 002121027984289971397994289891343434  \\
-6    & 1  & 002121028884289971397994289891343434  \\
-7    & 0  & 021210287881397984288791398984034340  \\
\\[-5pt]

\hline 
\multicolumn{3}{c}{$5_2$}                            \\
\hline \\[-3pt]
-8    & 1  & 000210025971062994299971379794034340  \\
-9    & 0  & 000000251210629891397946039406003554  \\
-9    & 2  & 002100028910299971379794037940003400  \\
-10   & 1  & 000000002100258910629791398946034354  \\
-10   & 3  & 002100028910299971639794627940343400  \\
-11   & 0  & 000000002100258910629791398984034340  \\
-11   & 2  & 000000002100028910289791398946034354  \\
-12   & 1  & 000000002100028910289791398984034340  \\
-12   & 3  & 000210002891028984299771379984034340  \\
-13   & 0  & 000000002121258984629791398984034340  \\
-13   & 2  & 000000002121028984289791398946034354  \\
-13   & 4  & 021210287891379984299771379984034340  \\
-14   & 1  & 000000002121028984289791398984034340  \\
-14   & 3  & 000210002791028984289791398946034354  \\
-15   & 0  & 000000021210289891397984289891343434  \\
-15   & 2  & 000021251284629891397984289840343400  \\
-15   & 4  & 021210287791379984289791398946034354  \\
-16   & 1  & 000021021284289891397984289840343400  \\
-16   & 3  & 002121028984289791398846289406343554  \\
-17   & 0  & 000021021284289891397984289891343434  \\
-17   & 2  & 002100028910289791398884289891343434  \\
-18   & 1  & 002121027984289891397984289891343434  \\
-19   & 0  & 002121028884289891397984289891343434  \\
-19   & 2  & 002121028984289791398884289891343434  \\
-20   & 1  & 212121398884289891397984289891343434  \\
\\[-5pt]

\hline
\multicolumn{3}{c}{$m(5_2)$}                         \\
\hline\\[-3pt] \\[-3pt]
1     & 0  & 000000002100029710298971379794034340  \\
0     & 1  & 000000002121029746298916379794034340  \\
-1    & 0  & 000000002121029784298971379794034340  \\
-1    & 2  & 000000021210297971639894627991343434  \\
-2    & 1  & 000000021210297971379894287991343434  \\
-2    & 3  & 000210212891399946298916379846034354  \\
-3    & 0  & 000021021246297916379894287991343434  \\
-3    & 2  & 000021021284297971639894627991343434  \\
-4    & 1  & 000021021284297971379894287991343434  \\
-4    & 3  & 002121027984297971639894627991343434  \\
-5    & 0  & 000210212891399984288991379984034340  \\
-5    & 2  & 000210212871399894289991398984034340  \\
-6    & 1  & 002121028884297971379894287991343434  \\
-6    & 3  & 021210279891398984298991379984034340  \\
-7    & 0  & 002121028984297791379984289891343434  \\
-7    & 2  & 002121028984289791399784287991343434  \\
-8    & 1  & 212121398984289791399784287991343434  \\
\\[-5pt]

\hline
\multicolumn{3}{c}{$6_1$}                            \\
\hline\\[-3pt]
-5    & 0  & 021000297100379910299791379946034354  \\
-6    & 1  & 000210255971602994629791398946034354  \\
-7    & 0  & 000210002971258994629791398946034354  \\
-7    & 2  & 021000289121397946299916379794034340  \\
-8    & 1  & 000210002891028984299791379946034354  \\
-9    & 0  & 000210212891398984299791379946034354  \\
-9    & 2  & 000210002891028984299791379984034340  \\
-10   & 1  & 000210212891398984299791379984034340  \\
-10   & 3  & 021210287891379984299791379984034340  \\
-11   & 0  & 021000287121397984289791398984034340  \\
-11   & 2  & 021000289121377984289791398984034340  \\
-12   & 1  & 021210287791397984289791398984034340  \\
-13   & 0  & 021210287881397984289791398984034340  \\
-13   & 2  & 021210287891397984288791398984034340  \\
\\[-5pt]

\hline
\multicolumn{3}{c}{$m(6_1)$}                         \\
\hline\\[-3pt]
-3    & 0  & 000210025971295894662971379794034340  \\
-4    & 1  & 000210002971029846298916379794034340  \\
-5    & 0  & 000210002871029894298971379794034340  \\
-5    & 2  & 000210002891029984298991379946034354  \\
-6    & 1  & 000210212871399894298971379794034340  \\
-6    & 3  & 000210002891029984298991379984034340  \\
-7    & 0  & 021210279871399894298971379794034340  \\
-7    & 2  & 000210212971399894289991398984034340  \\
-7    & 4  & 000210212891399984298991379984034340  \\
-8    & 1  & 021210288871399894298971379794034340  \\
-8    & 3  & 021210279891399984298991379984034340  \\
-9    & 2  & 021210288971399894289991398984034340  \\
-9    & 4  & 021210288891399984298991379984034340  \\
\\[-5pt]

\hline
\multicolumn{3}{c}{$7_1$}                            \\
\hline\\[-3pt]
-14   & 1  & 000210002971029894298991379946034354  \\
-15   & 0  & 000210212971399894298991379946034354  \\
-15   & 2  & 000210002971029894298991379984034340  \\
-16   & 1  & 000210212971399894298991379984034340  \\
-17   & 0  & 021210279971399894298991379984034340  \\
-17   & 2  & 021210287971389894298991379984034340  \\
-18   & 1  & 021210288971399894298991379984034340  \\
\\[-5pt]

\hline
\multicolumn{3}{c}{$m(7_1)$}                         \\
\hline\\[-3pt]
5     & 0  & 021210299971397994299791379946034354  \\
4     & 1  & 021000289100397910299791379946034354  \\
3     & 0  & 021000289100397910299791379984034340  \\
3     & 2  & 021210289791397994299791379946034354  \\
2     & 1  & 021000289121397984299791379946034354  \\
1     & 0  & 021000289121397984299791379984034340  \\
1     & 2  & 021210289891397974299791379946034354  \\
0     & 1  & 021210289791397984299791379984034340  \\
-1    & 0  & 021210289881397984299791379984034340  \\
\\[-5pt]

\hline
\multicolumn{3}{c}{$7_2$}                            \\
\hline\\[-3pt]
-10   & 1  & 021210297971379994299791379946034354  \\
-11   & 0  & 251210629891397946299716379946034354  \\
-11   & 2  & 021210289971397994299971379794034340  \\
-12   & 1  & 021210289891397946299716379946034354  \\
-13   & 0  & 021000289100397910289791398946034354  \\
-13   & 2  & 021210287891397984299791379946034354  \\
-14   & 1  & 021000289100397910289791398984034340  \\
-15   & 0  & 021000289121397984289791398946034354  \\
-15   & 2  & 021210289791397994289791398984034340  \\
-16   & 1  & 021000289121397984289791398984034340  \\
-17   & 0  & 021210289791397984289791398984034340  \\
-18   & 1  & 021210289881397984289791398984034340  \\
\\[-5pt]

\hline
\multicolumn{3}{c}{$m(7_2)$}                         \\
\hline\\[-3pt]
1     & 0  & 000210002971029894298971379794034340  \\
0     & 1  & 000210212971399894298971379794034340  \\
-1    & 0  & 021210279971399894298971379794034340  \\
-2    & 1  & 021210288971399894298971379794034340  \\
-3    & 0  & 021210289891399984298991379984034340  \\
\\[-5pt]

\hline
\multicolumn{3}{c}{$8_1$}                            \\
\hline\\[-3pt]
-9    & 0  & 251210629891397984299791379946034354  \\
-10   & 1  & 021210289891397984299791379946034354  \\
-11   & 0  & 021210289891397984299791379984034340  \\
\\[-5pt]

\hline
\multicolumn{3}{c}{$m(8_1)$}                         \\
\hline\\[-3pt]
-3    & 0  & 021210297971379894298991379946034354  \\
-4    & 1  & 021210289971399894298971379794034340  \\
\end{supertabular}
\normalsize  
\end{document}